\documentclass{amsart}

\usepackage{amsthm}

\newtheorem*{theorem*}{Theorem}

\usepackage{macros}
\usepackage[qeds]{formatting}
\usepackage{parskip}
\usepackage{multicol}
\usepackage{graphicx}

\makeatletter
\newcommand\myitem[1][]{\item[#1]\def\@currentlabel{#1}}
\makeatother

\renewcommand{\Cat}{\mathsf{Cat}}
\newcommand{\cSet}{\mathsf{cSet}}
\renewcommand{\Gpd}{\mathsf{Gpd}}
\renewcommand{\Set}{\mathsf{Set}}
\newcommand{\sSet}{\mathsf{sSet}}

\newcommand{\ihom}{\underline{\operatorname{hom}}}
\newcommand{\Ho}{\mathsf{Ho}}
\newcommand{\Hosharp}{\Ho^\sharp}

\newcommand{\Map}{\mathrm{Map}}
\newcommand{\Hom}{\mathrm{Hom}}
\newcommand{\rmco}{\mathrm{co}}
\newcommand{\coop}{{\operatorname{co-op}}}

\newcommand{\bd}{\partial}

\newcommand{\anod}{\mathsf{anod}}
\newcommand{\naive}{\mathsf{n.f.}}

\newcommand{\tBox}{\Box_\sharp} 

\renewcommand{\Im}{\mathrm{Im}}

\newcommand{\adjoint}{\dashv}
\newcommand{\Nerve}{\mathsf{N}_\Box}
\newcommand{\core}{\mathsf{core}}

\newcommand{\thetabar}{\overline{\theta}}

\newcommand{\pushoutcorner}[1][dr]{\save*!/#1+1.2pc/#1:(1,-1)@^{|-}\restore}

\newcommand{\pshf}[1]{\mathsf{Set}^{ #1 ^{\mathrm{op}}}}

\begin{document}
\title{Cubical models of $(\infty, 1)$-categories}
\author[B.~Doherty, K.~Kapulkin, Z.~Lindsey, C.~Sattler]{Brandon Doherty \and Krzysztof Kapulkin \and Zachery Lindsey \and Christian Sattler}
\maketitle

\begin{abstract}
  We construct a model structure on the category of cubical sets with connections whose cofibrations are the monomorphisms and whose fibrant objects are defined by the right lifting property with respect to inner open boxes, the cubical analogue of inner horns.
  We show that this model structure is Quillen equivalent to the Joyal model structure on simplicial sets via the triangulation functor.
  As an application, we show that cubical quasicategories admit a convenient notion of a mapping space, which we use to characterize the weak equivalences between fibrant objects in our model structure as DK-equivalences.
\end{abstract}

\section*{Introduction}
\thispagestyle{empty}
The category $\sSet$ of simplicial sets carries two canonical model structures: the Kan-Quillen model structure \cite{quillen:book}, presenting the homotopy theory of $\infty$-groupoids, and the Joyal model structure \cite{joyal:theory-of-quasi-cats}, presenting the homotopy theory of $(\infty,1)$-categories.
Both of these model structures have monomorphisms as their cofibrations and their fibrant objects are defined by a more or less restrictive lifting condition, depending on whether the 1-simplices of a fibrant object are  invertible.

The category $\cSet$ of cubical sets is also known to carry a model structure, called the Grothendieck model structure and constructed by Cisinski \cite{CisinskiAsterisque,CisinskiUniverses}, presenting the theory of $\infty$-groupoids.
This model structure is completely analogous to the Kan-Quillen model structure, but with open boxes replacing horns in the definition of fibrant objects.
The goal of the present work is provide a cubical analogue of the Joyal model structure, thus filling the bottom right corner in the table:

\begin{center}
  \begin{tikzpicture}
  \matrix (magic) [matrix of nodes,nodes={minimum width=3cm,minimum height=1cm,draw,very thin},draw,inner sep=0]
  {
  category $\setminus$ theory & $\infty$-groupoids & $(\infty,1)$-categories \\
  $\sSet$ & \cite{quillen:book} & \cite{joyal:theory-of-quasi-cats} \\
  $\cSet$ & \cite{CisinskiUniverses} & present work \\
  };
  \end{tikzpicture}
\end{center}

Our main theorem (cf.~\cref{cubical-joyal-ms,fibrant-objects-unmarked,T-Quillen-equivalence,fundamental-theorem}) states
\enlargethispage{11pt}
\begin{theorem*}
  The category $\cSet$ of cubical sets carries a model structure in which:
  \begin{itemize}
    \item the cofibrations are the monomorphisms;
    \item the fibrant objects are defined by having fillers for all inner open boxes;
    \item the weak equivalences between fibrant objects are the DK-equivalences.
  \end{itemize}
  \enlargethispage{11pt}
  Moreover, this model structure is Quillen equivalent to the Joyal model structure on the category $\sSet$ of simplicial sets via the triangulation functor $T \colon \cSet \to \sSet$.
\end{theorem*}

A few comments are in order.

First, there are many different notions of a cubical set, depending on the choice of maps in the indexing category $\Box$, called the \emph{box category}.
Here, we are taking as agnostic a view as possible and most of our theorems apply to a number of different such notions.

Specifically, in addition to the usual box category with face and degeneracy maps, as studied in \cite{CisinskiAsterisque,jardine:categorical-homotopy-theory}, we are also considering the box category with one or both connections.
The case of one connection, specifically the negative (a.k.a.~maximum) connection, was studied in \cite{CisinskiUniverses,maltsiniotis:connections-strict-test-cat,kapulkin-lindsey-wong}, and the case with both connections has been, since the present paper first appeared, investigated in \cite{campion-kapulkin-maehara}.
All of these are EZ-Reedy categories and test categories, and in the case when at least one connection is present in the box category, they are in fact \emph{strict} test categories, making them convenient to work with.
All of the box categories under consideration are displayed in the diagram at the end of \cref{sec:background}.

All of our results are true for the box category with both connections and the paper can be read with this category in mind.
However, the construction of the model structures and characterization of weak equivalences using cubical mapping spaces work for all four box categories mentioned above.
The Quillen equivalence with the Joyal model structure on simplicial sets requires at least one connection.
Finally, the comparison between cubical and simplicial mapping spaces requires both connections, although weaker comparison results hold for the other three box categories.

One particularly interesting aspect of our proof of the comparison between cubical and simplicial models of $(\infty,1)$-categories is that we do not work with the triangulation functor directly.
Instead, we rely on a result of \cite{kapulkin-lindsey-wong} exhibiting the category of simplicial sets as a co-reflective subcategory of the category of cubical sets with the negative connection via the straightening-over-the-point functor $Q \colon \sSet \to \cSet$, an instance of a more general construction of straightening, studied in \cite{kapulkin-voevodsky:cubical-straightening}.
While at first glance the straightening-over-the-point functor seems more involved, it ends up being significantly more convenient to work with than the triangulation functor.
And so, in order to show that $T$ is a Quillen equivalence, we first prove that $Q$ is a Quillen equivalence and establish that the derived functors of $T$ and $Q$ are each other's inverses.

Lastly, the concept of an \emph{inner open box} appearing in the statement of our main theorem is the cubical analogue of the notion of an inner horn in simplicial sets.
Its definition is somewhat subtle, which is the reason behind our taking a slight detour in the construction of the model structure on cubical sets.
At this point however, we shall simply note this subtlety here and give a precise definition in \cref{sec:joyal-cset}.

In order to establish a model structure on $\cSet$, we consider first a model structure on marked cubical sets.
A \emph{marked cubical set} is a cubical set with a distinguished subset of edges (to be thought of as ``equivalences''), containing all degenerate ones.
We then use the minimal marking functor, taking a cubical set to a marked cubical set in which the marked edges are precisely the degeneracies, to left-induce (in the sense of \cite{hkrs:transfer-thm}) a model structure on cubical sets.

In order to establish that the triangulation functor is a Quillen equivalence between our model structure on cubical sets and the Joyal model structure on simplicial sets, we introduce a cubical theory of cones, which generalizes the straightening-over-the-point construction.
Our cubical cones serve as a convenient way of relating simplicial and cubical shapes, and we believe that these tools will find applications beyond present work.

This paper is organized as follows.
In \cref{sec:background}, we collect the necessary results on model categories, cubical sets, and marked cubical sets.
Trying to keep the exposition as self-contained as possible, we include statements of frequently used results and those that may be harder to find in the existing literature.

In \cref{section:marked}, we construct the model structure on the category of marked cubical sets, using Jeff Smith's theorem \cite{beke:sheafifiable}.
While this model structure could be constructed using Olschok's theory \cite{olschok:thesis}, we decided to spell out the details of the construction to obtain a better understanding of the resulting classes of maps.
Although described here only as means to an end, we believe that this model structure is of independent interest and provides another convenient model for the theory of $(\infty,1)$-categories.

Then, in \cref{sec:structurally-marked} we show that this model structure is right-induced by a model structure on the category of structurally marked cubical sets, a presheaf category containing marked cubical sets as a reflective subcategory, constructed using the Cisinski theory.

In \cref{sec:joyal-cset}, we use the minimal marking functor to construct the desired model structure on the category of cubical sets.
We then analyze the resulting classes of maps, characterizing weak equivalences and fibrations between fibrant objects.

In \cref{sec:cones}, we develop the theory of cones and coherent families of composites, which is then used in \cref{sec:relation} to show that the model structure of \cref{sec:joyal-cset} is Quillen equivalent with the Joyal model structure.
This last argument is fairly combinatorial and includes a number of routine computations involving cubical identities.
For clarity of exposition, most of these computations are therefore relegated to Appendix \ref{appendix:calculations}.
Once again, we expect the tools used here to be of independent interest, e.g., in developing category theory in the context of our cubical quasicategories.

We conclude in \cref{sec:fund-thm}, by defining mapping spaces in cubical quasicategories, showing that they are Kan complexes, and characterizing weak equivalences between fibrant objects in our model structure on $\cSet$ as DK-equivalences, i.e., maps inducing equivalences on homotopy categories and equivalences of mapping spaces.
We further establish a relation between these cubical mapping spaces and their simplicial analogs via the triangulation and straightening adjunctions, and give a new proof of the characterization of categorical equivalences in the Joyal model structure as DK-equivalences.

Regarding the dependence on choice of the box category:
\begin{itemize}
  \item the results of \cref{section:marked,sec:structurally-marked,sec:joyal-cset} and the main results of \cref{sec:fund-thm}, in particular \cref{fundamental-theorem} work for all categories displayed in the diagram \cref{eq:cube-cat-inclusions} at the end of \cref{sec:background};
  \item the results of \cref{sec:cones,sec:relation} require at least one connection;
  \item the results of \cref{sec:fund-thm} relating cubical and simplicial mapping spaces require both connections.
\end{itemize}

\textbf{Acknowledgements.} During the work on this paper the first three authors were supported by the Natural Sciences and Engineering Research Council (NSERC) of Canada in a variety of forms, including the Canada Graduate Scholarship awarded to the first-named author and the Discovery Grant awarded to the second-named author.
The fourth author was supported by Swedish Research Council grant 2019-03765.
This material is based upon work supported by the National Science Foundation under Grant No.~1440140, while the second-named author was in residence at the Mathematical Sciences Research Institute in Berkeley, California, during Spring 2020.

\setcounter{tocdepth}{1}
\tableofcontents

\section{Cubical sets and marked cubical sets} \label{sec:background}

\subsection{Model categories}\label{Sec1Cisinski}
Here we will review various general results from the theory of model categories which we will use throughout subsequent sections. We begin with a result which allows us to construct model structures having specified classes of cofibrations and weak equivalences.

\begin{theorem}[Jeff Smith's Theorem, {\cite[Thm.~1.7, Prop.~1.15, 1.19]{beke:sheafifiable}}]\label{Jeff-Smith's-theorem}
Let $\mathsf{C}$ be a locally presentable category. Let $W$ be a class of morphisms forming an accessibly embedded, accessible subcategory of $\mathsf{C}^{\to}$, and $I$ a set of morphisms in $\mathsf{C}$. Suppose that the following conditions are satisfied.

\begin{itemize}
\item $W$ satisfies the two-out-of-three axiom.
\item $W$ contains all maps having the right lifting property with respect to the maps in $I$.
\item The intersection of $W$ with the saturation of $I$ is closed under pushouts and transfinite composition.
\end{itemize}

Then $\mathsf{C}$ admits a cofibrantly generated model structure with weak equivalences $W$ and generating cofibrations $I$.
\end{theorem}

Next we review some of the machinery of Cisinski theory \cite{CisinskiAsterisque}, which allows for the easy construction of model structures on presheaf categories having monomorphisms as cofibrations and weak equivalences defined in terms of homotopy with respect to a cylinder functor.

\begin{definition}\label{CylinderFunctor}
Let $\mathsf{C}$ be a small category. A \emph{cylinder functor} on $\mathsf{C}$ consists of an endofunctor $I$ on the presheaf category $\pshf{\mathsf{C}}$, together with natural transformations $\bd^{0}, \bd^{1} \colon \mathrm{id} \to I$, $\sigma \colon I \to \mathrm{id}$, such that:

\begin{itemize}
\item $\bd^{0}$ and $\bd^{1}$ are sections of $\sigma$;
\item For all $X \colon \mathsf{C}^{\mathrm{op}} \to \mathsf{Set}$, the map $(\bd^{0}_{X}, \bd^{1}_{X}) \colon X \sqcup X \to IX$ is a monomorphism;
\item $I$ preserves small colimits and monomorphisms;
\item For all monomorphisms $j \colon X \to Y$ in $\pshf{\mathsf{C}}$ and all $\varepsilon \in \{0,1\}$, the following square is a pullback:

\centerline{
\xymatrix{
X \ar[d]_{\bd^{\varepsilon}} \ar[r]^{j} & Y \ar[d]^{\bd^{\varepsilon}} \\
IX \ar[r]^{Ij} & IY \\
}
}
\end{itemize}
\end{definition}

In what follows, let $\mathsf{C}$ be a small category equipped with a cylinder functor $I \colon \pshf{\mathsf{C}} \to \pshf{\mathsf{C}}$.

\begin{definition}
Let $f, g \colon X \to Y$ be maps of presheaves on $\mathsf{C}$. An \emph{elementary homotopy} from $f$ to $g$ is a map  $H \colon IX \to Y$ such that $H \bd^{0} = f, H \bd^{1} = g$. A \emph{homotopy} is a zig-zag of elementary  homotopies. The set of maps from $X$ to $Y$ modulo the relation of homotopy is denoted $[X,Y]$. 
\end{definition}

Since our notion of homotopy is given by a cylinder functor, pre- and post-composition by a fixed map preserve this relation; thus a map $X \to Y$ induces maps $[Z,X] \to [Z,Y]$ and $[Y,Z] \to [X,Z]$ for any $Z$.

\begin{definition}
A \emph{cellular model} for $\pshf{\mathsf{C}}$ is a set $M$ of monomorphisms in $\pshf{\mathsf{C}}$ whose saturation is precisely the class of monomorphisms of $\pshf{\mathsf{C}}$.
\end{definition}

Let $M$ be a cellular model for $\pshf{\mathsf{C}}$, and $S$ a set of monomorphisms in $\pshf{\mathsf{C}}$. The set of morphisms $\Lambda(S)$ is defined by the following inductive construction. For a monomorphism $X \to Y$ in $\pshf{\mathsf{C}}$ and $\varepsilon \in \{0,1\}$, let $IX \cup_{\varepsilon} Y$ and $IX \cup (Y \sqcup Y)$ be defined by the following pushout squares:

\centerline{
\xymatrix{
X \ar[r] \ar[d]_{\bd^{\varepsilon}} & Y \ar[d] \ar[d]  & & X \sqcup X \ar[d] \ar[r] & Y \sqcup Y \ar[d] \\
IX \ar[r] & IX \cup_{\varepsilon} Y \pushoutcorner   & & IX \ar[r] & IX \cup (Y \sqcup Y) \pushoutcorner 
}
} 
We now define a set of monomorphisms $\Lambda(S)$ by an inductive construction. We begin by setting:

$$
\Lambda^{0}(S) = S \cup \{IX \cup_{\varepsilon} Y \to IY | X \to Y \in M, \varepsilon \in \{0,1\}\} 
$$

Now, given $\Lambda^{n}(S)$, we define:

$$
\Lambda^{n+1}(S) = \{IX \cup (Y \sqcup Y) \to IY | X \to Y \in \Lambda^{n}(S)\}
$$

Finally, we let $\Lambda(S) = \bigcup\limits_{n \geq 0} \Lambda^{n}(S)$. We now define several distinguished classes of maps and objects in $\pshf{\mathsf{C}}$.

\begin{itemize}
\item A \emph{cofibration} is a monomorphism; a \emph{trivial fibration} is a map having the right lifting property with respect to the cofibrations.
\item An \emph{anodyne map} is a map in the saturation of $\Lambda(S)$; a \emph{naive fibration} is a map having the right lifting property with respect to the anodyne maps.
\item A \emph{fibrant object} is a presheaf $X$ such that the map from $X$ to the terminal presheaf is a naive fibration.
\item A \emph{weak equivalence} is a map $X \to Y$ such that the induced map $[Y,Z] \to [X,Z]$ is a bijection for any fibrant $Z$.
\item A \emph{trivial cofibration} is a map which is both a cofibration and a weak equivalence; a \emph{fibration} is a map having the right lifting property with respect to the trivial cofibrations.
\end{itemize}

\begin{theorem}\label{CisinskiMS}
The classes above define a cofibrantly generated model structure on $\pshf{\mathsf{C}}$, in which a map between fibrant objects is a fibration if and only if it is a naive fibration.
\end{theorem}

\begin{proof}
The existence of the model structure is \cite[Thm. 1.3.22]{CisinskiAsterisque}; the characterization of fibrant objects is \cite[Thm. 1.3.36]{CisinskiAsterisque}.
\end{proof}

\begin{corollary}\label{CisinskiHo}
  The homotopy category of $\pshf{\mathsf{C}}$ with the model structure of \cref{CisinskiMS} can be described as follows:
  \begin{itemize}
    \item its objects are the fibrant presheaves;
    \item the maps from $X$ to $Y$ are given by $[X, Y]$. \qed
  \end{itemize}
\end{corollary}

\begin{example}\label{sSet-J}
Let $J$ denote the simplicial set depicted below:

\[
\xymatrix{
  1 \ar[r] \ar@{=}[dr] & 0 \ar[d] \ar@{=}[dr] &  \\
    & 1 \ar[r] & 0 }
\]

Taking the product with $J$ defines a cylinder functor on $\sSet$, with the natural transformations $\bd^0, \bd^1$ given by taking the product with the endpoint inclusions $\{0\} \hookrightarrow J, \{1\} \hookrightarrow J$. Applying \cref{CisinskiMS} with this cylinder functor, the cellular model $M = \{\bd \Delta^n \to \Delta^n | n \geq 0\}$, and $S = \{\Lambda^{n}_{i} | n \geq 2, 1 < i < n\}$ (the set of \emph{inner horn inclusions}), we obtain the \emph{Joyal model structure} on $\sSet$, characterized as follows:

\begin{itemize}
\item Cofibrations are monomorphisms;
\item Fibrant objects are \emph{quasicategories}, simplicial sets having fillers for all inner horns;
\item Fibrations between fibrant objects are characterized by the right lifting property with respect to the inner horn inclusions and the endpoint inclusions $\{\varepsilon\} \hookrightarrow J, \varepsilon \in \{0,1\}$;
\item Weak equivalences are \emph{weak categorical equivalences}, maps $X \to Y$ inducing bijections $[Y,Z] \to [X,Z]$ for all quasicategories $Z$.
\end{itemize}

Homotopy equivalences between quasicategories are referred to as \emph{categorical equivalences}. For more on the Joyal model structure, see \cite{joyal:theory-of-quasi-cats}; for the details of its construction as a Cisinski model structure, see \cite[Sec. 3.3]{cisinski:higher-categories-book}.
\end{example}

Next we review a theorem which allows us to induce one model structure from another using an adjunction between their respective categories.

\begin{definition}
Let $F : \mathsf{C} \rightleftharpoons \mathsf{D} : U$ be an adjunction between model categories. The model structure on $\mathsf{C}$ is \emph{left induced} by $F$ if $F$ preserves and reflects cofibrations and weak equivalences. Likewise, the model structure on $\mathsf{D}$ is \emph{right induced} by $U$ if $U$ preserves and reflects weak equivalences and fibrations.
\end{definition}

\begin{remark}
Note that for a given adjunction $\mathsf{C} \rightleftharpoons \mathsf{D}$ and a given model structure on $\mathsf{D}$, the left-induced model structure is unique, if one exists, since the definition determines the cofibrations and weak equivalences of $\mathsf{C}$. Likewise, for a given model structure on $\mathsf{C}$, the right-induced model structure is unique, if one exists.
\end{remark}

\begin{theorem}[{\cite[Thm.~2.2.1]{hkrs:transfer-thm}}]\label{transfer-theorem}
Let $F : \mathsf{C} \rightleftarrows \mathsf{D} : U$ be an adjunction between locally presentable categories such that $\mathsf{D}$ carries a cofibrantly generated model structure with all objects cofibrant. If, for every object $X \in \mathsf{C}$, the co-diagonal map admits a factorization $X \sqcup X \xrightarrow{i_{X}} IX \xrightarrow{p_{X}} X$, such that $Fi_{X}$ is a cofibration and $Fp_{X}$ is a weak equivalence, then $\mathsf{C}$ admits a model structure left-induced by $F$ from that of $\mathsf{D}$. \qed
\end{theorem}

Finally, we review some results which allow us to easily recognize Quillen adjunctions and Quillen equivalences.

\begin{proposition}[{\cite[Prop.~7.15]{joyal-tierney:qcat-vs-segal}}]
Let $F : \mathsf{C} \rightleftarrows \mathsf{D} : U$ be an adjunction between model categories. If $F$ preserves cofibrations and $U$ preserves fibrations between fibrant objects, then the adjunction is Quillen. \qed
\end{proposition}

This statement has an immediate corollary, which we will apply in practice:

\begin{corollary-qed} \label{Quillen-adj-fib-obs}
Let $F \colon \mathsf{C} \to \mathsf{D}$ be a left adjoint between model categories and suppose that fibrations between fibrant objects in $\mathsf{C}$ are characterized by right lifting against a class $S$.
If $F$ preserves cofibrations and sends $S$ to trivial cofibrations, then $F$ is a left Quillen functor.
\end{corollary-qed}

\begin{proposition}[{\cite[Cor.~1.3.16]{hovey:book}}]\label{QuillenEquivCreate-original}
Let $F : \mathsf{C} \rightleftarrows \mathsf{D} : U $ be a Quillen adjunction between model categories.
Then the following are equivalent.
\begin{enumerate}
\item $F \adjoint U$ is a Quillen equivalence.
\item $F$ reflects weak equivalences between cofibrant objects and, for every fibrant $Y$, the derived counit $F\widetilde{UY} \to  Y$ is a weak equivalence.
\item $U$ reflects weak equivalences between fibrant objects and, for every cofibrant $X$, the derived unit $X \to U (FX)'$ is a weak equivalence.
\end{enumerate}
\end{proposition}

Again, in practice we will often apply the following corollary:

\begin{corollary}\label{QuillenEquivCreate}
Let $F : \mathsf{C} \rightleftarrows \mathsf{D} : U $ be a Quillen adjunction between model categories.

\begin{enumerate}
\item\label{QuillenEquivUnit} If $U$ preserves and reflects weak equivalences, then the adjunction is a Quillen equivalence if and only if, for all cofibrant $X \in \mathsf{C}$, the unit $X \to UFX$ is a weak equivalence.
\item\label{QuillenEquivCounit} If $F$ preserves and reflects weak equivalences, then the adjunction is a Quillen equivalence if and only if, for all fibrant $Y \in \mathsf{D}$, the counit $FUY \to Y$ is a weak equivalence. \qed
\end{enumerate}
\end{corollary}

We will also have some use for the following consequence of this result, which concerns involutions of model categories. Recall that any involution of a category is self-adjoint, with the identity natural transformation as both unit and counit.

\begin{corollary}\label{QuillenEquivInvolution}
Let $\catC$ be a model category, and $F \colon \catC \to \catC$ an involution. If the adjunction $F \adjoint F$ is Quillen, then it is a Quillen equivalence.
\end{corollary}

\begin{proof}
If $F \adjoint F$ is Quillen, then $F$ preserves trivial cofibrations and trivial fibrations, hence all weak equivalences. The fact that $F$ is an involution thus implies that it reflects weak equivalences as well. Both the unit and counit of the adjunction are the identity natural transformation on $\catC$, thus we may apply either statement of \cref{QuillenEquivCreate} to conclude that the adjunction is a Quillen equivalence.
\end{proof}

\subsection{The box category and cubical sets}\label{subsection:cSet-basics}
We begin by defining the box category $\Box$.
The objects of $\Box$ are posets of the form $[1]^n$ and the maps are generated (inside the category of posets) under composition by the following four special classes:
\begin{itemize}
  \item \emph{faces} $\partial^n_{i,\varepsilon} \colon [1]^{n-1} \to [1]^n$ for $i = 1, \ldots , n$ and $\varepsilon = 0, 1$ given by:
  \[ \partial^n_{i,\varepsilon} (x_1, x_2, \ldots, x_{n-1}) = (x_1, x_2, \ldots, x_{i-1}, \varepsilon, x_i, \ldots, x_{n-1})\text{;}  \]
  \item \emph{degeneracies} $\sigma^n_i \colon [1]^n \to [1]^{n-1}$ for $i = 1, 2, \ldots, n$ given by:
  \[ \sigma^n_i ( x_1, x_2, \ldots, x_n) = (x_1, x_2, \ldots, x_{i-1}, x_{i+1}, \ldots, x_n)\text{;}  \]
  \item \emph{negative connections} $\gamma^n_{i,0} \colon [1]^n \to [1]^{n-1}$ for $i = 1, 2, \ldots, n-1$ given by:
  \[ \gamma^n_{i,0} (x_1, x_2, \ldots, x_n) = (x_1, x_2, \ldots, x_{i-1}, \max\{ x_i , x_{i+1}\}, x_{i+2}, \ldots, x_n) \text{.} \]
  \item \emph{positive connections} $\gamma^n_{i,1} \colon [1]^n \to [1]^{n-1}$ for $i = 1, 2, \ldots, n-1$ given by:
  \[ \gamma^n_{i,1} (x_1, x_2, \ldots, x_n) = (x_1, x_2, \ldots, x_{i-1}, \min\{ x_i , x_{i+1}\}, x_{i+2}, \ldots, x_n) \text{.} \]
\end{itemize}

These maps obey the following \emph{cubical identities}:

\begin{multicols}{2}
$\partial_{j, \varepsilon'} \partial_{i, \varepsilon} = \partial_{i+1, \varepsilon} \partial_{j, \varepsilon'}$ for $j \leq i$;

$\sigma_i \sigma_j = \sigma_j \sigma_{i+1} \quad \text{for } j \leq i$;

$\sigma_j \partial_{i, \varepsilon} = \left\{ \begin{array}{ll}
\partial_{i-1, \varepsilon} \sigma_j   & \text{for } j < i \text{;} \\
\id                                                       & \text{for } j = i \text{;} \\
\partial_{i, \varepsilon} \sigma_{j-1} & \text{for } j > i \text{;}
\end{array}\right.$

$\gamma_{j,\varepsilon'} \gamma_{i,\varepsilon} = \left\{ \begin{array}{ll} \gamma_{i,\varepsilon} \gamma_{j+1,\varepsilon'} & \text{for } j > i \text{;} \\
\gamma_{i,\varepsilon}\gamma_{i+1,\varepsilon} & \text{for } j = i, \varepsilon' = \varepsilon \text{;}\\
\end{array}\right.$

$\gamma_{j,\varepsilon'} \partial_{i, \varepsilon} =  \left\{ \begin{array}{ll}
\partial_{i-1, \varepsilon} \gamma_{j,\varepsilon'}   & \text{for } j < i-1 \text{;} \\
\id                                                         & \text{for } j = i-1, \, i, \, \varepsilon = \varepsilon' \text{;} \\
\partial_{i, \varepsilon} \sigma_i         & \text{for } j = i-1, \, i, \, \varepsilon = 1-\varepsilon' \text{;} \\
\partial_{i, \varepsilon} \gamma_{j-1,\varepsilon'} & \text{for } j > i \text{;} 
\end{array}\right.$

$\sigma_j \gamma_{i,\varepsilon} =  \left\{ \begin{array}{ll}
\gamma_{i-1,\varepsilon} \sigma_j  & \text{for } j < i \text{;} \\
\sigma_i \sigma_i           & \text{for } j = i \text{;} \\
\gamma_{i,\varepsilon} \sigma_{j+1} & \text{for } j > i \text{.} 
\end{array}\right.$
\end{multicols}

\begin{theorem}[{\cite[Thm.~5.1]{grandis-mauri}}] \label{normal-form}
  Every map in the category $\Box$ can be factored uniquely as a composite
  \[ (\partial_{c_1, \varepsilon'_1} \ldots \partial_{c_r, \varepsilon'_r})
     (\gamma_{b_1,\varepsilon_1} \ldots \gamma_{b_q,\varepsilon_q})
     (\sigma_{a_1} \ldots \sigma_{a_p})\text{,} \]
  where $1 \leq a_1 < \ldots < a_p$, $1 \leq b_1 \leq \ldots \leq b_q$, $b_i < b_{i+1}$ if $\varepsilon_{i} = \varepsilon_{i+1}$, and $c_1 > \ldots > c_r \geq 1$.   \qed
\end{theorem}

\begin{corollary}\label{Box-Reedy}
$\Box$ admits the structure of an EZ-Reedy category, in which:
\begin{itemize}
\item $\mathrm{deg}([1]^{n}) = n$;
\item $\Box_{+}$ is generated under composition by the face maps;
\item $\Box_{-}$ is generated under composition by the degeneracy and connection maps.\qed
\end{itemize}
\end{corollary}

The category of cubical sets, i.e., contravariant functors $\Box^\op \to \Set$ will be denoted by $\cSet$. We will write $\Box^n$ for the representable cubical set, represented by $[1]^n$. We adopt the convention of writing the action of cubical operators on the right. For instance, the $(1, 0)$-face of an $n$-cube $x \colon \Box^n \to X$ will be denoted $x \partial_{1, 0}$. By a \emph{degenerate} cube of a cubical set $X$, we will understand one that is in the image of either a degeneracy or a connection.
A \emph{non-degenerate} cube is one that is not degenerate.

We will occasionally represent cubical sets using pictures.
In that, $0$-cubes are represented as vertices, $1$-cubes as arrows, $2$-cubes as squares, and $3$-cubes as cubes.

For a $1$-cube $f$, we draw
\[
\xymatrix{ x \ar[r]^f & y} \]
to indicate $x = f \partial_{1,0}$ and $y = f \partial_{1,1}$.
For a $2$-cube $s$, we draw
\[
\xymatrix{
 x
  \ar[r]^h
  \ar[d]_f
&
   y
  \ar[d]^g
\\
  z
  \ar[r]^k
&
 w
}
\]
to indicate $s\partial_{1,0} = f$,  $s\partial_{1,1} = g$, $s\partial_{2,0} = h$,  and $s\partial_{2,1} = k$.
As for the convention when drawing $3$-dimensional boxes, we use the following ordering of axes:
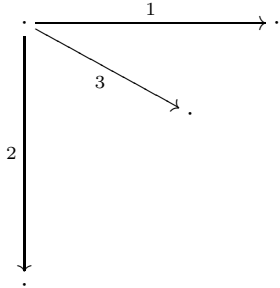
\[
\xymatrix{
 \cdot
  \ar[rrr]^1
  \ar[rrd]_3
  \ar[ddd]_2
& & &
   \cdot
\\
  & & \cdot
\\ & & &
\\
 \cdot & & &
}
\]
For readability, we do not label $2$- and $3$-cubes.
Similarly, if a specific $0$-cube is irrelevant for the argument or can be inferred from the context, we represent it by $\bullet$, and we omit labels on edges whenever the label is not relevant for the argument.

Lastly, a degenerate $1$-cube $x \sigma_1$ on $x$ is represented by
\[
\xymatrix{ x \ar@{=}[r] & x\text{,}} \]
while a $2$- or $3$-cube whose boundary agrees with that of a degenerate cube is assumed to be degenerate unless indicated otherwise.
For instance, a $2$-cube depicted as
\[
\xymatrix{
 x
  \ar@{=}[r]
  \ar[d]_f
&
   x
  \ar[d]^f
\\
  y
  \ar@{=}[r]
&
 y
}
\]
represents $f \sigma_1$.

We write $\partial \Box^n \to \Box^n$ for the maximal proper subobject of $\Box^n$, i.e., the union of all of its faces.
We will refer to these as the \emph{$n$-cube} and the \emph{boundary} of the $n$-cube, respectively.
The subobject of $\Box^n$ given by the union of all faces except $\bd_{i,\varepsilon}$ will be denoted $\sqcap^n_{i,\varepsilon}$ and referred to as an $(i, \varepsilon)$-open box.

In many cases, we will construct cubes in a cubical set $X$ by \emph{filling} open boxes, i.e. extending a map $\sqcap^n_{i,\varepsilon} \to X$ to $\Box^n$. When illustrating the filling of a 2-dimensional open box, the new edge obtained from the filling will be indicated with a dotted line. For instance, the diagram below illustrates the filling of a $(1,0)$-open box.

\centerline{
\xymatrix{
\bullet \ar@{..>}[d] \ar[r] & \bullet \ar[d] \\
\bullet \ar[r] & \bullet \\
}
}

From \cref{normal-form}, we obtain the following:

\begin{proposition}\label{StandForm}
Given a cubical set $X$, for any cube $x \colon \Box^{n} \to X$ there exist unique (possibly empty) sequences $a_{1} < ... < a_{p}, b_{1} \leq ... \leq b_{q}, \varepsilon_{1}, ..., \varepsilon_{q} \in \{0,1\}$, where $b_{i} < b_{i+1}$ if $\varepsilon_{i} = \varepsilon_{i+1}$, and a unique non-degenerate cube $y \colon \Box^{n-p-q} \to X$ such that $x = y\gamma_{b_{1},\varepsilon_{1}}...\gamma_{b_{q},\varepsilon_{q}}\sigma_{a_{1}}...\sigma_{a_{p}}$.  \qed
\end{proposition}

This factorization is called the \emph{standard form} of $x$.

\begin{corollary}\label{map-on-non-degen}
A map $X \to Y$ in $\cSet$ is determined by its action on the non-degenerate cubes of $X$. \qed
\end{corollary}

\begin{corollary}\label{mono-characterization}
A map $X \to Y$ in $\cSet$ is a monomorphism if and only if it maps non-degenerate cubes of $X$ to non-degenerate cubes of $Y$, and does so injectively. \qed
\end{corollary}

For brevity, we will often say that the standard form of a cube $x$ is $zf$, or ``ends with $f$", where $f$ is some map in $\Box$; this is understood to mean that $f$ is the rightmost map in the standard form of $x$. For instance, if the standard form of $x$ is $z\sigma_{a_{p}}$, then $z = y\gamma_{b_{1},\varepsilon_{1}}...\gamma_{b_{q},\varepsilon_{q}}\sigma_{a_{1}}...\sigma_{a_{p-1}}$ in the notation of \cref{StandForm}.

\begin{definition}
The \emph{critical edge} of $\Box^{n}$ with respect to a face $\partial_{i,\varepsilon}$ is the unique edge of $\Box^{n}$ which is adjacent to $\partial_{i,\varepsilon}$ and which, together with $\partial_{i,\varepsilon}$, contains both of the vertices $(0,...,0)$ and $(1,...,1)$.
\end{definition}

More explicitly, the critical edge with respect to $\partial_{i,\varepsilon}$ corresponds to the map $f \colon [1] \to [1]^{n}$ given by $f_{i} = \mathrm{id}_{[1]}, f_{j} = \mathrm{const}_{1-\varepsilon}$ for $j \neq i$.

The assignment $([1]^m, [1]^n) \mapsto [1]^{m+n}$ defines a functor $\Box \times \Box \to \Box$. Postcomposing it with the Yoneda embedding and left Kan extending, we obtain the \emph{geometric product} functor
\[
\xymatrix@C+0.5cm{
  \Box \times \Box
  \ar[r]
  \ar@{^{(}->}[d]
&
  \cSet
\\
  \cSet \times \cSet
  \ar[ru]_{\otimes}
&
}
\]

The pointwise formula for left Kan extensions gives us the following formula for the geometric product:

$$
X \otimes Y = \colim\limits_{\substack{x \colon \Box^m \to X \\ y \colon \Box^n \to Y}} \Box^{m+n}
$$

Note that the geometric product of cubical sets does not coincide with the cartesian product. However, the geometric product implements the correct homotopy type, and is better behaved than the cartesian product -- for instance, for $m, n \geq 0$ we have $\Box^m \otimes \Box^n = \Box^{m+n}$. Furthermore, the geometric product is taken to the cartesian product by the geometric realization functor to spaces.

\begin{proposition}
The geometric product $\otimes$ defines a monoidal structure on the category of cubical sets, with the unit given by $\Box^0$. \qed
\end{proposition}

This monoidal structure is however not symmetric. 
Indeed, the existence of a symmetry natural transformation would in particular imply that there is a non-identity bijection $[1]^2 \to [1]^2$ in $\Box$.

In particular, for any $X, Y \in \cSet$, the unique maps from $X$ and $Y$ to $\Box^0$ induce maps $\pi_{X} \colon X \otimes Y \to X, \pi_{Y} \colon X \otimes Y \to Y$.

Given a cubical set $X$, we form two non-isomorphic functors $\cSet \to \cSet$: the left tensor $- \otimes X$ and the right tensor $X \otimes -$.
As they are both co-continuous, they admit right adjoints and we write $\ihom_L(X, -)$ for the right adjoint of the left tensor and $\ihom_R(X, -)$ for the right adjoint of the right tensor. Explicitly, these functors are given by $\ihom_L(X,Y)_{n} = \cSet(\Box^n \otimes X,Y)$, $ \ihom_R(X,Y)_{n} = \cSet(X \otimes \Box^n, Y)$.
Thus the monoidal structure on $\cSet$ given by the geometric product is closed, but non-symmetric.

The construction of an arbitrary small colimit as a coequalizer of coproducts gives us the following lemma about colimts in presheaf categories.

\begin{lemma}\label{colim-factor}
Let $\mathsf{C}$ be a category  and $D$ a small diagram in $\pshf{\mathsf{C}}$. Then any map $\mathsf{C}(-,c) \to \colim D$ factors through some map in the colimit cone. \qed
\end{lemma}

This lemma allows us to describe the geometric product of cubical sets explicitly. 

\begin{proposition}\label{geometric-product-description}
For $X, Y \in \cSet$, the geometric product $X \otimes Y$ admits the following description.

\begin{itemize}
\item For $k \geq 0$, the $k$-cubes of $X \otimes Y$ consist of all pairs $(x \colon \Box^{m} \to X, y \colon \Box^{n} \to Y)$ such that $m + n = k$, subject to the identification $(x\sigma_{m_{1}+1},y) = (x,y\sigma_{1})$.
\item For $x \colon \Box^m \to X, y \colon \Box^n \to Y$, the faces, degeneracies, and connections of the $(m+n)$-cube $(x,y)$ are computed as follows:
\begin{itemize}
\item $(x,y)\bd_{i,\epsilon} = \begin{cases} (x\bd_{i,\epsilon},y) & 1 \leq i \leq m \\ (x,y\bd_{i-m,\epsilon}) & m + 1 \leq i \leq m + n \end{cases}$
\item $(x,y)\sigma_{i} = \begin{cases} (x\sigma_{i},y) & 1 \leq i \leq n_{1} + 1 \\ (x,y\sigma_{i-m}) & m + 1 \leq i \leq m + n + 1 \end{cases}$  
\item $(x,y)\gamma_{i,\varepsilon} = \begin{cases} (x\gamma_{i,\varepsilon},y) & 1 \leq i \leq m \\ (x,y\gamma_{i-m,\varepsilon}) & m + 1 \leq i \leq m + n \end{cases}$
\end{itemize}
\end{itemize}
\end{proposition}

\begin{proof}
We begin by noting that  for every pair $(x \colon \Box^{m} \to X, y \colon \Box^{n} \to Y)$ there is a corresponding $(m+n)$-cube $(x,y) \colon \Box^{m+n} \to X \otimes Y$ given by the colimit cone. Next we will show that faces, degeneracies and connections of these cones are computed as described in the statement.

For such an $(m+n)$-cube $(x,y)$, consider a face $(x,y)\bd_{i,\varepsilon}$ for $1 \leq i \leq m$. We can express the face map $\bd_{i,\varepsilon}^{m+n}$ as $\bd_{i,\varepsilon}^{m} \otimes \Box^{n}$; thus $(x,y)\bd_{i,\varepsilon} = (x\bd_{i,\varepsilon},y)$ by the naturality of the colimit cone.

\centerline{
\xymatrix{
\Box^{m-1} \otimes \Box^{n} \ar[rr]^-{\bd_{i,\varepsilon} \otimes \Box^n} \ar[drr]_{(x\bd_{i,\varepsilon},y)} & &  \Box^m \otimes \Box^n \ar[d]^{(x,y)} \\
& & X \otimes Y \\
}
}

Likewise, for $m+1 \leq i \leq m+n$ we have $\bd_{i,\varepsilon}^{m+n} = \Box^{m} \otimes \bd_{i-m,\varepsilon}^{n}$, implying $(x,y)\bd_{i,\varepsilon} = (x,y\bd_{i-m,\varepsilon})$. Similar proofs hold for degeneracies and connections. In particular, this implies that for any $(x,y)$ we have $(x\sigma_{m+1},y) = (x,y\sigma_{1})$, as both are equal to $(x,y)\sigma_{m+1}$.

To see that all cubes in $X \otimes Y$ are of this form, note that by \cref{colim-factor}, every cube of $X \otimes Y$ is equal to $(x,y)\psi$ for some such pair $(x,y)$ and some map $\psi$ in $\Box$. We have shown that the set of cubes arising from pairs is closed under faces, degeneracies and connections; since these classes generate all maps in $\Box$, this proves our claim.

Finally, we must show that the cubes of $X \otimes Y$ are not subject to any additional identifications, beyond the identification $(x\sigma_{m+1},y) = (x,y\sigma_{1})$ mentioned above. In other words, we must show that for each $k \geq 0$, $(X \otimes Y)_{k}$ is the quotient of the set $\{(x \colon \Box^{m} \to X, y \colon \Box^{n} \to Y)|m+n = k\}$ under the smallest equivalence relation $\sim$ such that $(x'\sigma_{m+1},y') \sim (x',y'\sigma_{1})$ for all $x' \colon \Box^{m'} \to X, y' \colon \Box^{n'} \to Y$ such that $m' + n' = k-1$.

To that end, let $x \colon \Box^m \to X, y \colon \Box^n \to Y, x' \colon \Box^{m'} \to X, y' \colon \Box^{n'} \to Y$, such that $m + n = m' + n'$ and $(x,y) = (x',y')$ in $(X \otimes Y)$. Without loss of generality, assume $m \geq m'$. We compute the image of this cube under the map $\pi_{X} \colon X \otimes Y \to X$.

\begin{align*}
\pi_{X}(x,y) & = \pi_{X}(x',y) \\
\therefore x\sigma_{m+1}\sigma_{m+2}...\sigma_{m+n} & = x'\sigma_{m'+1}...\sigma_{m'+n'} \\ 
\end{align*}

(If $n$ or $n'$ is equal to 0, we interpret the corresponding string of degeneracies to be empty.) We can apply face maps to both sides of this equation to reduce the left-hand side to $x$. If $m = m'$ then this gives the equation $x = x'$, and a similar calculation shows $y = y'$. Otherwise, we have $x = x'\sigma_{m'+1}...\sigma_{m}$. In this case, a similar calculation shows $y' = y\sigma_{1}...\sigma_{1}$, where $\sigma_{1}$ is applied $m-m'$ times on the right-hand side of the equation. From this we can see that $(x,y) \sim (x',y')$. Thus we see that quotienting the set of pairs $(x,y)$ of appropriate dimensions by $\sim$ does indeed suffice to obtain $(X \otimes Y)_{k}$.
\end{proof}

\begin{corollary}
For cubical sets $X$ and $Y$, we have $(X \otimes Y)_1 \cong (X_1 \times Y_0) \cup_{(X_0 \times Y_0)} (X_0 \times Y_1)$.  \qed
\end{corollary}

By the \emph{pushout product} $f \hatotimes g$ of maps $f \colon A \to B$ and $g \colon X \to Y$, we will denote the canonical map from the pushout $A \otimes Y \cup_{A \otimes X} B \otimes X \to B \otimes Y$.
The following lemma, which can be verified by simple computation, allows us to express boundary inclusions and open box inclusions as pushout products with respect to this monoidal structure.

\begin{lemma}\label{open-box-pop}\leavevmode
\begin{enumerate}
  \item For $m, n \geq 0$, we have 
  \[ (\bd \Box^{m} \to \Box^{m}) \hatotimes (\bd \Box^{n} \to \Box^{n}) = (\bd \Box^{m+n} \to \Box^{m+n})\text{.}\]
  \item For $1 \leq i \leq m$ and $\varepsilon \in \{0,1\}$, the open-box inclusion $\sqcap^{n}_{i,\varepsilon} \hookrightarrow \Box^{n}$ is the pushout product 
  \[(\bd \Box^{i-1} \hookrightarrow \Box^{i-1}) \hatotimes (\{1-\varepsilon\} \hookrightarrow \Box^{1}) \hatotimes (\bd \Box^{m-i} \hookrightarrow \Box^{m-i}) \text{.} \]
\end{enumerate}
\end{lemma}

We conclude this section by defining certain functors which we will use to compare model structures.

We define two endofunctors $(-)^\rmco, (-)^\coop  \colon \Box \to \Box$ as follows:

\begin{itemize}
\item Both $(-)^\coop$ and $(-)^\rmco$ act as the identity on objects;
\item $(-)^\rmco$ acts on generating morphisms as follows:
\begin{itemize}
\item $(\bd^{n}_{i,\varepsilon})^\rmco = \bd^{n}_{n-i+1,\varepsilon}$;
\item $(\sigma^{n}_{i})^\rmco = \sigma^{n}_{n-i+1}$;
\item $(\gamma^{n}_{i,\varepsilon})^\rmco = \gamma^{n}_{(n-1)-i+1,\varepsilon}$
\end{itemize}
\item $(-)^\coop$ acts on generating morphisms as follows:
\begin{itemize}
\item $(\bd^{n}_{i,\varepsilon})^\coop = \bd^{n}_{i,1-\varepsilon}$;
\item $(\sigma^{n}_{i})^{\coop} = \sigma^{n}_{i}$;
\item $(\gamma^{n}_{i,\varepsilon})^\coop = \gamma^{n}_{i,1-\varepsilon}$.
\end{itemize}
\end{itemize}

From the definition we can see that the endofunctors $(-)^\rmco$ and $(-)^\coop$ commute; we denote their composite by $(-)^\op$.

By left Kan extension, we obtain functors $(-)^\rmco, (-)^\coop, (-)^\op \colon \cSet \to \cSet$.

\centerline{
\xymatrix{
\Box \ar@{^(->}[d] \ar[r] & \Box \ar@{^(->}[r] & \cSet \\
\cSet \ar[urr] \\
}
}

Some simple computations show:

\begin{lemma}\label{op-co-involution}
The functors $(-)^\rmco, (-)^\coop, (-)^\op$ are involutions of $\cSet$. \qed
\end{lemma}

In particular, for $X \in \cSet$, the cubes of $X$ are in bijection with those of $X^{\rmco}$, $X^{\coop}$, and $X^{\op}$; given $x \colon \Box^n \to X$ we have corresponding cubes $x^\rmco \colon \Box^n = (\Box^n)^\rmco \to X^\rmco$, $x^\coop \colon \Box^n  = (\Box^n)^\coop \to X^\coop$, $x^\op \colon \Box^n = (\Box^n)^\op \to X^\op$. 

Let $\Box_0$ denote the subcategory of $\Box$ generated by the face, degeneracy, and negative connection maps, and let $\cSet_0$ denote the presheaf category $\pshf{\Box_0}$. This is the category of cubical sets studied in \cite{kapulkin-lindsey-wong}.

By pre-composition, the inclusion $i \colon \Box_0 \hookrightarrow \Box$ defines a functor $i^\ast \colon \cSet \to \cSet_0$. Left and right Kan extension define left and right adjoints of this functor, respectively denoted $i_!, i_{\ast} \colon \cSet_0 \to \cSet$.

We may characterize the functors $i^\ast, i_{\ast}, i_!$ as follows:

\begin{itemize}
\item For $X \in \cSet$, $n \geq 0$ we have $(i^\ast X)_n = X_n$, with structure maps computed as in $X$. However, certain degenerate cubes of $X$  become non-degenerate in $i^\ast X$, namely those which cannot be expressed as degeneracies or negative connections of any other cube.
\item For $X \in \cSet_0$, we have $(i_{\ast} X)_n = \cSet_0(i^\ast \Box^n, X)$.
\item  For $X \in \cSet_0$, $i_! X$ is obtained by freely adding positive connections to $X$. Given a map $f \colon X \to Y$ in $\cSet_0$, $i_! f$ agrees with $f$ on the non-degenerate cubes of $i_! X$; by \cref{map-on-non-degen} this is enough to determine $i_! f$.
\end{itemize}

\begin{lemma}\label{i-adjunction-non-degenerate}
Given a map $f \colon i_! X \to Y$ in $\cSet$, $f$ and the adjunct map $\overline{f} \colon X \to i^\ast Y$ agree on non-degenerate cubes. \qed
\end{lemma}

Similarly, we let $\Box_1$ denote the subcategory of $\Box$ generated by the face, degeneracy, and positive connection maps, and let $\cSet_1$ denote the presheaf category $\pshf{\Box_1}$. Likewise, we let  $\Box_{\varnothing}$ denote the subcategory of $\Box$ generated by the face and degeneracy maps, and let $\cSet_{\varnothing}$ denote the presheaf category $\pshf{\Box_{\varnothing}}$. Most of our arguments will not require fixing a specific choice of box category, so that our main results will be valid in all of these categories of cubical sets. For concreteness, we will work with the category $\cSet$ of cubical sets having both positive and negative connections; at the beginning of each section we will note the extent to which that section's results apply to the other three categories of cubical sets. 

The restriction of the nerve functor defines a functor $\Box \to \sSet$; taking the left Kan extension of this functor along the Yoneda embedding, we obtain the \emph{triangulation} functor $T \colon \cSet \to \sSet$.

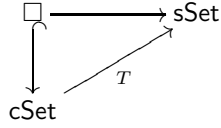
\[
\xymatrix@C+0.5cm{
  \Box
  \ar[r]
  \ar@{^{(}->}[d]
&
  \sSet
\\
  \cSet
  \ar[ru]_{T}
&
}
\]

The triangulation functor has a right adjoint $U \colon \sSet \to \cSet$ given by $(UX)_{n} = \sSet((\Delta^1)^n,X)$. Intuitively, we think of triangulation as creating a simplicial set $TX$ from a cubical set $X$ by subdividing the cubes of $X$ into simplices.

We now record two basic facts about triangulation. In the given references, these results are proven using a different definition of the category $\Box$, lacking connection maps, but the proofs apply equally well to the cubical sets under consideration here.

\begin{proposition}[{\cite[Ex.~8.4.24]{CisinskiAsterisque}}]\label{T-prod}
The triangulation functor sends geometric products to cartesian products; that is, for cubical sets $X$ and $Y$, there is a natural isomorphism $T(X \otimes Y) \cong TX \times TY$. \qed
\end{proposition}

\begin{corollary}\label{T-pres-pop}
Triangulation preserves pushout products; that is, for maps $f, g$ in $\cSet$ there is a natural isomorphism $T(f \hatotimes g) \cong Tf \widehat{\times} Tg$.
\end{corollary}

\begin{proof}
Immediate by \cref{T-prod} and the fact that $T$ preserves colimits as a left adjoint.
\end{proof}

\begin{proposition}[{\cite[Lem.~8.4.29]{CisinskiAsterisque}}]\label{T-pres-cof}
The triangulation functor preserves monomorphisms. \qed
\end{proposition}

\subsection{Homotopy theory of cubical sets} 

\begin{lemma}\label{cSetCellularModel}
The boundary inclusions $\bd \Box^{n} \to \Box^{n}$ form a cellular model for $\cSet$. 
\end{lemma}

\begin{proof}
This follows from \cref{Box-Reedy}, which allows us to apply \cite[Prop.~1.5]{ara:higher-quasicats}. 
(Note that the definition of boundary objects for an arbitrary Reedy category appearing in \cite{ara:higher-quasicats} recovers precisely the objects $\bd \Box^{n}$ when instantiated with $\Box$.)
\end{proof}

\begin{definition}
A map of cubical sets is a \emph{Kan fibration} if it has the right lifting property with respect to all open box fillings. A cubical set $X$ is a \emph{cubical Kan complex} if the map $X \to \Box^{0}$ is a Kan fibration.
\end{definition}

The functor $\Box^{1} \otimes - \colon \cSet \to \cSet$, together with the natural transformations $\bd^{1}_{1,0} \otimes -, \bd^{1}_{1,1} \otimes - \colon \mathrm{id}  \to \Box^{1} \otimes -$, and $\pi \colon \Box^{1} \otimes - \to \mathrm{id}$, defines a cylinder functor on $\cSet$ in the sense of \cref{CylinderFunctor}. Thus, for any $X, Y \in \cSet$ we have a set $[X,Y]$ of homotopy classes of maps from $X$ to $Y$ defined by this cylinder functor.

\begin{theorem}[Cisinski] \label{Grothendieck-ms-cSet}
  The category $\cSet$ carries a cofibrantly generated model structure, referred to as the Grothendieck model structure, in which
  \begin{itemize}
    \item cofibrations are the monomorphisms;
    \item weak equivalences are maps $X \to Y$ inducing bijections $[Y,Z] \to [X,Z]$ for all cubical Kan complexes $Z$;
    \item fibrations are the Kan fibrations.
  \end{itemize}
\end{theorem}

\begin{proof}
The existence of the model structure and characterization of the cofibrations, weak equivalences, and fibrant objects follows from applying \cref{CisinskiMS} with the cylinder functor $I = \Box^1 \otimes -$, cellular model $M = \{\bd\Box^{n} \to \Box^{n} | n \geq 0\}$, and $S = \varnothing$. The characterization of the fibrations is given in \cite[Thm. 1.7]{CisinskiUniverses}. (This proof is for cubical sets with only negative connections, but the proof for cubical sets with both kinds of connections is identical.)
\end{proof}

\begin{proposition}\label{cSet-strict-test}
The category $\Box$ is a strict test category, and the test category model structure on $\cSet$ is the Grothendieck model structure.
\end{proposition}

\begin{proof}
For the proof that $\Box$ is a strict test category, see \cite[Thm. 4.3]{maltsiniotis:connections-strict-test-cat} and \cite[Thm. 3]{buchholtz-morehouse:varieties-of-cubes}. That the test category model structure is the Grothendieck model structure follows from \cite[Thm. 1.7]{CisinskiUniverses}.
\end{proof}

The canonical inclusion $\Box \to \Cat$ induces the adjoint pair $\tau_1 \colon \cSet \rightleftarrows \Cat : \! \Nerve$ via hom-out and the left Kan extension.
In particular, $\Nerve(\catC)_n = \Cat([1]^n, \catC)$.
The functor $\tau_1$ takes a cubical set $X$ to its \emph{fundamental category}, which is obtained as the quotient of the free category on the graph $X_1 \rightrightarrows X_0$ modulo the relations: $\sigma_1 x = \id_x$ and $gf = qp$ for every $2$-cube
\[
\xymatrix{
  \cdot
  \ar[r]^f
  \ar[d]_p
&
   \cdot
  \ar[d]^g
\\
  \cdot
  \ar[r]^q
&
 \cdot
}
\]

\subsection{Marked cubical sets}\label{Sec1Marked}
To define marked cubical sets, we need to introduce a new category $\tBox$, a slight enlargement of $\Box$.
The category $\tBox$ consists of objects of the form $[1]^n$ for $n = 0, 1, \ldots$ and an object $[1]_e$.
The maps of $\tBox$ are generated by the usual generating maps of $\Box$ along with $\varphi \colon [1] \to [1]_e$ and $\zeta \colon [1]_e \to [1]^0$ subject to an additional identity $\zeta \varphi = \sigma^1_1$.

\begin{proposition-qed} \label{tBox-Reedy}
  The category $\tBox$ is an EZ-Reedy category with the Reedy structure defined as follows:
  \begin{itemize}
    \item $\deg ([1]^0) = 0$, $\deg [1] = 1$, $\deg ([1]_e) = 2$, and $\deg ([1]^n) = n+1$ for $n \geq 2$;
    \item $(\tBox)_+$ is generated by face maps and $\varphi$ under composition;
    \item $(\tBox)_-$ is generated by degeneracy maps, connections, and $\zeta$ under composition. \qedhere
  \end{itemize}
\end{proposition-qed}

A \emph{structurally marked cubical set} is a contravariant functor $X \colon \tBox^\op \to \Set$ and a morphism of structurally marked cubical sets is a natural transformation of such functors.
We will write $\cSet''$ for the category of structurally marked cubical sets.
When working with the category of structurally marked cubical sets, we will write $X_n$ for the value of $X$ at $[1]^n$ and $X_e$ for the value of $X$ at $[1]_e$.

Structurally marked cubical sets should be thought of as cubical sets with (possibly multiple) labels on their edges such that for each vertex $x$, the degenerate edge $x\sigma_{1}$ has, in particular, a distinguished label $x\zeta$.

A \emph{marked cubical set} is a structurally marked cubical set for which the map $X_e \to X_1$ is a monomorphism.
We write $\cSet'$ for the category of marked cubical sets.
Alternatively, we may view a marked cubical set as a pair $(X, W_X)$ consisting of a cubical set $X$ together with a subset $W_X \subseteq X_1$ of edges that includes all degenerate edges and a morphism of marked cubical sets is a map of cubical sets that preserves marked edges.

The functor taking a (structurally) marked cubical set to its underlying cubical set admits both a left and a right adjoint, given by the minimal and maximal marking respectively.
The minimal marking on a cubical set $X$, denoted $X^\flat$, marks exactly the degenerate edges, whereas the maximal marking, denoted $X^\sharp$, marks all edges of $X$.
If considered as structurally marked cubical sets, the marked edges of $X^\flat$ and $X^\sharp$ are marked exactly once.
Altogether we obtain the following adjunctions
\[
\xymatrix@C+1.5cm{
  \cSet'(')
  \ar[r]
&
  \cSet
  \ar@/^1pc/[l]^{(-)^\sharp}
  \ar@/_1pc/[l]_{(-)^\flat}
}
\]
The notation $\cSet'(')$ above indicates that the same constructions can be applied to both marked and structurally marked cubical sets.
In the context of (structurally) marked cubical sets, we regard a cubical set with its minimal marking by default, writing $X$ for $X^\flat$.

There is moreover an inclusion $\cSet' \to \cSet''$. This inclusion admits a left adjoint taking $X \in \cSet''$ to $\Im X$ given by $(\Im X)_n = X_n$ and $(\Im X)_e = \varphi^*(X_e)$, i.e., the image of $X_e$ under $\varphi^* = X(\varphi)$. The inclusion is easily seen to not have a right adjoint, since it fails to preserve the pushout of $\Box^1 \to (\Box^1)^\sharp$ against itself.

Altogether we obtain the following diagram: 
\begin{equation} \label{eq:all-functors-marking} \tag{*}
\vcenter{\xymatrix@C+1.5cm{
  \cSet''
  \ar[rdd]
  \ar@/^1pc/[rr]^{\Im}
&
&
  \cSet' 
  \ar[ldd]
  \ar@{_{(}->}[ll]
\\ 
&
&
\\
&
  \cSet
  \ar@/^1pc/[luu]^{(-)^\flat}
  \ar@/_1pc/[luu]_{(-)^\sharp}
  \ar@/^1pc/[ruu]^{(-)^\flat}
  \ar@/_1pc/[ruu]_{(-)^\sharp}
&
}}
\end{equation}

A geometric product entirely analogous to that of Section \ref{subsection:cSet-basics} exists for structurally marked cubical sets. We extend $\Box \times \Box \to \cSet$ to $\Box_\sharp \times \Box_\sharp \to \cSet''$ by taking $[1]_e \otimes [n]$ to have $\Box^{n+1}$ as the underlying cubical set with edges of the form $(0,x_2,\ldots,x_{n+1}) < (1,x_2,\ldots, x_{n+1})$ uniquely marked. Similarly, let $[n] \otimes [1]_e$ have $\Box^{n+1}$ as its underlying cubical set, and marked edges those of the form $(x_1,\ldots,x_n, 0) < (x_1,\ldots,x_n,1)$. Finally, let $[1]_e \otimes [1]_e := (\Box_2)^\sharp$. The left Kan extension yields $\otimes \colon \cSet'' \times \cSet'' \to \cSet''$.

This geometric product admits a concrete description analogous to that of \cref{geometric-product-description}.

\begin{proposition}\label{geometric-product-marked}
For $X, Y \in \cSet''$, the geometric product $X \otimes Y$ admits the following description.

\begin{itemize}
\item The underlying cubical set of $X \otimes Y$ is the geometric product of the underlying cubical sets of $X$ and $Y$.
\item $(X \otimes Y)_{e}$ is the set of all pairs of the form $(\overline{x} \colon (\Box^1)^\sharp \to X,y \colon \Box^0 \to Y)$ or $(x \colon \Box^0 \to X, \overline{y} \colon (\Box^1)^\sharp \to Y)$, subject to the identification $(x\zeta,y) = (x,y\zeta)$ for $x \colon \Box^0 \to X, y \colon \Box^0 \to Y$.
\item Structure maps not arising from those of the underlying cubical set are computed as follows:
\begin{itemize}
\item $(x,y)\zeta = (x\zeta,y) = (x,y\zeta)$;
\item $(\overline{x},y)\varphi = (\overline{x}\varphi,y)$;
\item $(x,\overline{y})\varphi = (x,\overline{y}\varphi)$.
\end{itemize}
\end{itemize}
\end{proposition}

\begin{proof} 
To compute the underlying cubical set of $X \otimes Y$, we analyze maps $\Box^{k} \to X \otimes Y$ exactly as in the proof of \cref{geometric-product-description}.

Now we consider maps $(\Box^1)^\sharp \to X \otimes Y$. First observe that for every pair of maps $\overline{x} \colon (\Box^1)^\sharp \to X, y \colon \Box^0 \to Y$ we have a map $(x,y) \colon (\Box^1)^\sharp \cong (\Box^1)^\sharp \otimes \Box^0 \to X \otimes Y$ in the colimit cone, and the same holds for $x \colon \Box^0 \to X, \overline{y} \colon (\Box^1)^\sharp \to Y$. Once again, the stated computations of structure maps follow from the naturality of the colimit cone.

Now we will show that every map $p \colon (\Box^1)^\sharp \to X \otimes Y$ has the form described above. By \cref{colim-factor}, for every such map we have a commuting diagram

\centerline{
\xymatrix{
(\Box^{1})^\sharp \ar[r]^-{\psi} \ar[dr]_{p} & \Box^m_{(e)} \otimes \Box^n_{(e)} \ar[d] \\
& X \otimes Y \\
}
}

where $\Box^m_{(e)}, \Box^{n}_{(e)}$ denote representable presheaves, and the map $\Box^m_{(e)} \otimes \Box^n_{(e)} \to X \otimes Y$ is part of the colimit cone. We proceed by case analysis on $\Box^m_{(e)}$ and $\Box^{n}_{(e)}$.

First note that if $\psi$ factors through $\zeta$, then $p = (x,y)\zeta$ for some $x \colon \Box^0 \to X, y \colon \Box^0 \to Y$. This takes care of the case $\Box^m_{(e)} = \Box^m, \Box^n_{(e)} = \Box^n$, since any map from $(\Box^1)^\sharp$ into these objects factors through $\zeta$. 

Now assume $\psi$ does not factor through $\zeta$, implying that at least one of $\Box^m_{(e)}, \Box^n_{(e)}$ is $(\Box^1)^\sharp$; then $\Box^m_{(e)} \otimes \Box^n_{(e)}$ is either $\Box^m \otimes (\Box^{1})^\sharp$, $(\Box^1)^\sharp \otimes \Box^n$, or $(\Box^2)^\sharp$. Since every map $(\Box^1)^\sharp \to (\Box^2)^\sharp$ factors through either $\Box^1 \otimes (\Box^1)^\sharp$ or $(\Box^1)^\sharp \otimes \Box^1$, we need only consider the first two cases. If $\Box^m_{(e)} \otimes \Box^n_{(e)} = \Box^m \otimes (\Box^1)^\sharp$, then $\psi$ picks out the unique marking on an edge of the form $(x_1,\ldots,x_m, 0) < (x_1,\ldots,x_m,1)$. In other words, $\psi$ factors through the map $(x_1,\ldots,x_m) \otimes (\Box^1)^\sharp \colon \Box^0 \otimes (\Box^1)^\sharp \to \Box^m \otimes (\Box^1)^\sharp$. 

Thus we have reduced the problem to the case $\Box^m_{(e)} = \Box^0, \Box^n_{(e)} = (\Box^1)^\sharp$, implying $\Box^m_{(e)} \otimes \Box^n_{(e)} = \Box^1_{\sharp}$. Since the only endomorphism of $(\Box^1)^\sharp$ which does not factor through $\zeta$ is the identity, this implies that $p = (x,\overline{y})$ for some $x \colon \Box^0 \to X, \overline{y} \colon (\Box^1)^\sharp \to Y$. In the case $\Box^m_{(e)} = (\Box^1)^\sharp, \Box^n_{(e)} = \Box^n$, a similar analysis shows $p = (\overline{x},y)$ for some $\overline{x} \colon (\Box^1)^\sharp \to X, y \colon \Box^0 \to Y$.

To show that the elements of $(X \otimes Y)_{e}$ are subject to no further identifications, consider two pairs $(\overline{x},y), (\overline{x}',y')$ which are identified in $(X \otimes Y)_{e}$. Considering the image of the cube corresponding to these pairs under the projections $\pi_{X}$, $\pi_{Y}$, we see that $\overline{x} = \overline{x}', y = y'$. A similar proof holds for identified pairs of the form $(x,\overline{y}), (x',\overline{y}')$. Finally, if $(\overline{x},y) = (x,\overline{y})$, then applying the projections shows $\overline{x} = x\zeta, \overline{y} = y\zeta$.
\end{proof}

\begin{corollary}\label{geometric-product-marked-pushout}
For $X, Y \in \cSet''$, $(X \otimes Y)_e \cong (X_e \times Y_0) \cup_{(X_0 \times Y_0)} (X_0 \times Y_e)$. \qed
\end{corollary}

What \cref{geometric-product-marked} shows is that for $x \colon \Box^1 \to X, y \colon \Box^0 \to Y$, the set of markings on the edge $(x,y)$ in $X \otimes Y$ is simply the set of markings on $x$ in $X$, that the analogous result holds for $x \colon \Box^0 \to X, y \colon \Box^1 \to Y$, and that for a pair of vertices $x,y$ the distinguished marking $(x,y)\zeta$ is identified with both $(x\zeta,y)$ and $(x,y\zeta)$.

\begin{proposition}
The geometric product on $\cSet''$ restricts to a monoidal product on $\cSet'$. \qed
\end{proposition}

\begin{corollary}\label{flat-monoidal}
  All functors in the diagram (\ref{eq:all-functors-marking}) are monoidal. \qed
\end{corollary}

As in the case of cubical sets, given a marked cubical set $A$, we form two non-isomorphic functors $\cSet'(') \to \cSet'(')$: the left tensor $- \otimes A$ and the right tensor $A \otimes -$.
As they are both co-continuous, they admit right adjoints and we write $\ihom_L(A, -)$ for the right adjoint of the left tensor $- \otimes A$ and $\ihom_R(A, -)$ for the right adjoint of the right tensor $A \otimes -$.

Observe that we may extend the functors $(-)^\rmco, (-)^\coop, (-)^\op \colon \Box \to \Box$ of Section \ref{subsection:cSet-basics} to obtain involutions $(-)^\rmco, (-)^\coop, (-)^\op \colon \Box_{\sharp} \to \Box_{\sharp}$, by having these functors act as the identity on the object $[1]_{e}$ and the maps $\phi$ and $\zeta$. By left Kan extension we obtain involutions $(-)^\rmco, (-)^\coop \colon \cSet'' \to \cSet''$, which restrict to involutions of $\cSet'$. Given $X' \in \cSet'$ with underlying cubical set $X$, the underlying cubical set of $(X')^\rmco$ is $X^\rmco$, with an edge $x^\rmco \colon \Box^1 \to (X')^\rmco$ marked if and only if $x$ is marked in $X'$, and similarly for $(X')^\coop$.

\begin{proposition}[{\cite[Prop. 1.17]{campion-kapulkin-maehara}}]\label{op-co-monoidal}
The endofunctors $(-)^\rmco, (-)^\coop, (-)^\op$ on $\cSet, \cSet'$, and $\cSet''$ interact with the geometric product as follows:

\begin{itemize}
\item The functor $(-)^\rmco$ is strong anti-monoidal, i.e. $(X \otimes Y)^\rmco \cong Y^\rmco \otimes X^\rmco$;
\item The functor $(-)^\coop$ is strong monoidal, i.e. $(X \otimes Y)^\coop \cong X^\coop \otimes Y^\coop$;
\item The functor $(-)^\op$ is strong anti-monoidal, i.e. $(X \otimes Y)^\op \cong Y^\op \otimes X^\op$. \qed
\end{itemize}
\end{proposition}

Using \cref{op-co-monoidal} and the adjunctions $(-)^\rmco \adjoint (-)^\rmco, (-)^\coop \adjoint (-)^\coop$, we obtain:

\begin{corollary}\label{op-co-hom}
For $X, Y$ in $\cSet$, $\cSet'$, or $\cSet''$, we have isomorphisms, natural in $X$ and $Y$:

\begin{itemize}
\item $\ihom_{L}(X,Y)^\rmco \cong \ihom_{R}(X^\rmco,Y^\rmco)$, $\ihom_{R}(X,Y)^\rmco \cong \ihom_{L}(X^\rmco,Y^\rmco)$;
\item $\ihom_{L}(X,Y)^\coop \cong \ihom_{L}(X^\coop,Y^\coop)$, $\ihom_{R}(X,Y)^\coop \cong \ihom_{R}(X^\coop,Y^\coop)$;
\item $\ihom_{L}(X,Y)^\op \cong \ihom_{R}(X^\op,Y^\op)$, $\ihom_{R}(X,Y)^\op \cong \ihom_{L}(X^\op,Y^\op)$. \qed
\end{itemize}
\end{corollary}

Finally, we relate the adjunction $T \adjoint U$ to the involutions $(-)^\rmco$, $(-)^\coop$, and $(-)^\op$ of $\cSet$ and the involution $(-)^\op$ of $\sSet$.

\begin{proposition}
We have the following natural isomorphisms in $\sSet$ and $\cSet$:
\begin{multicols}{2}
\begin{enumerate}
\item\label{T-co} $T \circ (-)^\rmco \cong T$;
\item\label{T-coop} $T \circ (-)^\coop \cong (-)^\op \circ T$;
\item $T \circ (-)^\op \cong (-)^\op \circ T$;
\item $(-)^\rmco \circ U \cong U$;
\item $(-)^\coop \circ U \cong U \circ (-)^\op$;
\item $(-)^\op \circ U \cong U \circ (-)^\op$.
\end{enumerate}
\end{multicols}
\end{proposition}

\begin{proof}
It suffices to prove \ref{T-co} and \ref{T-coop}. As $T$ and the involutions preserve colimits, it suffices to establish the desired natural isomorphisms on the objects $\Box^n$. For this, observe that the maps between these objects are generated, under composition and the geometric product, by the maps $\bd_{1,\varepsilon} \colon [0] \to [1]$, $\sigma_1 \colon [1] \to [0]$, and $\gamma_{1,\varepsilon} \colon [1]^2 \to [1]$. By \cref{T-prod,op-co-monoidal}, it thus suffices to show that $T \circ (-)^\rmco$ and $T$ (resp. $T \circ (-)^\coop$ and $(-)^\op \circ T$) agree on these maps; this can easily be verified.
\end{proof}

For reference, the following diagrams show the inclusions of the four box categories under consideration, and the functors between the corresponding categories of cubical sets induced by pre-composition.

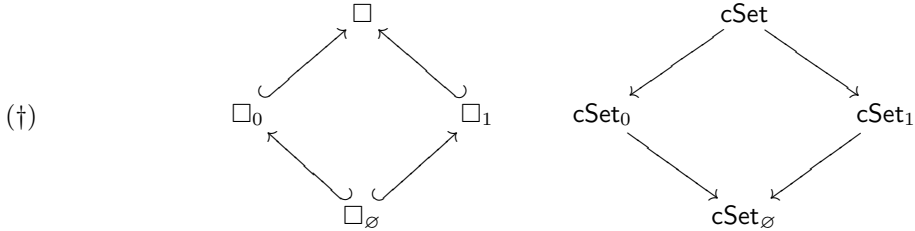
\begin{equation}\label{eq:cube-cat-inclusions} \tag{$\dagger$}
\vcenter{\xymatrix{
& \Box  & & & \cSet \ar[dl] \ar[dr] \\
\Box_0 \ar@{^(->}[ur] && \Box_1 \ar@{_(->}[ul] & \cSet_0 \ar[dr] && \cSet_1 \ar[dl] \\
& \Box_\varnothing \ar@{^(->}[ur] \ar@{_(->}[ul]  & & & \cSet_\varnothing\\
}}
\end{equation}

The following result shows that, regardless of our choice of box category, certain distinguished cubical sets admit connections on 1-cubes in a suitable sense (though these will be non-degenerate in general). This will allow us to freely make use of such cubes in some of our proofs, even when not working in $\cSet$.

The following result shows that marked cubical quasicategories admit connections on 1-cubes, in a suitable sense.

\begin{proposition}\label{cqcats-have-connections}
Let $X$ be a cubical set having fillers for all open boxes with the critical edge degenerate, and $x \colon \Box^1 \to X$ a 1-cube of $X$. Then there exist a pair of 2-cubes of $X$ as depicted below:

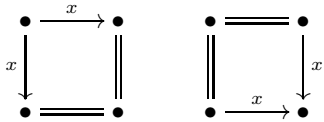
\centerline{
\xymatrix{
\bullet \ar[r]^{x} \ar[d]_{x} & \bullet \ar@{=}[d] & \bullet \ar@{=}[r] \ar@{=}[d] & \bullet \ar[d]^{x}  \\
\bullet \ar@{=}[r] & \bullet & \bullet \ar[r]^{x} & \bullet \\
}
}
\end{proposition}

\begin{proof}
We prove the statement for the left-hand cube, i.e. the negative connection; the proof for the right-hand cube, i.e. the positive connection, is dual. In $\cSet_0$ and $\cSet$ this is trivial, but we will provide a general proof which is also valid in $\cSet_{\mathrm{\varnothing}}$ and $\cSet_1$.

We may fill an open box in dimension 2 to obtain a 2-cube as depicted below:

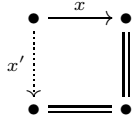
\centerline{
\xymatrix{
\bullet \ar[r]^{x} \ar@{..>}[d]_{x'} & \bullet \ar@{=}[d] \\
\bullet \ar@{=}[r] & \bullet \\
}
}

The edge $x'$, however, may not be equal to $x$. To correct for this possibility, we may fill the following open box in dimension 3, where the top and left faces are the 2-cube depicted above,  while the back face is missing:

 
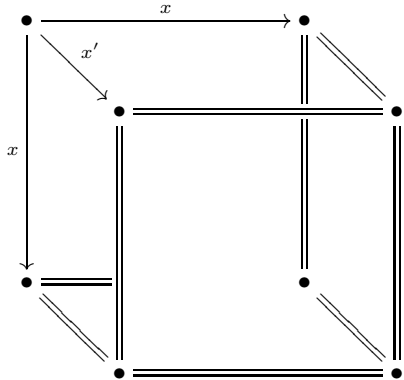
\[
\xymatrix@!C{
 \bullet
 \ar[rrr]^{x}
 \ar[ddd]_{x}
 \ar[dr]^{x'}
&&&
 \bullet
 \ar@{=}[ddd]|!{[dl];[dr]}{\hole}
 \ar@{=}[dr]
\\&
 \bullet
 \ar@{=}[rrr]
 \ar@{=}[ddd]
&&&
 \bullet
 \ar@{=}[ddd]
\\
\\
 \bullet
 \ar@{=}[rrr]|!{[uur];[dr]}{\hole}
 \ar@{=}[dr]
&&&
 \bullet
 \ar@{=}[dr]
\\&
 \bullet
 \ar@{=}[rrr]
&&&
 \bullet
}
\]

The filler for the back face is the desired 2-cube.
\end{proof}

In particular, the result above applies in the \emph{marked cubical quasicategories} to be studied in \cref{section:marked}, as well as in the \emph{cubical quasicategories} to be studied in \cref{sec:joyal-cset}.

\section{Model structure on marked cubical sets}\label{section:marked}
The goal of this section is to construct a combinatorial model category structure on the category $\cSet'$ of marked cubical sets.
One would like to do this by applying Cisinski theory, as described in Section \ref{Sec1Cisinski}, but unfortunately $\cSet'$ is not a presheaf category.
Although there exists a generalization of Cisinski theory to a non-topos case (due to Olschok \cite{olschok:thesis}), we choose to construct the model structure directly, using Jeff Smith's \cref{Jeff-Smith's-theorem} to obtain a better understanding of it as a result.
It is also worth pointing out that our language (e.g., cellular model, cylinder functor) follows the conventions of Cisinski to make the analogy with the Cisinski machinery clear.

None of the arguments in this section rely on the existence of connections, thus they are valid in all of the categories of cubical sets shown in the diagram \cref{eq:cube-cat-inclusions} at the end of \cref{sec:background}.

\subsection{Classes of maps}

To begin, we lay out the definitions of the classes of maps that will comprise the model structure.

The \emph{cofibrations} are the monomorphisms.
The \emph{trivial fibrations} are the maps with the right lifting property with respect to the cofibrations.

Using \cref{cSetCellularModel}, one obtains:

\begin{lemma} \label{cofibration-generators}
The cofibrations are the saturation of the set consisting of the boundary inclusions $\bd \Box^n \to \Box^n$ for $n \geq 0$ and the inclusion $\Box^1 \to (\Box^1)^\sharp$. \qed
\end{lemma}

By \cref{cofibration-generators}, we have a cofibrantly generated weak factorization system (cofibrations, trivial fibrations).

\begin{definition}\label{gen-anodyne-def} We introduce three classes of maps in $\cSet'$.
\begin{enumerate}
\item Let the \emph{marked open box inclusions} $\iota_{i,\varepsilon}^{n}$ be the marked cubical set maps whose underlying cubical set maps are the open box inclusions $\sqcap^n_{i,\varepsilon} \to \Box^n$, with the critical edge marked in each (except for the domain of $\iota_{i,\varepsilon}^{1}$, i.e. $\Box^{0}$, in which the critical edge is not present).
  \item Let $K$ be the cubical set depicted as:
\[
\xymatrix{
  1 \ar[r] \ar@{=}[d] &0 \ar[d] \ar@{=}[r] & 0 \ar@{=}[d] \\
  1 \ar@{=}[r] &1 \ar[r] &0 }
\] Let $K'$ be the marked cubical set that has the middle edge in the above marked. Define the \emph{saturation} map to be the inclusion $K \subseteq K'$.
  \item For each of the four faces of the square, let the \emph{3-out-of-4 map} associated to that face be the inclusion of $\Box^2$ with all but that face marked into $(\Box^2)^\sharp$. 
\end{enumerate}
\end{definition}

The \emph{anodyne maps} are defined as the saturation of the set of maps consisting of the marked open box inclusions, the saturation map, and the 3-out-of-4 maps.
The \emph{naive fibrations} are those maps that have the right lifting property against anodyne maps. Call an object $X$ of $\cSet'$ a \emph{marked cubical quasicategory} if the map $X \to \Box^{0}$ is a naive fibration. 

Note that the definition of a marked open box inclusion combines the intuition behind both the inner and the special outer horns from the theory of marked simplicial sets.
For instance, filling 2-dimensional marked open box amounts to composing two edges with an inverse of an equivalence, as can be seen in the following diagrams:

\centerline{
\xymatrix{
\bullet \ar[r] & \bullet \ar[d] & \bullet \ar[d] \ar[r]^{\sim} & \bullet & \bullet \ar[d] & \bullet \ar[d]^{\sim} & \bullet \ar[d]_{\sim} \ar[r] & \bullet \ar[d] \\
\bullet \ar[r]^{\sim} & \bullet & \bullet \ar[r] & \bullet & \bullet \ar[r] & \bullet & \bullet & \bullet \\
\ar@{}[r]^{\iota_{1,0}^{2}} & & \ar@{}[r]^{\iota_{1,1}^{2}}  & & \ar@{}[r]^{\iota_{2,0}^{2}}  &  & \ar@{}[r]^{\iota_{2,1}^{2}} & \\ 
}
}

\begin{remark}
Viewing marked cubical quasicategories as $(\infty,1)$-categories, the marked edges represent equivalences. The generating anodyne maps have the following $(\infty,1)$-categorical meanings.
\begin{itemize}
\item The $n$-dimensional marked open box fillings for $n \geq 2$ correspond to composition of maps and homotopies, analogous to filling inner and special horns in quasicategories.
 They also ensure that every morphism presented by a marked edge has a left and right inverse, i.e., is an equivalence.
\item The 1-dimensional marked open box fillings, $\iota_{1,\varepsilon}^{1} \colon \Box^{0} \to (\Box^{1})^{\sharp}$, are the inclusions of endpoints into the marked interval; thus marked edges may be lifted along naive fibrations, analogous to the lifting of isomorphisms along isofibrations in 1-category theory.
\item The saturation map ensures that equivalences, maps having both left and right inverses, are marked.
\item The 3-out-of-4 maps represent the principle that if three maps in a commuting square are equivalences, then so is the fourth.
 They encode a condition analogous to the two-out-of-three property.
\end{itemize}
\end{remark}

\begin{remark}
For $n \geq 1$, the representable marked cubical set $\Box^{n}$ is not a marked cubical quasicategory, as it lacks fillers for certain marked open boxes. This stands in constrast to the case of simplicial sets, in which the representables $\Delta^{n}$ are quasicategories.
\end{remark}

\begin{lemma}\label{equivalence-iff-marked}
Let $X$ be a marked cubical quasicategory, and $x \colon \Box^1 \to X$ an edge of $X$. Then $x$ is marked if and only if it factors through the inclusion of the middle edge $\Box^1 \to K$.
\end{lemma}

\begin{proof}
The inclusions $K \to K'$ and $(\Box^1)^\sharp \to K'$ are both anodyne (the latter as a composite of marked open box fillings). The stated result thus follows from the fact that $X \to \Box^0$ has the right lifting property with respect to both of these maps.
\end{proof}

\begin{lemma}\label{marked-cubical-quasicat-3-of-4}
For a marked cubical set $X$ to be a marked cubical quasicategory, it suffices for the map $X \to \Box^{0}$ to have the right lifting property with respect to marked open box fillings and the saturation map. 
\end{lemma}

\begin{proof}
Assume that $X$ has the right lifting property with respect to marked open box inclusions and the saturation map. The proof of \cref{equivalence-iff-marked} only requires lifting with respect to these maps, so the marked edges of $X$ are precisely those which factor through $K$. 

To show that $X \to \Box^0$ lifts against the 3-out-of-4 maps, we must show that, if three sides of a 2-cube in $X$ are marked, then so is the fourth. Using the fact that the three marked sides factor through $K$, we can show that the fourth does as well by an exercise in filling three-dimensional marked open boxes. We illustrate this argument for the case where the $(1,0)$-face is unmarked; the other three cases are similar.

Consider the following 2-cube in $X$:
\[
\xymatrix{
x \ar[r]^{p}_{\sim} \ar[d]_{f} & y \ar[d]^{g}_{\sim} \\
w \ar[r]^{q}_{\sim} & z \\
}
\]
To show that $f$ factors through $K$, we must construct a pair of 2-cubes as depicted below:
\[
\xymatrix{
w \ar[r]^{f^{-1}_{R}} \ar@{=}[d] & x \ar[d]^{f} \ar@{=}[r] & x \ar@{=}[d] \\
w \ar@{=}[r] & w \ar[r]^{f^{-1}_{L}} & x \\
}
\]
As we have shown that marked 1-cubes factor through $K$, we assume the existence of similar 2-cubes for $g, p,$ and $q$, with their left and right inverses denoted similarly.

We construct the left inverse $f_{L}^{-1}$ by marked open box filling, as depicted below.
\[
\xymatrix{
w \ar@{..>}[r]^{f_{L}^{-1}} \ar[d]_{q} & x \ar[d]^{p} \\
z \ar[r]^{g_{R}^{-1}} & y \\
}
\]

To obtain the 2-cube witnessing $f^{-1}_{L}$ as a left inverse for $f$,  we fill the following $(2,0)$-marked open box.
 \[
\xymatrix@!C{
 x
 \ar@{=}[rrr]
 \ar[ddd]_{p}
 \ar[dr]^{f}
&&&
 x
 \ar@{=}[ddd]|!{[dl];[dr]}{\hole}
 \ar@{=}[dr]
\\&
 w
 \ar[rrr]^{f^{-1}_{L}}
 \ar[ddd]_{q}
&&&
 x
 \ar[ddd]^{p}
\\
\\
 y
 \ar@{=}[rrr]|!{[uur];[dr]}{\hole}
 \ar[dr]_{g}
&&&
 y
 \ar@{=}[dr]
\\&
 z
 \ar[rrr]^{g^{-1}_{R}}
&&&
 y
}
\]

Similarly, we construct $f_{R}^{-1}$ by marked open box filling as follows.
\[
\xymatrix{
w \ar@{..>}[r]^{f^{-1}_{R}} \ar[d]_{q} & x \ar[d]^{p} \\
z \ar[r]^{g^{-1}_{R}} & y \\
}
\]
To obtain the 2-cube witnessing $f^{-1}_{R}$ as a right inverse for $f$,  we fill the following $(2,0)$-marked open box.

 
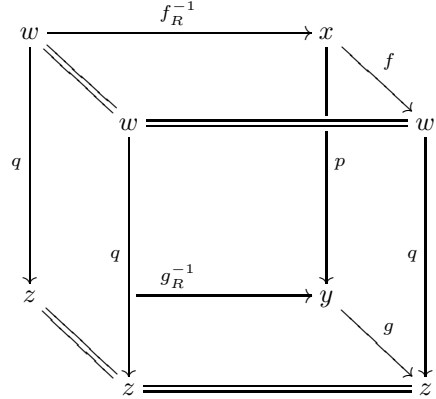
\[
\xymatrix@!C{
 w
 \ar[rrr]^{f_{R}^{-1}}
 \ar[ddd]_{q}
 \ar@{=}[dr]
&&&
 x
 \ar[ddd]|!{[dl];[dr]}{\hole}^{p}
 \ar[dr]^{f}
\\&
 w
 \ar@{=}[rrr]
 \ar[ddd]_{q}
&&&
 w
 \ar[ddd]_{q}
\\
\\
 z
 \ar[rrr]|!{[uur];[dr]}{\hole}^{g_{R}^{-1}}
 \ar@{=}[dr]
&&&
 y
 \ar[dr]^{g}
\\&
 z
 \ar@{=}[rrr]
&&&
 z
}
\]

Thus we see that $f$ factors through $K$, and is therefore marked.
\end{proof}

\begin{remark}
In view of \cref{marked-cubical-quasicat-3-of-4}, it is natural to wonder whether omitting the 3-out-of-4 maps as generators would change the class of anodyne maps. To see that it would, observe that, using the small object argument, we can factor any three-out-of-four map as a composite of a map in the saturation of the marked open box fillings and saturation map, followed by a map having the right lifting property with respect to these maps. Examining the details of this construction, we can see that the second of these maps will not have the right lifting property with respect to the 3-out-of-4 maps. Thus the 3-out-of-4 maps are not in the saturation of the other two classes of generating anodynes.

One may further note that, without the 3-out-of-4 maps as generators, anodyne maps would not be closed under pushout product with cofibrations, e.g., $\iota^1_{1,0} \hatotimes (\bd \Box^{1} \to (\Box^1)^{\sharp})$ is a $3$-out-of-$4$ map. 
This makes them crucial for our development.
\end{remark}

\begin{definition}
Given a map $f \colon X \to Y$ of marked cubical sets, a \emph{naive fibrant replacement} of $f$ consists of a diagram as depicted below, with $\overline{X}$ and $\overline{Y}$ marked cubical quasicategories, $\iota_{X}$ and $\iota_{Y}$ anodyne, and $\overline{f}$ a naive fibration.

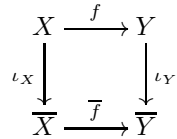
\centerline{
\xymatrix{
X \ar[r]^{f} \ar[d]_{\iota_{X}} & Y \ar[d]^{\iota_{Y}} \\
\overline{X} \ar[r]^{\overline{f}} & \overline{Y}
}
}
\end{definition}

We have a cofibrantly generated weak factorization system (anodyne maps, naive fibrations). 
This induces a functorial factorization of any map $X \to Y$ as
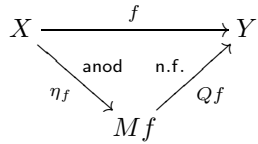
\[
\xymatrix{
  X
  \ar[rr]^{f}
  \ar[dr]_{\eta_f}^{\anod}
&&
  Y
\\&
  Mf
  \ar[ur]_{Qf}^{\naive}
}
\]
where $Q$ is an endofunctor on $(\cSet')^\to$ sending objects to naive fibrations and $\eta \co \Id \to Q$ is pointwise anodyne. Where $f$ is the unique map $X \to \Box^{0}$, we write $\eta_{X}$ for $\eta_{f}$. Given $f \colon X \to Y$, we can use this factorization to obtain a canonical naive fibrant replacement of $f$:
\[
\xymatrix@C+0.5cm{
  X
  \ar[r]^{f}
  \ar[d]_{\eta_{\eta_{Y}f}}
&
  Y
  \ar[d]^{\eta_Y}
\\
  \overline{X}
  \ar[r]^{Q(\eta_Y f)}
&
  \overline{Y}
\rlap{.}}
\]
We declare $f$ to be a \emph{weak equivalence} if $Q(\eta_Y f)$ is a trivial fibration. A \emph{trivial cofibration} is a map that is a cofibration and weak equivalence, and
a \emph{fibration} is a map that has the right lifting property against trivial cofibrations.

We now want to show that if $Y$ is a marked cubical quasicategory, so is $\ihom_L(X,Y)$. The following lemma on pushout-products helps with the proof of this fact.

\begin{lemma}\label{anodyne-pop}
The pushout product of two cofibrations is a cofibration. Furthermore, the pushout product of an anodyne map and a cofibration is anodyne.
\end{lemma}

\begin{proof}
Since $\otimes$ preserves colimits in each variable and anodynes are stable under pushouts and transfinite compositions, we can use induction on skeleta to show that if $S \to T$ is one of the generating cofibrations (resp. anodynes), then $(S \to T) \hat{\otimes} (\partial \Box^n \to \Box^n)$ and $(S \to T) \hat{\otimes} (\Box^1 \to (\Box^1)^\sharp)$ are cofibrations (resp. anodyne). This will show that if $i$ and $j$ are cofibrations, and $i$ is anodyne, then $i \hatotimes j$ is anodyne; the proof for the case where $j$ is anodyne is entirely analogous.

Several cases can be taken care of by the following fact: If $f \colon A \to B$ is an inclusion which is a bijection on vertices and $p \colon X \to Y$ is an isomorphism of underlying cubical sets, then $f \hatotimes p$ is an isomorphism. To see this, first observe that because the pushout product is an isomorphism of underlying cubical sets, we need only consider which edges are marked. That the sets of marked edges in $(A \otimes Y) \cup_{A \otimes X} (B \otimes X)$ and $B \otimes Y$ coincide follows from \cref{geometric-product-marked-pushout}, together with the fact that $f$ is a bijection on vertices.

This claim, along with the fact that taking the pushout product with $\varnothing \to \Box^0$ is the identity, handles all but the following pushout products:

\begin{itemize}
\item $(\bd \Box^{m} \to \Box^{m}) \hatotimes (\bd \Box^n \to \Box^n)$: this is the map $\bd \Box^{m+n} \to \Box^{m+n}$. This completes the proof of the first statement, concerning the pushout product of two cofibrations; the remaining cases complete the second statement, concerning the pushout product of a cofibration and an anodyne map.

\item $\iota^m_{i,\varepsilon} \hatotimes (\partial \Box^n \to \Box^n)$: the underlying cubical set map is the open box inclusion $\sqcap^{m+n}_{i,\varepsilon} \to \Box^{m+n}$, with edges in the codomain being marked if and only if they are present and marked in the domain. The critical edge is marked, so this is anodyne as a pushout of a marked open box inclusion.

\item $\iota^1_{i,\varepsilon} \hatotimes (\bd \Box^{1} \to (\Box^1)^{\sharp})$: this is the 3-out-of-4 map associated to the  face $(1,1-\varepsilon)$. 
\end{itemize}
\end{proof}

\begin{corollary}\label{fibrantExps}  
If $f \colon A \to B$ is a cofibration and $g \colon X \to Y$ is a naive fibration, then the pullback exponential $f \triangleright g \colon \ihom(A,Y) \to \ihom(A,X) \times_{\ihom(A,Y)} \ihom(B,Y)$ (where $\ihom$ may designate either $\ihom_{L}$ or $\ihom_{R}$) is a naive fibration. Furthermore, if $f$ is anodyne or $g$ is a trivial fibration, then $f \triangleright g$ is a trivial fibration. 

In particular, if $Y$ is a marked cubical quasicategory, then for any $X$, $\ihom(X,Y)$ is a marked cubical quasicategory.
\end{corollary}

\begin{proof}
Let $i \colon C \rightarrow D$ be anodyne; we wish to show that $f \triangleright g$ has the right lifting property with respect to $i$. By a standard duality, it suffices to show that $g$ has the right lifting property with respect to $i \hat{\otimes} f$. This map is anodyne by \cref{anodyne-pop}, so the first statement holds. 

For the second statement, we can apply the same result with $i$ an arbitrary cofibration. Then $g$ has the right lifting property with respect to $i \hatotimes f$, either because $f$, and hence also $i \hatotimes f$, are anodyne, or because $i \hatotimes f$ is a cofibration and $g$ is a trivial fibration.

The third statement follows from the first by the fact that $\ihom(X,Y) \to \Box^{0}$ is the pullback exponential of the cofibration $\varnothing \to X$ with the naive fibration $Y \to \Box^{0}$.
\end{proof}

\subsection{Homotopies}\label{section:marked-htpy}

Next we define the closely-related concepts of connected components in a marked cubical set, and homotopies of maps between cubical sets.

\begin{definition}
For a marked cubical set $X$, let $\sim_{0}$ denote the relation on $X_{0}$, the set of vertices of $X$, given by $x \sim_{0} y$ if there is a marked edge from $x$ to $y$ in $X$. Let $\sim$ denote the smallest equivalence relation on $X_{0}$ containing $\sim_{0}$.
\end{definition}

\begin{remark}
For $x, y \in X_{0}$, one can easily see that $x \sim y$ if and only if $x$ and $y$ are connected by a zigzag of marked edges. 
\end{remark}

\begin{definition}
For a marked cubical set $X$, the \emph{set of connected components} $\pi_0(X)$ is $X_0/\sim$.
\end{definition}

We may observe that the construction of $\pi_{0}(X)$ is functorial, since maps of marked cubical sets preserve marked edges, and hence preserve the equivalence relation $\sim$.

\begin{definition} 
An \emph{elementary left homotopy} $h \co f \sim g$ between maps $f, g \co A \to B$ is a map $h \co (\Box^1)^\sharp \otimes A \to B$ such that $h|_{\braces{0} \otimes A} = f$ and $h|_{\braces{1} \otimes A} = g$.
Note that the elementary left homotopy $h$ corresponds to an edge $(\Box^1)^\sharp \to \ihom_L(A, B)$ between the vertices corresponding to $f$ and $g$.
A \emph{left homotopy} between $f$ and $g$ is a zig-zag of elementary left homotopies.
\end{definition}

A left homotopy from $f$ to $g$ corresponds to a zig-zag of marked edges in $\ihom_L(A, B)$ and so maps from $A$ to $B$ are left homotopic exactly if $f$ and $g$ are in the same connected component of $\ihom_L(A, B)$.
We write $[A, B]$ for the set of left homotopy classes of maps $A \to B$.

These induce notions of \emph{elementary left homotopy equivalence} and \emph{left homotopy equivalence}. Each of these notions has a ``right'' variant using $A \otimes (\Box^1)^\sharp$ and $\ihom_R(A,B)$. Unless the potential for confusion arises or a statement depends on the choice, we will drop the use of ``left'' and ``right''.

\begin{lemma}\label{EqRel}
In a marked cubical quasicategory $X$, the relations $\sim_{0}$ and $\sim$ coincide.
\end{lemma}

\begin{proof}
Using 2-dimensional open box fillers with certain edges degenerate, and the 3-out-of-4 property, we can reduce any zigzag of marked edges connecting $x$ and $y$ in $X$ to a single marked edge from $x$ to $y$.
\end{proof}

By adjointness, we obtain the following corollary.

\begin{corollary}
If $f, g \co A \to B$ are homotopic and $B$ a marked cubical quasicategory, then $f$ and $g$ are elementarily homotopic.
Hence, between marked cubical quasicategories homotopy equivalences coincide with elementary homotopy equivalences.
\end{corollary}

\begin{proof}
By \cref{fibrantExps}, $\ihom(A,B)$ is a marked cubical quasicategory, and so $\sim_{0}$ is an equivalence relation on $\ihom(A,B)_0$ by \cref{EqRel}. Translating what this means for homotopies gives the result.
\end{proof}

\begin{lemma}\label{precomp-htpy}
If $f, g \colon X \to Y$ are left homotopic, then for any $Z$, then the induced maps $\ihom_{L}(Y,Z) \to \ihom_{L}(X,Z)$ are right homotopic.
\end{lemma}

\begin{proof}
We consider the case of elementary homotopies; the general result follows from this. An elementary left homotopy $f \sim g$ is given by a map $H \colon (\Box^{1})^\sharp \otimes X \to Y$. Pre-composition with $H$ induces a map $\ihom_{L}(Y,Z) \to \ihom_{L}((\Box^1)^\sharp \otimes X,Z)$. Under the adjunction defining $\ihom_{L}$, this corresponds to a map $\ihom_{L}(Y,Z) \otimes (\Box^1)^\sharp \otimes X \to Z$, which in turn corresponds to a map $\ihom_{L}(Y,Z) \otimes (\Box^1)^\sharp \to \ihom_{L}(X,Z)$. This defines an elementary right homotopy between the pre-composition maps induced by $f$ and $g$.
\end{proof}

\subsection{Category theory in a marked cubical quasicategory}\label{sec:marked-cubical-quasicat-theory}

Let $X$ be a marked cubical quasicategory and $x,y \in X_0$.
We will write $X_1(x,y)$ for the subset of $X_1$ consisting of $1$-cubes $f$ with $f \partial_{1,0}= x$ and $f\partial_{1,1} = y$.
Define an equivalence relation relation $\sim_X$ on the set $X_1(x, y)$ of edges from $x$ to $y$ as follows: $f \sim_X g$ if and only if there is a $2$-cube in $X$ of the form 
  \[
\xymatrix{
  x
  \ar[r]^f
  \ar@{=}[d]
&
  y
  \ar@{=}[d]
\\
  x
  \ar[r]^g
&
  y
}
\]
It is straightforward to verify that this is indeed an equivalence relation: reflexivity follows from degeneracies, whereas symmetry and transitivity are given by filling $3$-dimensional open boxes.

We now define three increasingly strong refinements of the concept of a homotopy equivalence.

\begin{definition}
Let $f \colon X \to Y$ be a map in $\cSet$. Then:

\begin{itemize}
\item $f$ is a \emph{semi-adjoint equivalence} if there exist $g \colon Y \to X$ and homotopies $H \colon gf \sim \mathrm{id}_{X}$, $K \colon fg \sim \mathrm{id}_{Y}$ such that $fH \sim Kf$ as edges of $\ihom(X,Y)$;
\item $f$ is a \emph{strong homotopy equivalence} if there exist $g, H, K$ as above with $fH = Kf$;
\item a map $g \colon Y \to X$ is a \emph{strong deformation section} of $f$ if $fg = \mbox{id}_{Y}$ and there exists a homotopy $H \colon gf \sim \mathrm{id}_{X}$  such that $fH = \mathrm{id}_{f}$.
\end{itemize}
\end{definition}

Our next goal will be two show the following:

\begin{lemma}\label{htpy-equiv-iff-strong}
Let $f \colon X \to Y$ be a map of marked cubical quasicategories. The following are equivalent:

\begin{enumerate}
\item $f$ is a homotopy equivalence;
\item $f$ is a semi-adjoint equivalence.
\end{enumerate}

Furthermore, if $f$ is a naive fibration, then these are equivalent to:
\begin{enumerate}
\setcounter{enumi}{2}
\item $f$ is a strong homotopy equivalence.
\end{enumerate}
\end{lemma}

We will prove this by means of a 2-categorical argument.

We define the \emph{homotopy category} $\Ho X$ of a marked cubical quasicategory $X$ as follows:
\begin{itemize}
  \item the objects of $\Ho X$ are the $0$-cubes of $X$;
  \item the morphisms from $x$ to $y$ in $\Ho X$ are the equivalence classes of edges $X_1(x,y)/\sim_X$;
  \item the identity map on $x \in X_0$ is given by $x \sigma_1$;
  \item the composition of $f \colon x \to y$ and $g \colon y \to z$ is given by filling the open box
  
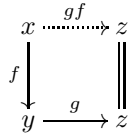
\[
\xymatrix{
  x
  \ar@{..>}[r]^{gf}
  \ar[d]_{f}
&
  z
  \ar@{=}[d]
\\
  y
  \ar[r]^{g}
&
  z
}
\]
\end{itemize}


\begin{lemma}
  The above data define a category. 
\end{lemma}

\begin{proof}
We must show that composition is well-defined, associative, and unital with the given identities.

To see that it is well-defined, suppose that $f \sim f', g \sim g', h \sim h'$, and $gf = h$. Then we can construct the following $(3,1)$-open box:

 
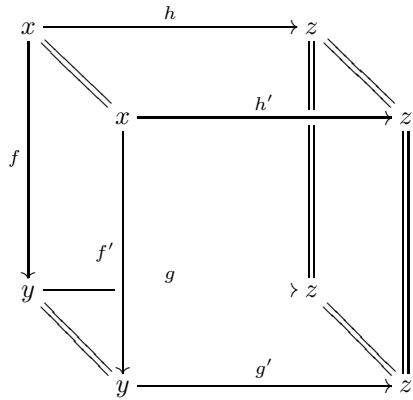
\[
\xymatrix@!C{
 x
 \ar[rrr]^{h}
 \ar[ddd]_{f}
 \ar@{=}[dr]
&&&
 z
 \ar@{=}[ddd]|!{[dl];[dr]}{\hole}
 \ar@{=}[dr]
\\&
 x
 \ar[rrr]^{h'}
 \ar[ddd]_{f'}
&&&
 z
 \ar@{=}[ddd]
\\
\\
 y
 \ar[rrr]|!{[uur];[dr]}{\hole}^{g}
 \ar@{=}[dr]
&&&
 z
 \ar@{=}[dr]
\\&
 y
 \ar[rrr]^{g'}
&&&
 z
}
\]

As the critical edge is degenerate, this open box admits a filler; the $(3,1)$-face of this filler witnesses $g'f' = h'$.

To see that composition is associative, consider a composable triple of edges $f, g, h$. We can construct the following $(3,0)$-open box:

 
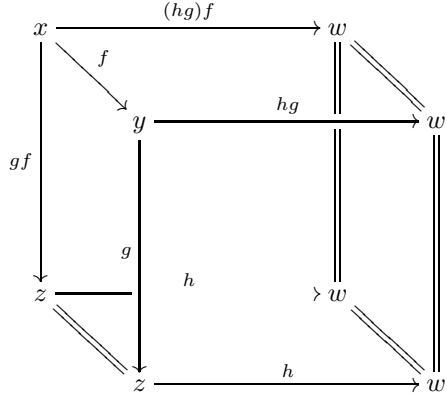
\[
\xymatrix@!C{
 x
 \ar[rrr]^{(hg)f}
 \ar[ddd]_{gf}
 \ar[dr]^{f}
&&&
 w
 \ar@{=}[ddd]|!{[dl];[dr]}{\hole}
 \ar@{=}[dr]
\\&
 y
 \ar[rrr]^{hg}
 \ar[ddd]_{g}
&&&
 w
 \ar@{=}[ddd]
\\
\\
 z
 \ar[rrr]|!{[uur];[dr]}{\hole}^{h}
 \ar@{=}[dr]
&&&
 w
 \ar@{=}[dr]
\\&
 z
 \ar[rrr]^{h}
&&&
 w
}
\]

This open box admits a filler, since the critical edge is degenerate; the $(3,0)$-face of this filler witnesses $h(gf) = (hg)f$.

Finally, for an edge $f$ from $x$ to $y$, the equalities $(\sigma_1 y)f = f$ and $f(\sigma_1 x) = f$ are witnessed by the 2-cubes $f \gamma_{1,0}$ (or the substitute constructed in \cref{cqcats-have-connections}) and $f \sigma_2$, respectively.
\end{proof}

\begin{lemma} \label{either-dir-comp}
Let $X$ be a marked cubical quasicategory.  There is a $2$-cube of the form
 \[
\xymatrix{
  x
  \ar[r]^f
  \ar[d]_p
&
  y
  \ar[d]^g
\\
  z
  \ar[r]^q
&
  w
}
\]
if and only if $gf = qp$ in $\Ho X$.
\end{lemma}

\begin{proof}
Consider the following 3-cube:

 
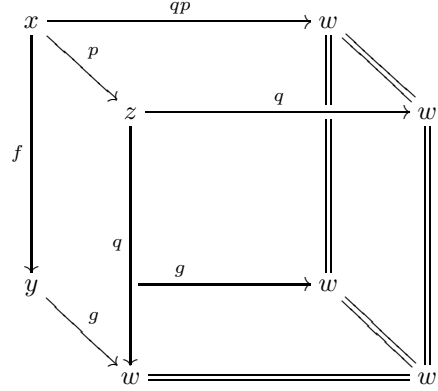
\[
\xymatrix@!C{
 x
 \ar[rrr]^{qp}
 \ar[ddd]_{f}
 \ar[dr]^{p}
&&&
 w
 \ar@{=}[ddd]|!{[dl];[dr]}{\hole}
 \ar@{=}[dr]
\\&
 z
 \ar[rrr]^{q}
 \ar[ddd]_{q}
&&&
 w
 \ar@{=}[ddd]
\\
\\
 y
 \ar[rrr]|!{[uur];[dr]}{\hole}^{g}
 \ar[dr]^{g}
&&&
 w
 \ar@{=}[dr]
\\&
 w
 \ar@{=}[rrr]
&&&
 w
}
\]

The equality $gf = qp$ in $\Ho X$ is equivalent to the existence of a filler for the back face of this cube, using the fact that composition in $\Ho X$ is well-defined. Thus we want to show that there is a filler for the back face if and only if there is a filler for the left face. If we assume that either of these 2-cubes exists, then together with the remaining faces of the cube depicted above, it forms a marked open box in $X$, with critical edge $w \sigma_{1}$. Thus we can fill this open box to obtain a filler for the missing face.
\end{proof}

\begin{remark}
The above argument makes use of negative connections, as two of the faces of the 3-cube used in its proof are negative connections of 1-cubes. However, it can still be adapted to the categories $\cSet_{\varnothing}$ and $\cSet_1$ using \cref{cqcats-have-connections}. 
\end{remark}

\begin{lemma}\label{tau-ho-equiv}
  Let $X$ be the underlying cubical set of a marked cubical quasicategory $X'$. The categories $\Ho X'$ and $\tau_1 X$ are equivalent.
\end{lemma}

\begin{proof}
  There is a natural inclusion $\Ho X' \to \tau_1 X$, which is the identity on objects and takes a $1$-cube $f$ to a string of length $1$ consisting of $f$.
This is clearly faithful and essentially surjective.
To see that it is full, we simply fill in $2$-dimensional open boxes with one degenerate edge to reduce a sequence of arbitrary length to a sequence of length $1$.
\end{proof}

The assignment $X \mapsto \Ho X$ extends to a functor taking a marked cubical quasicategory to its homotopy category. Postcomposing this functor with $\core \colon \Cat \to \Gpd$, we obtain a groupoid $\Hosharp X$.

\begin{lemma}
The groupoid $\Hosharp X$ can be constructed directly as follows:

\begin{itemize}
\item Objects are 0-cubes of $X$;
\item Morphisms from $x$ to $y$ are equivalence classes of marked edges from $x$ to $y$;
\item Composition and identities are defined as in $\Ho X$.
\end{itemize}
\end{lemma}

\begin{proof}
Let $X$ be a marked cubical quasicategory. By definition, an edge $f \colon \Box^{1} \to X$ is invertible in $\Ho X$ if and only if it factors through the map $\Box^{1} \to K$ which picks out the middle edge. Since the inclusions $(\Box^{1})^{\sharp} \to K'$ and $K \to K'$ are anodyne, this holds if and only if $f$ is marked.
\end{proof}

\begin{definition}
Define a strict 2-category $\Ho_2 \cSet'$ whose objects are the marked cubical quasicategories and whose mapping category from $X$ to $Y$ is 
\[\Ho_2 \cSet'(X,Y) := \Ho \ihom_L(X,Y)\text{.}\]
This means the 1-morphisms are the usual 1-morphisms $X \to Y$, and the 2-morphisms are maps $X \otimes \Box^1 \to Y$, modulo an equivalence relation. Denote the (vertical) composition in $\Ho \ihom_L(X,Y)$ with $\circ$. The (horizontal) composition

$$\Ho\ihom_L(Y,Z) \times \Ho\ihom_L(X,Y) \to \Ho\ihom_L(X,Z)$$ (which will be written by concatenation) is defined on objects by the usual composition. If $H \colon Y \otimes \Box^1 \to Z$ and $K \colon X \otimes \Box^1 \to Y$ are morphisms $K\colon g \to g'$ and $H\colon  f \to f'$, respectively, define the morphism $KH\colon  gf \to g'f'$ by choosing a filler for the open box of $\ihom_L(X,Z)$ depicted by \[
\xymatrix{
  gf
  \ar[r]^{Kf}
  \ar@{=}[d]
&
  g'f
  \ar[d]^{g'H}
\\
  gf
  \ar@{..>}[r]^{KH}
&
  g'f'
}
\] where the top edge is induced by the composite $X \otimes \Box^1 \to Y \otimes \Box^1 \to Z$ and the right edge by $X \otimes \Box^1 \to Y \to Z$. The fact that the $\ihom_L(X,Y)$ are marked cubical quasicategories ensures this defines a well-defined, associative, unital, and functorial operation. For functoriality, note that the morphism $X \otimes \Box^1 \otimes \Box^1 \stackrel{H \otimes \Box^1}{\to} Y \otimes \Box^1 \stackrel{K}{\to} Z$ yields a 2-cube $\Box^2 \to \ihom_L(X,Z)$ which can be depicted as  \[
\xymatrix{
  gf
  \ar[r]^{Kf}
  \ar[d]_{gH}
&
  g'f
  \ar[d]^{g'H}
\\
  gf'
  \ar[r]^{Kf'}
&
  g'f'
}
\] and so by \cref{either-dir-comp}, we have $(g'H)\circ(Kf) = (Kf')\circ(gH)$, which implies the interchange law.
\end{definition}

\begin{definition}
Let $\Hosharp_2 \cSet'$ denote the maximal $(2,1)$-category contained in $\Ho_{2} \cSet'$, i.e. the 2-category whose objects are marked cubical sets, with $\Hosharp_2 \cSet'(X,Y) = \Hosharp\ihom_{L}(X,Y)$, and the 2-categorical operations induced by those of $\Ho_2 \cSet'$.
\end{definition}

The $\Hosharp$ construction, together with the following general results about $(2,1)$-categories, give us
the desired result about compatibility of homotopies.

\begin{lemma}[Undergraduate Lemma] \label{iso-comm}
Let $X$ be an object in a $(2,1)$-category $\mathsf{C}$, and let $H\colon  p \sim \mbox{id}_X$ be a morphism in $\mathsf{C}(X,X)$. Then $pH = Hp$.
\end{lemma}

\begin{proof}
By the interchange law, $$H \circ (pH) = (H \mbox{id}_X) \circ (pH) = (\mbox{id}_X H) \circ (Hp) = H \circ (Hp).$$ Since $\mathsf{C}(X,X)$ is a groupoid, we can cancel $H$.
\end{proof}

\begin{lemma}[Graduate Lemma] \label{grad-lemma}
Let $X, Y$ be objects in a $(2,1)$-category $\mathsf{C}$, $f\colon X \leftrightarrows Y: g$ two morphisms between them, and $H\colon  gf \to \mbox{id}_X$ and $K\colon  fg \to \mbox{id}_Y$ two 2-cells. Then there is a 2-cell $K'\colon  fg \to \mbox{id}_Y$ for which $K'f = fH$.
\end{lemma}

\begin{proof}
Define $K' := K \circ (fHg) \circ (Kfg)^{-1}.$ Now, we compute:

\begin{align*}
K'f &= Kf \circ (fHgf) \circ (Kfgf)^{-1} & \\
&= Kf \circ (fgfH) \circ (Kfgf)^{-1} & \mbox{ (by \ref{iso-comm})}  \\
&= fH & \mbox{ (by naturality/interchange)} 
\end{align*}
\end{proof}

\begin{proof}[Proof of \cref{htpy-equiv-iff-strong}]
The implications $(iii) \Rightarrow (ii) \Rightarrow (i)$ are clear. The implication $(i) \Rightarrow (ii)$ follows from applying \cref{grad-lemma} to the $(2,1)$-category $\Hosharp_{2}\cSet'$.

Now let $f$ be a naive fibration and a semi-adjoint equivalence. By \cref{fibrantExps}, the map $\ihom(X,X) \to \ihom(X,Y)$ is a naive fibration. A simple exercise in 2-dimensional marked open box filling, using this fact and the definition of a semi-adjoint equivalence, shows that there exists a homotopy $H' \colon gf \sim \mathrm{id}_{X}$ such that $fH' = Kf$.
\end{proof}

\subsection{Fibration category of marked cubical quasicategories}

\begin{lemma}\label{anodyne-htpy-equiv}
Every anodyne map between marked cubical quasicategories is a homotopy equivalence.
\end{lemma}

\begin{proof}
Now let $f \colon X \to Y$ be anodyne, with $X$ and $Y$ marked cubical quasicategories. 
We can obtain a retraction $r \colon Y \to X$ as a lift in the following diagram:

\centerline{
\xymatrix{
X \ar[d]_{f} \ar@{=}[r] & X \ar[d] \\
Y \ar[r] & \Box^{0} \\
}
}

We can then obtain a left homotopy $fr \sim \mathrm{id}_{Y}$ as a lift in the following diagram:

\centerline{
\xymatrix{
(\bd \Box^{1} \otimes Y) \cup ((\Box^{1})^{\sharp} \otimes X) \ar[d] \ar[rr]^-{[[fr, \mathrm{id}_{Y}],f\pi_{1}]} & & Y \ar[d] \\
(\Box^{1})^{\sharp} \otimes Y \ar[rr] & & \Box^{0} \\ 
}
}

The lift exists since the left-hand map is anodyne by \cref{anodyne-pop}. 

An analogous proof shows that $f$ is a right homotopy equivalence.
\end{proof}

\begin{lemma}\label{tfib-iff-retract}
Let $f \colon X \to Y$ be a naive fibration. The following are equivalent:

\begin{enumerate}
\item\label{tfib-iff-retract-tfib} $f$ is a trivial fibration;
\item\label{tfib-iff-retract-sds} $f$ has a strong deformation section;
\item\label{tfib-iff-retract-streq} $f$ is a strong homotopy equivalence.
\end{enumerate}
\end{lemma}

\begin{proof}
If $f \colon X \to Y$ is a trivial fibration, then we can obtain a section $g \colon Y \to X$ as a lift of the following diagram:

\centerline{
\xymatrix{
\varnothing \ar[r] \ar[d] & X \ar[d]^{f} \\
Y \ar@{=}[r] & Y \\
}
}

We can then obtain a left homotopy $H \colon gf \sim \mathrm{id}_{X}$ satisfying $fH = \mathrm{id}_{f}$ as a lift in the following diagram:

\centerline{
\xymatrix{
X \sqcup X \ar@{>->}[d] \ar[rr]^-{[sf,\mathrm{id}_{X}]} & & X \ar[d]^{f} \\
(\Box^{1})^{\sharp} \otimes X \ar[rr]^-{f \pi_{X}} & & Y \\
}
}

This shows \ref{tfib-iff-retract-tfib} $\Rightarrow$ \ref{tfib-iff-retract-sds}, and the implication \ref{tfib-iff-retract-sds} $\Rightarrow$ \ref{tfib-iff-retract-streq} is trivial. 
To show that \ref{tfib-iff-retract-streq} $\Rightarrow$ \ref{tfib-iff-retract-tfib}, we first show that \ref{tfib-iff-retract-streq} implies the following condition:
\begin{itemize}
\item[(iii)'] the canonical map $\iota^1_{1,0} \triangleright f \to f$ in $(\cSet')^\to$ admits a section.
\end{itemize}

To see \ref{tfib-iff-retract-streq} $\Rightarrow$ (iii)', suppose $f$ is a strong homotopy equivalence with homotopy inverse $g \colon Y \to X$ and homotopies $H \colon gf \sim \mathrm{id}_{X}, K \colon fg \sim \mathrm{id}_{Y}$ satisfying $fH = Kf$. Then we have the following commuting diagram in $\cSet'$:

\centerline{
\xymatrix{
X \ar[r] \ar[d]_{f} & \ihom((\Box^1)^\sharp,X) \ar[d]_{\iota^1_{1,0} \triangleright f} \ar[r] & X \ar[d]^f \\
Y \ar[r] & X \times_{Y} \ihom((\Box^1)^\sharp,Y) \ar[r] & Y \\
}
}

The top-left map is the adjunct of $H$, while the bottom-left map is induced by $g$ and the adjunct of $K$; the right-hand square is as in the statement of (iii)', and hence the composite square is simply the identity square on $f$.

Finally, note that $\iota^{1}_{1,1} \triangleright f$ is a trivial fibration by \cref{fibrantExps}. Therefore, if the square given in the statement of (iii)' has a section, then $f$ is a trivial fibration as a retract of a trivial fibration. Thus (iii)' $\Rightarrow$ \ref{tfib-iff-retract-tfib}.
\end{proof}

\begin{corollary} \label{nfib+weq=trivFib}
  A map $f \colon X \to Y$ between marked cubical quasicategories is a trivial fibration exactly if it is a homotopy equivalence and a naive fibration.
\end{corollary}

\begin{proof}
This follows from \cref{htpy-equiv-iff-strong,tfib-iff-retract}, together with the fact that every trivial fibration is a naive fibration since all anodyne maps are cofibrations.
\end{proof}

\begin{proposition}
The category of marked cubical quasicategories forms a fibration category, with naive fibrations as the fibrations and homotopy equivalences as the weak equivalences.
\end{proposition}

\begin{proof}
The class of homotopy equivalences is closed under 2-out-of-3. \cref{nfib+weq=trivFib} shows that the maps between marked cubical quasicategories which are naive fibrations and homotopy equivalences are exactly the trivial fibrations; both fibrations and trivial fibrations are defined via a right lifting property, and hence they are stable under pullback. By \cref{anodyne-htpy-equiv}, each anodyne map between marked cubical quasicategories is a homotopy equivalence, and so the (anodyne, naive fibration)-factorization gives the factorization axiom.
\end{proof}

\begin{lemma}\label{we-marked-fibrant-objects}
  Let $f \colon X \to Y$ be a map between marked cubical quasicategories. Then the following conditions are equivalent:
  \begin{enumerate}
    \item\label{we-marked-fibrant-objects-we} $f$ is a weak equivalence;
    \item\label{we-marked-fibrant-objects-left} $f$ is a left homotopy equivalence;
    \item\label{we-marked-fibrant-objects-right} $f$ is a right homotopy equivalence.
  \end{enumerate}
\end{lemma}

\begin{proof}
Consider the canonical naive fibrant replacement of $f$ used in the definition of the weak equivalences:
  \[
\xymatrix{
  X
  \ar[r]^f
  \ar[d]_{\iota_X}
&
  Y
  \ar[d]^{\iota_Y}
\\
  \overline{X}
  \ar[r]^{\overline{f}}
&
  \overline{Y}
}
\]
(here $\iota_{Y} = \eta_{Y}, \overline{f} = Q(\eta_{Y} f), \iota_{X} = \eta_{\eta_{Y}f}$).

By \cref{anodyne-htpy-equiv}, $\iota_{X}$ and $\iota_{Y}$ are left homotopy equivalences. Since left homotopy equivalences satisfy the two-out-of-three property, $f$ is a left homotopy equivalence if and only if $\overline{f}$ is one. By \cref{nfib+weq=trivFib}, $\overline{f}$ is a left homotopy equivalence if and only if it is a trivial fibration, i.e. if and only if $f$ is a weak equivalence. So \ref{we-marked-fibrant-objects-we} $\Leftrightarrow$ \ref{we-marked-fibrant-objects-left}; an analogous argument shows \ref{we-marked-fibrant-objects-we} $\Leftrightarrow$ \ref{we-marked-fibrant-objects-right}.
\end{proof}

\subsection{Cofibration category of marked cubical sets}

Our next result shows that the definition of the weak equivalences is not sensitive to the choice of naive fibrant replacement.

\begin{lemma}\label{we-independent-factorization}
Let $f \colon X \to Y$ be a map of marked cubical sets. The following are equivalent:
\begin{enumerate}
  \item\label{we-independent-factorization-we} $f$ is a weak equivalence.
  \item\label{we-independent-factorization-exists} there exists a naive fibrant replacement of $f$ by a trivial fibration;
  \item\label{we-independent-factorization-any} any naive fibrant replacement of $f$ is a trivial fibration.
\end{enumerate}
\end{lemma}

\begin{proof}
The implications \ref{we-independent-factorization-we} $\Rightarrow$ \ref{we-independent-factorization-exists} and \ref{we-independent-factorization-any} $\Rightarrow$ \ref{we-independent-factorization-we} are immediate from the definition of the weak equivalences. To prove \ref{we-independent-factorization-exists} $\Rightarrow$ \ref{we-independent-factorization-any}, consider a map $f \colon X \to Y$ having a naive fibrant replacement by a trivial fibration $\overline{f} \colon \overline{X} \to \overline{Y}$, and an arbitrary naive fibrant replacement $\overline{f}' \colon \overline{X}' \to \overline{Y}'$ of $f$. As depicted below, let $\overline{f}'' \colon \overline{X}'' \to \overline{Y}''$ be a naive fibrant replacement of the induced map between the pushouts $\overline{X} \cup_{X} \overline{X}' \to \overline{Y} \cup_{Y} \overline{Y}'$.

 \[
\xymatrix@!C{
 X
 \ar[rr]
 \ar[dd]
 \ar[dr]^{f}
&&
 \overline{X}'
 \ar[dd]|!{[dl];[dr]}{\hole}
 \ar[dr]^{\overline{f}'}
\\&
 Y
 \ar[rr]
 \ar[dd]
&&
 \overline{Y}'
 \ar[dd]
\\
 \overline{X}
 \ar[rr]|!{[uur];[dr]}{\hole}
 \ar[dr]^{\overline{f}}
&&
 \overline{X} \cup_{X} \overline{X}'
 \ar[dr]
 \ar[rr]|{\hole}
 &&
 \overline{X}''
 \ar[dr]^{\overline{f}''}
\\&
 \overline{Y}
 \ar[rr]
&&
 \overline{Y} \cup_{Y} \overline{Y}'
 \ar[rr]
 &&
 \overline{Y}''
}
\]

The maps $\overline{X} \to \overline{X}'', \overline{Y} \to \overline{Y}'', \overline{X}' \to \overline{X}'', \overline{Y}' \to \overline{Y}''$ are anodyne, as anodyne maps are closed under pushout and composition. Furthermore, $\overline{f}$ is a trivial fibration by assumption. Thus all of these maps are homotopy equivalences by \cref{anodyne-htpy-equiv} and \cref{nfib+weq=trivFib}. So we can apply the two-out-of-three property to see that $\overline{f}''$ is a homotopy equivalence; applying it again, we see that $\overline{f}'$ is a homotopy equivalence. Thus $\overline{f}'$ is a trivial fibration by \cref{nfib+weq=trivFib}. Since $\overline{f}'$ was arbitrary, we have shown that $f$ satisfies \ref{we-independent-factorization-any}.
\end{proof}

\begin{corollary}\label{anodyne-we}
Every anodyne map is a weak equivalence.
\end{corollary}

\begin{proof}
Let $f \colon X \to Y$ be anodyne. The following diagram gives a naive fibrant replacement of $f$:

\centerline{
\xymatrix{
X \ar[r]^{f} \ar[d]_{\eta_{Y} f} & Y \ar[d]^{\eta_{Y}} \\
\overline{Y} \ar@{=}[r] & \overline{Y} \\
}
}

Since $\mathrm{id}_{\overline{Y}}$ is a trivial fibration, $f$ is a weak equivalence by \cref{we-independent-factorization}.
\end{proof}

\begin{proposition} \label{we-characterizations}
The following are equivalent for a marked cubical map $A \to B$:
\begin{enumerate}
\item\label{we-characterizations-we} $A \to B$ is a weak equivalence;
\item\label{we-characterizations-htpy-eq} for any marked cubical quasicategory $X$, the induced map $\ihom(B, X) \to \ihom(A, X)$ is a homotopy equivalence;
\item\label{we-characterizations-pi} for any marked cubical quasicategory $X$, the induced map $\pi_0(\ihom(B, X)) \to \pi_0(\ihom(A, X))$ is a bijection.
\end{enumerate}
\end{proposition}

\begin{proof}
  First, suppose that $A \to B$ is a weak equivalence. Thus, there is a square
  \[
\xymatrix{
  A
  \ar[r]
  \ar[d]
&
  B
  \ar[d]
\\
  \overline{A}
  \ar[r]
&
  \overline{B}
}
\]
with $A \to \overline{A}$ and $B \to \overline{B}$ anodyne, and $\overline{A} \to \overline{B}$ a trivial fibration.
By \cref{nfib+weq=trivFib}, $\overline{A} \to \overline{B}$ is a left homotopy equivalence.

Applying $\ihom_L(-, X)$ to the diagram above, we obtain a diagram in which all objects are marked cubical quasicategories by \cref{fibrantExps}:
  \[
\xymatrix{
  \ihom_L(A, X)
&
  \ihom_L(B, X)
  \ar[l]
\\
  \ihom_L(\overline{A}, X)
  \ar[u]
&
  \ihom_L(\overline{B}, X)
   \ar[l]
  \ar[u]
}
\]
The vertical maps are trivial fibrations by \cref{fibrantExps}, hence homotopy equivalences by \cref{nfib+weq=trivFib}.
By \cref{precomp-htpy}, the bottom horizontal map is a right homotopy equivalence, since $\overline{A} \to \overline{B}$ is a left homotopy equivalence. Hence so is the upper horizontal map by 2-out-of-3.
Thus we have proven \ref{we-characterizations-we} $\Rightarrow$ \ref{we-characterizations-htpy-eq}.

The implication \ref{we-characterizations-htpy-eq} $\Rightarrow$ \ref{we-characterizations-pi} is clear, so it remains to show \ref{we-characterizations-pi} $\Rightarrow$ \ref{we-characterizations-we}. For that, we first observe that it suffices to consider $A$ and $B$ marked cubical quasicategories. To see this, consider the canonical naive fibrant replacement $\overline{f} \colon \overline{A} \to \overline{B}$ of a map $f \colon A \to B$. By definition, $f$ is a weak equivalence if and only if $\overline{f}$ is a trivial fibration; by \cref{nfib+weq=trivFib} and \cref{we-marked-fibrant-objects}, this holds if and only if $\overline{f}$ is a weak equivalence. Furthermore, the anodyne maps $\iota_{X}, \iota_{Y}$ are weak equivalences by \cref{anodyne-we}, and therefore satisfy \ref{we-characterizations-pi}; hence $f$ satisfies \ref{we-characterizations-pi} if and only if $\overline{f}$ does, by the 2-out-of-3 property for bijections.

Hence we can assume $A$ and $B$ are marked cubical quasicategories. Now take $X := A$ and set $g := (\pi_0f^*)^{-1}[\id_A]$. The verification that a representative of the class $g \in \pi_0 \ihom_L(B, A)$ defines a homotopy inverse of $f$ is straightforward; thus $f$ is a weak equivalence by \cref{we-marked-fibrant-objects}.
\end{proof}

\begin{corollary}\label{we-2-of-3}
The weak equivalences satisfy the 2-out-of-6 property (and hence the 2-out-of-3 property).
\end{corollary}

\begin{proof}
This is immediate from condition $(iii)$ of \cref{we-characterizations}.
\end{proof}

\begin{corollary}\label{KendEqMarked}
The endpoint inclusions $\Box^0 \to K$ are trivial cofibrations.
\end{corollary}

\begin{proof}
The maps in question are clearly cofibrations. To see that they are weak equivalences, consider the following commuting diagram:

\centerline{
\xymatrix{
\Box^0 \ar[r] \ar[d] & K \ar[d] \\
(\Box^1)^\sharp \ar[r] & K'
}
}

The left, right, and bottom maps are anodyne, hence weak equivalences by \cref{anodyne-we}. Thus the top map is a weak equivalence by \cref{we-2-of-3}.
\end{proof}

\begin{lemma} \label{triv-fib-is-we}
Trivial fibrations are weak equivalences. 
\end{lemma}

\begin{proof}
If $A \to B$ is a trivial fibration, then it is a homotopy equivalence by \cref{nfib+weq=trivFib}. Hence $\ihom(B,X) \to \ihom(A,X)$ is a homotopy equivalence for all marked cubical quasicategories $X$ by \cref{precomp-htpy}, and hence $A \to B$ a weak equivalence by \cref{we-characterizations}. 
\end{proof}

\begin{proposition} \label{cset'-cof-cat}
The category of marked cubical sets forms a cofibration category with the above classes of weak equivalences and cofibrations.
\end{proposition}

\begin{proof}
The class of weak equivalences is closed under 2-out-of-3 by \cref{we-2-of-3}.
The category has an initial object and pushouts.
Cofibrations are the left class in a weak factorization system, hence stable under pushout.
Using the characterization of weak equivalences given by item~(ii) of \cref{we-characterizations}, stability of cofibrations that are weak equivalences under pushout reduces to stability of trivial fibrations under pullback.
By \cref{triv-fib-is-we}, trivial fibrations are weak equivalences, so the (cofibration, trivial fibration)-factorization gives the factorization axiom.
\end{proof}

\subsection{Model structure for marked cubical quasicategories}

\begin{definition}
A marked cubical set is \emph{finite} (resp.~\emph{countable}) if it has only finitely (resp. countably) many non-degenerate cubes. The \emph{cardinality} of a finite marked cubical set is its total number of non-degenerate cubes, in all dimensions.
\end{definition}

\begin{lemma}\label{tfib-accessible}
The trivial fibrations form an $\omega_{1}$-accessible, $\omega_{1}$-accessibly embedded subcategory of $(\mathsf{cSet}')^{\to}$.
\end{lemma}

\begin{proof}
It suffices to show two things: that filtered colimits (and hence in particular $\omega_{1}$-filtered colimits) in $\mathsf{cSet}'$ preserve trivial fibrations, and that any trivial fibration can be expressed as an $\omega_{1}$-filtered colimit in $\mathsf{cSet}'$ of trivial fibrations between countable marked cubical sets. The first statement follows from the fact that the domains and codomains of the generating cofibrations are finite.

For the second statement, consider a trivial fibration $f \colon X \to Y$. Let $P$ denote the poset of countable subcomplexes of $X$; note that we consider edges of subcomplexes of $X$ to be marked if and only if they are marked in $X$. This category is $\omega_{1}$-filtered since any countable union of countable subcomplexes is countable. 

Let $i$ denote the inclusion $P \hookrightarrow \mathsf{cSet}'$; the colimit of this diagram is $X$. The images under $f$ of the countable subcomplexes of $X$, with the natural inclusions, also define a diagram $fi \colon P \to \mathsf{cSet}'$. One can easily show that trivial fibrations are surjective on underlying cubical sets; thus every cube of $Y$ appears in $fS$ for some countable subcomplex $S \subseteq X$. So $fi$ is a filtered diagram of subcomplexes of $Y$, in which the maps are inclusions and each cube of $Y$ is contained in some object of the diagram, with every marked edge of $Y$ being marked in some subcomplex in the diagram. From this, one can show that the colimit of $fi$ is $Y$. The map $f$ induces a natural transformation from $i$ to $fi$, whose induced map on the colimits is $f$ itself.

However, it may not be the case that for every component of this natural transformation is a trivial fibration. Thus we will replace $i$ by a different diagram, still having colimit $X$, with a natural transformation to $fi$ which does satisfy this property. For each countable subcomplex $S \subseteq X$, we will define a new countable subcomplex $\overline{S} \subseteq X$, such that $f\overline{S} = fS$, $f|_{\overline{S}} \colon \overline{S} \to f\overline{S}$ is a trivial fibration, and for $S' \subseteq S$, we have $\overline{S'} \subseteq \overline{S}$.

We first define $\overline{S}$ for finite $S$, proceeding by induction on cardinality. For $S = \varnothing$, we can simply set $\overline{S} = \varnothing$. Now assume that we have defined $\overline{S}$ for $|S| \leq m$, and consider a subcomplex $S$ of cardinality $m+1$. We will inductively define a family of subcomplexes $\overline{S}^{i}$ for $i \geq 0$, each countable and satisfying $f\overline{S}^{i} = fS$. Begin by setting $\overline{S}^{0} = S \cup \bigcup\limits_{S' \subsetneq S} \overline{S'}$. Then $\overline{S}^{0}$ is countable, $f\overline{S}^{0} = fS$, and for $S' \subseteq S$ we have $\overline{S'} \subseteq \overline{S}^{0}$.
 
Now assume that we have defined $\overline{S}^{i}$ for some $i \geq 0$, and let $\mathcal{D}$ be the set of all diagrams $D$ of the form:

\centerline{
\xymatrix{
\bd \Box^{n} \ar[r]^{\bd x_{D}} \ar[d] & \overline{S}^{i} \ar[d] \\
\Box^{n} \ar[r]^{y_{D}} & fS \\
}
}

Because $\overline{S}^{i}$ and $fS$ are countable, while $\bd \Box^{n}$ and $\Box^{n}$ are finite for any given $n$, there are countably many such diagrams. Because $f$ is a trivial fibration, for each such diagram we may choose a filler in $X$, i.e. an $n$-cube $x_{D} \colon \Box^{n} \to X$ whose boundary is $\bd x_{D}$, such that $fx_{D} = y_{D}$. Let $\overline{S}^{i+1} = \overline{S}^{i} \cup \bigcup\limits_{D \in \mathcal{D}} \{x_{D}\}$. Then $\overline{S}^{i+1}$ is still countable, since we have added at most countably many cubes to $\overline{S}^{i}$, and its image under $f$ is still $fS$, since each $x_{D}$ was chosen to map to a specific $y_{D} \in fS$.

Now let $\overline{S} = \bigcup\limits_{i \geq 0} \overline{S}^{i}$. This is countable, its image is $fS$, and for any $S' \subseteq S$ we have $\overline{S'} \subseteq \overline{S}$. Now consider a diagram:

\centerline{
\xymatrix{
\bd \Box^{n} \ar[r]^{\bd x} \ar[d] & \overline{S} \ar[d] \\
\Box^{n} \ar[r]^{y} & fS \\
}
}

Because $\Box^{n}$ is finite, the image of $\bd x$ is contained in some finite subcomplex of $\overline{S}$, hence in some $\overline{S}^{i}$, so it has a filler in $\overline{S}^{i+1}$ which maps to $y$. Furthermore, $f|_{\overline{S}}$ has the right lifting property with respect to the map $\Box^{1} \to (\Box^{1})^{\sharp}$, i.e. an edge $x \colon \Box^{1} \to \overline{S}$ is marked if and only $fx$ is marked, since this is true of edges in $X$. Thus $f|_{\overline{S}} \colon \overline{S} \to fS$ is a trivial fibration.

For a countably infinite $S \subseteq X$ we let $\overline{S} = \bigcup \overline{S'}$, where the union is taken over all finite subcomplexes $S' \subseteq S$. Then $f|_{\overline{S}}$ is the filtered colimit of the trivial fibrations $f|_{\overline{S'}}$, hence it is a trivial fibration.

The subcomplexes $\overline{S}$ with the natural inclusions define a diagram $\overline{i} \colon P \to \mathsf{cSet}'$, and $f$ induces a natural trivial fibration $\overline{i} \implies fi$. Observe that $\overline{i}$ is a filtered diagram of subcomplexes of $X$, in which the maps are inclusions and edges in the objects are marked if and only if they are marked in $X$; furthermore, every cube of $X$ is contained in some finite subcomplex $S$, and hence in $\overline{S}$. From this we can deduce that the colimit of $\overline{i}$ is $X$, by the same argument we used to show that the colimit of $fi$ is $Y$. The induced map between colimits is $f$; thus we have expressed $f$ as an $\omega_{1}$-filtered colimit of trivial fibrations between countable marked cubical sets.
\end{proof}

\begin{lemma} \label{we-accessible}
The weak equivalences form an $\omega_{1}$-accessible, $\omega_{1}$-accessibly embedded subcategory of $(\mathsf{cSet}')^{\to}$.
\end{lemma}

\begin{proof}
The (anodyne, naive fibration) factorization gives us a naive fibrant replacement functor $F \colon (\mathsf{cSet}')^{\to} \to (\mathsf{cSet}')^{\to}$. By \cite[Prop.~D.2.10]{joyal:theory-of-quasi-cats}, this functor is $\omega_{1}$-accessible, since the domains and codomains of the generating anodyne maps are all countable. By definition, the category of weak equivalences $\mathsf{we}$ is given by the following pullback in $\mathsf{Cat}$:

\centerline{
\xymatrix{
\mathsf{we}  \ar[r] \ar[d] \drpullback & (\mathsf{cSet}')^{\to} \ar[d]^{F} \\
\mathsf{tfib} \ar@{>->}[r] & (\mathsf{cSet}')^{\to} \\
}
}

By \cref{tfib-accessible}, $\mathsf{tfib}$ is an $\omega_{1}$-accessible category, and its embedding into $(\mathsf{cSet}')^{\to}$ is an $\omega_{1}$-accessible functor. By \cite[Thm. 5.1.6]{makkai-pare}, the category of $\omega_{1}$-accessible categories and $\omega_{1}$-accesible functors has finite limits, and these are computed in $\mathsf{Cat}$. Thus $\mathsf{we}$ is $\omega_{1}$-accessible, and its embedding into $(\mathsf{cSet}')^{\to}$ is an $\omega_{1}$-accessible functor.
\end{proof}

\begin{theorem}[Analogue of model structure on marked simplicial sets]\label{cubical-marked-ms}
The above classes of weak equivalences, cofibrations, and fibrations define a model structure on $\cSet'$.
\end{theorem}

\begin{proof}
We verify the assumptions of \cref{Jeff-Smith's-theorem}.

The category of marked cubical sets is locally finitely presentable.
Weak equivalences are an $\omega_{1}$-accessibly embedded, $\omega_{1}$-accessible subcategory of $(\cSet')^\to$ by \cref{we-accessible}.
Cofibrations have a small set of generators by \cref{cofibration-generators}.

Weak equivalences are closed under 2-out-of-3 and weak equivalences that are cofibrations are closed under pushout by \cref{cset'-cof-cat}.
Weak equivalences are closed under transfinite composition by \cref{we-accessible}, implying that the same holds for trivial cofibrations.
Every map lifting against cofibrations is a weak equivalence by \cref{triv-fib-is-we}.
\end{proof}

We refer to the model structure constructed above as the \emph{cubical marked model structure}. We will now analyze this model structure, beginning with a strengthening of \cref{anodyne-pop} and \cref{fibrantExps}.

\begin{lemma}\label{geo-prod-we}
If $X \to Y$ is a weak equivalence, then so is $A \otimes X \to A \otimes Y$ for any $A \in \cSet'$.
\end{lemma}

\begin{proof}
By the adjunction $A \otimes - \adjoint \ihom_{R}(A,-)$, for $Z \in \cSet'$ we have a natural isomorphism $\ihom_{R}(A \otimes X,Z) \cong \ihom_{R}(X,\ihom_{R}(A,Z))$. Let $Z$ be a marked cubical quasicategory; then we have a commuting diagram

\centerline{
\xymatrix{
\ihom_{R}(A \otimes Y,Z) \ar[d]^{\cong} \ar[r] & \ihom_{R}(A \otimes X,Z) \ar[d]^{\cong} \\
\ihom_{R}(Y,\ihom_{R}(A,Z)) \ar[r] & \ihom_{R}(X,\ihom_{R}(A,Z)) \\
}
}

By \cref{fibrantExps}, $\ihom_{R}(A,Z)$ is a marked cubical quasicategory, so the bottom map is a homotopy equivalence by \cref{we-characterizations}. Hence the top map is a homotopy equivalence; thus we see that $A \otimes X \to A \otimes Y$ is a weak equivalence by \cref{we-characterizations}. 
\end{proof}

\begin{lemma}\label{we-pop}
The pushout product of a cofibration and a weak equivalence is a weak equivalence. 
\end{lemma}

\begin{proof}
Let $i \colon A \to B$ be a cofibration and $f \colon X \to Y$ a weak equivalence; we will show that $i \hatotimes f$ is a weak equivalence (the case of $f \hatotimes i$ is similar). Consider the diagram which defines $i \hatotimes f$:

\centerline{
\xymatrix{
A \otimes X \ar[r] \ar[d] & B \otimes X \ar[d] \ar@/^1.5pc/[ddr] \\
A \otimes Y \ar[r] \ar@/_1.5pc/[drr] & A \otimes Y \cup_{A \otimes X} B \otimes X \ar[dr]^{i \hatotimes f} \\
&& B \otimes Y
}
}

The maps $A \otimes X \to A \otimes Y$ and $B \otimes X \to B \otimes Y$ are weak equivalences by \cref{geo-prod-we}. The map $A \otimes X \to B \otimes X$ is a cofibration by \cref{anodyne-pop}. The model structure is left proper, since all objects are cofibrant; thus the map from $B \otimes X$ into the pushout is a weak equivalence. Hence $i \hatotimes f$ is a weak equivalence by 2-out-of-3. 
\end{proof}

\begin{corollary}\label{tcof-pop}
Let $i \colon A \to B, j \colon A' \to B'$ be cofibrations. If either $i$ or $j$ is trivial, then so is the pushout product $i \hatotimes j$.
\end{corollary}

\begin{proof}
This is immediate from \cref{anodyne-pop,we-pop}.
\end{proof}

\begin{corollary}\label{markedExps}
If $i$ is a cofibration and $f$ is a fibration, then the pullback exponential $i \triangleright f$ is a fibration, which is trivial if $i$ or $f$ is trivial. \qed
\end{corollary}

\begin{corollary}
The category $\cSet'$, equipped with the cubical marked model structure and the geometric product, is a monoidal model category. \qed
\end{corollary}

Next we will characterize the fibrant objects, and fibrations between fibrant objects, of this model structure.

\begin{proposition}\label{marked-fib-obs}
A map between marked cubical quasicategories is a fibration if and only if it is a naive fibration. In particular, the fibrant objects of the cubical marked model structure are precisely the marked cubical quasicategories.
\end{proposition}

\begin{proof}
It is clear that every fibration is a naive fibration. Now let $f \colon X \to Y$ be a naive fibration between marked cubical quasicategories, and $i \colon A \to B$ a trivial cofibration. We wish to show that $f$ has the right lifting property with respect to $i$; for this it suffices to show that $i \triangleright f$ has the right lifting property with respect to the map $\varnothing \to \Box^0$. For this, in turn, it suffices to show that $i \triangleright f$ is a trivial fibration.

First, note that $i \triangleright f$ is a naive fibration between marked cubical quasicategories by \cref{fibrantExps}. Therefore, by \cref{nfib+weq=trivFib}, it is a trivial fibration if and only if it is a homotopy equivalence. Now consider the diagram which defines $i \triangleright f$:

\centerline{
\xymatrix{
\ihom(B,X) \ar[dr]^{i \triangleright f} \ar@/^1.5pc/[drr] \ar@/_1.5pc/[ddr] \\
& \drpullback \ihom(A,X) \times_{\ihom(A,Y)} \ihom(B,Y) \ar[r] \ar[d]  & \ihom(A,X) \ar[d] \\
& \ihom(B,Y) \ar[r] & \ihom(A,Y) \\
}
}

The maps $\ihom(B,X) \to \ihom(A,X)$ and $\ihom(B,Y) \to \ihom(A,Y)$ are trivial fibrations by \cref{markedExps}; the map from the pullback to $\ihom(A,X)$ is a trivial fibration as a pullback of a trivial fibration. Thus $i \triangleright f$ is a weak equivalence by 2-out-of-3, hence a homotopy equivalence by \cref{we-marked-fibrant-objects}.
\end{proof}


\begin{proposition}\label{op-co-Quillen-equivalence-marked}
The adjunctions $(-)^\rmco \adjoint (-)^\rmco, (-)^\coop \adjoint (-)^\coop$ are Quillen self-equivalences of $\cSet'$.
\end{proposition}

\begin{proof}
By \cref{QuillenEquivInvolution}, it suffices to show that the adjunctions are Quillen. To do this, we apply \cref{Quillen-adj-fib-obs}. Both $(-)^\rmco$ and $(-)^\coop$ preserve cofibrations, marked open box inclusions, and three-out-of-four maps; thus it remains to consider only the saturation map. 

The image of the saturation map under $(-)^\coop$ is isomorphic to the saturation map itself, and is therefore a trivial cofibration.  Now consider the map $K^\rmco \to (K')^\rmco$. To show that this is a trivial cofibration, it suffices to show that it has the left lifting property with respect to fibrations between marked cubical quasicategories. If $X$ is a marked cubical quasicategory, then $K^\rmco \to (K')^\rmco$ has the left lifting property against $X \to \Box^0$ by the fact that the marked edges in $X$ are precisely those which are invertible in $\Ho X$. Since $K^\rmco \to (K')^\rmco$ is an epimorphism, it therefore has the left lifting property against all maps between marked cubical quasicategories.
\end{proof} 

\section{Model structure on structurally marked cubical sets} \label{sec:structurally-marked}

The model structure on marked cubical sets described in the previous section resembles the Cisinski model structure on a presheaf category. In this section, we show that the category $\cSet''$ of structurally marked cubical sets (see Section \ref{Sec1Marked}) admits a Cisinski model structure which right induces the model structure on marked cubical sets from the previous section via the embedding $\cSet' \hookrightarrow \cSet''$, and that the two are Quillen equivalent.

None of the arguments in this section make use of connections, thus they are valid in all of the categories of cubical sets shown in the diagram \cref{eq:cube-cat-inclusions} at the end of \cref{sec:background}.

Since $\cSet''$ is a presheaf category, we may apply \cref{CisinskiMS} in order to construct a model structure on this category.
To do that, we first find a cellular model for $\cSet''$, i.e., a generating set of monomorphisms, using the Reedy category structure of $\tBox$, established in \cref{tBox-Reedy}.

\begin{lemma}\label{struct-mark-cell-mod}
	The monomorphisms of $\cSet''$ are the saturation of the set consisting of the boundary inclusions $\partial \Box^n \hookrightarrow \Box^n$ and the inclusion $\Box^1 \hookrightarrow (\Box^1)^\sharp$. \qed
\end{lemma}

The functor $(\Box^{1})^{\sharp} \otimes - \colon \cSet'' \to \cSet''$, together with the natural transformations $\bd^{1}_{1,0} \otimes -, \bd^{1}_{1,1} \otimes - \colon \mathrm{id}  \to (\Box^{1})^{\sharp} \otimes -$, and $\pi \colon (\Box^{1})^{\sharp} \otimes - \to \mathrm{id}$, defines a cylinder functor on $\cSet''$ in the sense of \cref{CylinderFunctor}. 

Thus we have a notion of homotopy defined in terms of this cylinder functor: an elementary homotopy $f \sim g \colon X \to Y$ is a map $H \colon (\Box^1)^\sharp \otimes X \to Y$ with $H|_{\{0\} \otimes X} = f, H|_{\{1\} \otimes X} = g$, and a homotopy is a zigzag of elementary homotopies. In keeping with the notation of Section \ref{Sec1Cisinski},  we will write $[X,Y]$ for the set of homotopy classes of maps from $X$ to $Y$.

\begin{lemma} \label{cylinders-agree}\leavevmode
  \begin{enumerate}
    \item The cylinder functors in $\cSet'$ and $\cSet''$ agree, i.e., the latter is the image of the former under the embedding $\cSet' \to \cSet''$.
    \item For marked cubical sets $X$ and $Y$, the embedding $\cSet' \to \cSet''$ induces a bijection
      \[ [X, Y]_{\cSet'} \to [X, Y]_{\cSet''}\text{,} \]
      where the subscript indicates which category the homotopy classes are taken in.
  \end{enumerate}
\end{lemma}

\begin{proof}
  Both of these statements follow easily from the fact that the embedding $\cSet' \to \cSet''$ is monoidal, established in \cref{flat-monoidal}.
\end{proof}

Let $S$ be the set of maps in $\cSet''$ consisting of the following maps:
\begin{itemize}
	\item the marked open box inclusions,
	\item the saturation map, and
	\item the 3-out-of-4 maps.
\end{itemize}

\begin{definition}
A map of structurally marked cubical sets is \emph{anodyne} if it is in the saturation of $S$.
\end{definition}

Note that the anodyne generators in $\cSet''$ are precisely those of $\cSet'$, embedded via $\cSet' \to \cSet''$.

\begin{remark}
It might seem natural to include the map $(\Box^1)^\sharp \to (\Box^1)^{2\sharp}$, the inclusion of the marked interval into the interval with two distinct markings, in $S$, so that adding a marking to an already-marked edge of a structurally marked cubical set would not change its homotopy type. In fact, however, this map is already anodyne, as it is a pushout of a 3-out-of-4 map.
\end{remark}

The following lemma shows that this definition of anodyne maps is consistent with that of \cref{Sec1Cisinski}.

\begin{lemma}\label{struct-mark-Lambda}
The set $\Lambda(S)$ is contained in the saturation of $S$.
\end{lemma}

\begin{proof}
The proof of \cref{anodyne-pop} applies equally well in this context, showing that a pushout product of a monomorphism with a map in the saturation of $S$ is again in the saturation of $S$. This implies that $\Lambda^{0}(S)$ is contained in the saturation of $S$; applying the same lemma inductively, we see that each set $\Lambda^{n}(S)$ is contained in the saturation of $S$.
\end{proof}

\begin{theorem}\label{model-struct-mark}
	The category $\cSet''$ of structurally marked cubical sets carries a cofibrantly generated model structure in which:
	\begin{itemize}
		\item the cofibrations are the monomorphisms;
		\item the fibrant objects, and fibrations between fibrant objects, are defined by the right lifting property with respect to the set of generating anodyne maps $S$;
		\item the weak equivalences are maps $X \to Y$ inducing bijections $[Y,Z] \to [X,Z]$ for all fibrant objects $Z$.
	\end{itemize}
\end{theorem}

\begin{proof}
	The existence of the model structure follows from \cref{CisinskiMS}. \cref{struct-mark-Lambda} shows that the set of generating anodyne maps is exactly $S$.
\end{proof}

The remainder of this section will be devoted to analyzing this model structure and its relationship with the model structure on marked cubical sets of \cref{cubical-marked-ms}.
More precisely, we will prove:

\begin{theorem} \label{compare-marked-struct-marked}\leavevmode
  \begin{enumerate}
    \item The image adjunction $\cSet'' \rightleftarrows \cSet'$ of the diagram \cref{eq:all-functors-marking} is a Quillen equivalence.
    \item The cubical marked model structure is right induced from the model structure of \cref{model-struct-mark} by the embedding $\cSet' \to \cSet''$.
\end{enumerate}    
\end{theorem}

Before proving this theorem, we establish a number of intermediate results.

\begin{proposition} \label{Quillen-adj-marked-struct-marked}
	The image adjunction $\cSet'' \rightleftarrows \cSet'$ of the diagram \cref{eq:all-functors-marking} is a Quillen adjunction between the model structure of \cref{model-struct-mark} to the cubical marked model structure.
\end{proposition}

\begin{proof}
	By \cref{Quillen-adj-fib-obs}, it suffices to show that $\Im$ preserves monomorphisms and takes generating anodynes to anodynes.
	Both of these statements are immediate.
\end{proof}

\begin{lemma}
A map of structurally marked cubical sets $f \colon X \to Y$ is a trivial fibration if and only if the underlying map of cubical sets is a trivial fibration in model structure of \cref{Grothendieck-ms-cSet} (i.e., has the right lifting property with respect to monomorphisms) and, for each edge $x$ of $X$, the map from the set of markings of $x$ to that of $fx$ is surjective.
\end{lemma}

\begin{proof}
By \cref{struct-mark-cell-mod}, $f$ is a trivial fibration if and only if it has the right lifting property with respect to all boundary inclusions and the inclusion of the interval into the marked interval. Having the right lifting property with respect to all boundary inclusions is equivalent to being a trivial fibration on underlying cubical sets; having the right lifting property with respect to the inclusion of the interval into the marked interval is equivalent to each map of marking sets being surjective.
\end{proof}

\begin{corollary}\label{markedUnitLift}
For all structurally marked cubical sets $X$, the adjunction unit $X \to \mathrm{Im}X$ is a trivial fibration. \qed
\end{corollary}

\begin{lemma} \label{premarkedFibs}
The functor $\Im \colon \cSet'' \to \cSet'$ preserves fibrations and trivial fibrations.
\end{lemma}

\begin{proof}
We will show that if $p \colon X \to Y$ is a fibration between structurally marked cubical sets, then $\Im\, p\colon \Im X \to \Im Y$ is also a fibration. Given a trivial cofibration of marked cubical sets $i \colon A \hookrightarrow B$ with maps $\alpha \colon A \to \Im X$ and $\beta \colon B \to \Im Y$ making the square commute, apply \cref{markedUnitLift} to $\varnothing \to A \to \Im X$ to get $\alpha' \colon A \to X$ with $\alpha = u_X \alpha'$ and then again to the square
\[
\xymatrix{
A \ar[d]_{i} \ar[r]^{p\alpha'} & Y \ar[d]^{u_Y} \\
B \ar[r]^{\beta} & \Im Y }\]
 to get $\beta' \colon B \to Y$ that fits into a square
\[
\xymatrix{
A \ar[d]_{i} \ar[r]^{\alpha'} & X \ar[d]^{p} \\
B \ar[r]^{\beta'} & Y }\] 

whose lift $L \colon B \to X$ yields $u_X L \colon B \to \Im X$ which satisfies the equations 

\[ (\Im p) u_X L = u_Y p L 
                         = u_Y \beta' 
                         = \beta \quad \text{and} \quad
u_x L i = u_x \alpha' = \alpha\]
and so provides the lift. Thus $\Im p$ is a fibration.

The proof for trivial fibrations is analogous.
\end{proof}

\begin{lemma}\label{ImFibrant}
Let $X$ be a structurally marked cubical set. Then $X$ is fibrant if and only if $\mathrm{Im}X$ is fibrant (in the model structure of \cref{model-struct-mark}).
\end{lemma}

\begin{proof}
If $\mathrm{Im}X$ is fibrant, then $X$ is fibrant by \cref{markedUnitLift}. Conversely, if $X$ is fibrant, then $\Im X$ is fibrant in $\cSet'$ by \cref{premarkedFibs}, hence also in $\cSet''$ by \cref{Quillen-adj-marked-struct-marked}.
\end{proof}

\begin{proof}[Proof of \cref{compare-marked-struct-marked}]
  First, let us show that the right derived functor of the embedding $\cSet' \to \cSet''$ is an equivalence.
  Since $[X, Y]_{\cSet'} \to [X, Y]_{\cSet''}$ is bijective by \cref{cylinders-agree}, it is full and faithful.
  For essential surjectivity, by \cref{CisinskiHo}, given fibrant $X \in \cSet''$, we need fibrant $Y \in \cSet'$ weakly equivalent to $X$ in $\cSet''$.
  This is given by \cref{premarkedFibs,ImFibrant}.

  Now, let us show that cubical marked model structure is right induced.
  Since $\Im$ is a left Quillen equivalence and all objects are cofibrant, it preserves and reflects weak equivalences by \cref{QuillenEquivCreate-original}, hence so does the embedding.
  Since $\Im$ preserves fibrations, the embedding reflects them.
  That the embedding preserves fibrations is part of \cref{Quillen-adj-marked-struct-marked}.
\end{proof}

\section{Joyal model structure on cubical sets} \label{sec:joyal-cset}

Recall the adjunction $\cSet \rightleftarrows \cSet'$ of Section \ref{Sec1Marked}, in which the left adjoint is the minimal marking functor and the right adjoint is the forgetful functor. In this section we will use this adjunction to induce a model structure on $\cSet$ from the model structure on $\cSet'$ of \cref{cubical-marked-ms}.

None of the arguments in this section rely on the existence of connections, thus they are valid in all of the categories of cubical sets shown in the diagram \cref{eq:cube-cat-inclusions} at the end of \cref{sec:background}.

\begin{lemma}\label{K-cylinder}
For $X \in \cSet$, the image of the factorizations $X \sqcup X \to K \otimes X \to X$ and $X \sqcup X \to X \otimes K \to X$ under the minimal marking functor define cylinder objects for $X^\flat$ in $\cSet'$.
\end{lemma}

\begin{proof}
By definition, the minimal marking functor sends the first map in each of these factorizations to a cofibration, i.e. a monomorphism; that it sends the second to a weak equivalence follows from \cref{KendEqMarked,tcof-pop}.
\end{proof}

\begin{theorem}[Analogue of Joyal model structure] \label{cubical-joyal-ms}
The category $\cSet$ of cubical sets carries a model structure in which:
\begin{itemize}
\item the cofibrations are the monomorphisms,
\item the weak equivalences are created by the minimal marking functor,
\item the fibrations are right orthogonal to trivial cofibrations.
\end{itemize}
\end{theorem}

\begin{proof}
Apply \cref{transfer-theorem} to the adjunction $\cSet \rightleftarrows \cSet'$ and the cubical marked model structure, with the factorization $X \sqcup X \to K \otimes X \to X$. \cref{K-cylinder} shows that this factorization satisfies the hypotheses of \cref{transfer-theorem}
\end{proof}

We refer to the model structure constructed above as the \emph{cubical Joyal model structure}. Its weak equivalences will be referred to as \emph{weak categorical equivalences}, respectively.

\begin{proposition}\label{unmarked-marked-Quillen-equivalence}
The adjunction $\cSet \rightleftarrows \cSet'$ is a Quillen equivalence.
\end{proposition}

\begin{proof}
The minimal marking functor preserves and reflects weak equivalences by definition, thus we may apply \cref{QuillenEquivCreate} \ref{QuillenEquivCounit}. Let $X$ be a marked cubical quasicategory; abusing notation slightly, let $X^{\flat}$ denote the minimal marking of the underlying cubical set of $X$. We must show that the inclusion $X^{\flat} \to X$ is a weak equivalence. 

The marked edges of $X^{\flat}$ are precisely the degenerate edges; by \cref{equivalence-iff-marked}, the marked edges of $X$ are precisely those edges $\Box^1 \to X$ which factor through $K$. Thus $X^{\flat} \to X$ is a pushout of a coproduct of saturation maps, hence a trivial cofibration.
\end{proof}

We define some terminology which will be used in the analysis of this model structure.

\begin{itemize}
\item For $n \geq 2$, $1 \leq i \leq n, \varepsilon \in \{0,1\}$, the $(i,\varepsilon)$-\emph{inner open box}, denoted $\widehat{\sqcap}^{n}_{i,\varepsilon}$, is the quotient of an open box with the critical edge quotiented to a point. The $(i,\varepsilon)$-\emph{inner cube}, denoted $\widehat{\Box}^{n}_{i,\varepsilon}$, is defined similarly. The $(i,\varepsilon)$-\emph{inner open box inclusion} is the inclusion $\widehat{\sqcap}^{n}_{i,\varepsilon} \hookrightarrow \widehat{\Box}^{n}_{i,\varepsilon}$.
\item The class of \emph{anodyne maps} is the saturation of the set of inner open box inclusions.
\item An \emph{inner fibration} is a map having the right lifting property with respect to the inner open box inclusions.
\item An \emph{isofibration} is a map having the right lifting property with respect to the endpoint inclusions $\Box^0 \hookrightarrow K$.
\item A \emph{cubical quasicategory} is a cubical set $X$ such that the map $X \to \Box^0$ is an inner fibration.
\item An \emph{equivalence} in a cubical set $X$ is an edge $\Box^{1} \to X$ which factors through the inclusion of the middle edge $\Box^{1} \to K$.
\item For $n \geq 2, 1 \leq i \leq n, \varepsilon \in \{0,1\}$, a \emph{special open box} in a cubical set $X$ is a map $\sqcap^{n}_{i,\varepsilon} \to X$ which sends the critical edge to an equivalence.
\end{itemize}

The concept of homotopy developed in \cref{section:marked} adapts naturally to this setting, using equivalences in place of marked edges.

\begin{definition}
For a cubical set $X$, let $\sim_{0}$ denote the relation on $X_{0}$, the set of vertices of $X$, given by $x \sim_{0} y$ if there is an equivalence from $x$ to $y$ in $X$. Let $\sim$ denote the smallest equivalence relation on $X_{0}$ containing $\sim_{0}$.
\end{definition}

\begin{remark}
For $x, y \in X_{0}$, one can easily see that $x \sim y$ if and only if $x$ and $y$ are connected by a zigzag of equivalences. 
\end{remark}

\begin{definition}
For a cubical set $X$, the \emph{set of connected components} $\pi_0(X)$ is $X_0/\sim$.
\end{definition}

\begin{definition} 
An \emph{elementary left homotopy} $h \co f \sim g$ between maps $f, g \co A \to B$ is a map $h \co K \otimes A \to B$ such that $h|_{\braces{0} \otimes A} = f$ and $h|_{\braces{1} \otimes A} = g$.
Note that the elementary left homotopy $h$ corresponds to an edge $K \to \ihom_L(A, B)$ between the vertices corresponding to $f$ and $g$.
A \emph{left homotopy} between $f$ and $g$ is a zig-zag of elementary left homotopies.
\end{definition}

A left homotopy from $f$ to $g$ corresponds to a zig-zag of equivalences in $\ihom_L(A, B)$ and so maps from $A$ to $B$ are left homotopic exactly if $\pi_0(f) = \pi_0(g)$, where the set of connected components is taken in $\ihom_L(A, B)$.

These induce notions of \emph{elementary left homotopy equivalence} and \emph{left homotopy equivalence}. Each of these notions has a ``right'' variant using $A \otimes K$ and $\ihom_R(A,B)$. As in \cref{section:marked}, unless the potential for confusion arises or a statement depends on the choice, we will drop the use of ``left'' and ``right''. Homotopy equivalences between cubical quasicategories will be referred to as \emph{categorical equivalences}.

\begin{definition}
Let $X$ be a cubical set. The \emph{natural marking} on $X$ is a marked cubical set $X^\natural$ whose underlying cubical set is $X$, with edges marked if and only if they are equivalences.
\end{definition}

Since maps of cubical sets preserve equivalences, this defines a functor $(-)^\natural \colon \cSet \to \cSet'$.

Many results about the cubical Joyal model structure follow easily from the corresponding results about the cubical marked model structure.

\begin{lemma}\label{cof-tcof-pop-unmarked}
If $i, j$ are cofibrations in $\cSet$, then the pushout product $i \hatotimes j$ is a cofibration. Moreover, if either $i$ or $j$ is trivial then so is $i \hatotimes j$.
\end{lemma}

\begin{proof}
This is immediate from \cref{flat-monoidal,anodyne-pop,tcof-pop}.
\end{proof}

\begin{corollary}\label{fibrant-exps-unmarked}
Let $i, f$ be maps in $\cSet$. If $i$ is a cofibration and $f$ is a fibration, then the pullback exponential $i \triangleright f$ is a fibration. \qed
\end{corollary}

\begin{corollary}
The category $\cSet$, equipped with the cubical Joyal model structure and the geometric product, is a monoidal model category. \qed
\end{corollary}

Next we will characterize the fibrant objects, and fibrations between fibrant objects, in the cubical Joyal model structure.

\begin{lemma} \label{KendEq} 
The inner open box inclusions $\widehat{\sqcap}^{n}_{i,\varepsilon} \to \widehat{\Box}^{n}_{i,\epsilon}$, and the endpoint inclusions $\Box^0 \to K$, are trivial cofibrations. 
\end{lemma}

\begin{proof}
The minimal marking of an inner open box inclusion is a pushout of a marked open box inclusion in $\cSet'$. The minimal marking of $\Box^0 \to K$ is a trivial cofibration by \cref{KendEqMarked}.
\end{proof}

\begin{lemma} \label{joyal-complex-special-open-box}
Cubical quasicategories have fillers for special open boxes.
\end{lemma}

\begin{proof}
We only consider positive filling problems; the negative case is dual.
We argue by induction on the dimension of the filling problem. 

For a special open box of dimension 2, one can explicitly construct a filler by extending the given open box to an inner open box of dimension 3. We illustrate this construction for the case of a $(1,0)$-open box; the case of a $(2,0)$-open box is similar.

Consider the following open box, where the edge $e$ is an equivalence:

\centerline{
\xymatrix{
x \ar[r]^{f}  & y \ar[d]^{g} \\
w \ar[r]^{e} & z \\ 
}
}

Our assumption that $e$ is an equivalence means that there exist a pair of 2-cubes as follows:

\centerline{
\xymatrix{
z \ar[r]^{e^{-1}_{R}} \ar@{=}[d] & w \ar[d]^{e} \ar@{=}[r] & w \ar@{=}[d] \\
z \ar@{=}[r] & z \ar[r]^{e^{-1}_{L}} & w \\
}
} 

We extend this to a 3-dimensional $(1,1)$-open box, as depicted below. Here the front face is that which witnesses $e^{-1}_{R}$ as a right inverse to $e$, while the top and bottom faces are obtained by two-dimensional inner open box filling.

 
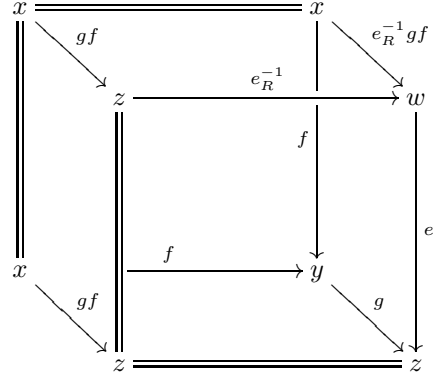
\[
\xymatrix@!C{
 x
 \ar@{=}[rrr]
 \ar@{=}[ddd]
 \ar[dr]^{gf}
&&&
 x
 \ar[ddd]|!{[dl];[dr]}{\hole}_{f}
 \ar[dr]^{e^{-1}_{R}gf}
\\&
 z
 \ar[rrr]^{e^{-1}_{R}}
 \ar@{=}[ddd]
&&&
 w
 \ar[ddd]^{e}
\\
\\
 x
 \ar[rrr]|!{[uur];[dr]}{\hole}^{f}
 \ar[dr]^{gf}
&&&
 y
 \ar[dr]^{g}
\\&
 z
 \ar@{=}[rrr]
&&&
 z
}
\]

This open box is inner, hence it has a filler; the $(1,1)$-face of this 3-cube is a filler for the original 2-dimenisonal open box.

Now let $X$ be a cubical quasicategory, and suppose that $X$ has fillers for all special open boxes of dimension less than $n$. Consider a filling problem in $X$ of dimension $n$:
\[
\xymatrix@C+2cm{
  (\bd \Box^a \otimes \Box^1 \otimes \Box^b) \cup (\Box^a \otimes \braces{0} \otimes \Box^b) \cup (\Box^a \otimes \Box^1 \otimes \bd \Box^b)
  \ar[r]
  \ar@{>->}[d]
&
  X
\\
  \Box^a \otimes \Box^1 \otimes \Box^b
  \ar@{.>}[ur]
}
\]

We regard the codomain of the left map as a negative face of a larger cube via the map
\[
\xymatrix{
  \Box^a \otimes \Box^1 \otimes \Box^b
  \ar@{>->}[r]
&
  \Box^a \otimes \Box^1 \otimes \braces{0} \otimes \Box^b
}
\]
and the domain as the corresponding subobject.
The original filling problem then becomes a filling problem in $X$ of the form
\[
\begin{split}
&(\bd \Box^a \otimes \Box^1 \otimes \braces{0} \otimes \Box^b)
\\\cup&
(\Box^a \otimes \braces{0} \otimes \braces{0} \otimes \Box^b)
\\\cup&
(\Box^a \otimes \Box^1 \otimes \braces{0} \otimes \bd \Box^b)
\\&\to
\Box^a \otimes \Box^1 \otimes \braces{0} \otimes \Box^b
\end{split}
\]
where the critical edge is
\[
0^a000^b \to 0^a100^b
.\]
We will solve this problem by extending the given partial data to the whole of
\[
\Box^a \otimes \Box^1 \otimes \Box^1 \otimes \Box^b
.\]

For $n \geq 0$, let $\Gamma^n \subseteq \Box^n$ denote the union of the positive faces.
We use degeneracies in the new direction to fill
\[
\begin{split}
&(\Gamma^a \otimes \Box^1 \otimes \braces{0} \otimes \Box^b)
\\\cup&
(\Box^a \otimes \braces{0} \otimes \braces{0} \otimes \Box^b)
\\\cup&
(\Box^a \otimes \Box^1 \otimes \braces{0} \otimes \Gamma^b)
\end{split}
\to
\begin{split}
&(\Gamma^a \otimes \Box^1 \otimes \Box^1 \otimes \Box^b)
\\\cup&
(\Box^a \otimes \braces{0} \otimes \Box^1 \otimes \Box^b)
\\\cup&
(\Box^a \otimes \Box^1 \otimes \Box^1 \otimes \Gamma^b)
.\end{split}
\]
Since the critical edge is an equivalence, we can fill the square
\begin{equation} \label{joyal-complex-special-open-box:0}
\begin{gathered}
\xymatrix{
  0^a 00 0^b
  \ar@{=}[r]
  \ar[d]
&
  0^a 01 0^b
  \ar@{=}[d]
\\
  0^a 10 0^b
  \ar@{.>}[r]
&
  0^a 11 0^b
}
\end{gathered}
\end{equation}
where the dotted edge is again an equivalence.

In the following, we will indicate the filling direction of (generalized) open boxes by underlining the appropriate factor in the pushout monoidal product. What this means is that we can factor the given generalized open box inclusion as a series of open box fillings in different dimensions, each of which fills in the specified direction.
We now fill the generalized open box
\[
\braces{0^a} \otimes (\braces{0} \to \Box^1) \hatotimes \underline{(\braces{0} \to \Box^1)} \hatotimes (\braces{0^b} \to \Box^b)
\]
if $a, b \geq 1$.
Here, the critical edges are of the form $u v 0 w \to u v 1 w$ where $u, v, w$ are certain vertices of $\Box^a, \Box^1, \Box^b$, respectively.
All of these edges are degenerate except for the bottom edge in~\eqref{joyal-complex-special-open-box:0}, which is an equivalence.
Moreover, this edge only appears as a critical edge in filling problems of lower dimension.
So we may indeed fill this generalized open box using the fact that $X$ is a cubical quasicategory and the induction hypothesis.
Dually, we fill the generalized open box
\[
(\braces{0^a} \to \Box^a) \hatotimes (\braces{0} \to \Box^1) \hatotimes \underline{(\braces{0} \to \Box^1)} \otimes \braces{0^b}
\]
if $a, b \geq 1$.

We now fill the generalized open box
\[
(\braces{0^a} \cup \Gamma^a \to \bd \Box^a) \hatotimes (\braces{0} \to \Box^1) \hatotimes \underline{(\braces{0} \to \Box^1)} \hatotimes (\bd \Box^b \to \Box^b)
\]
if $a \geq 1$.
Again, the critical edges are of the form as above and we may argue as before.
Dually, we fill the generalized open box
\[
(\bd \Box^a \to \Box^a) \hatotimes (\braces{0} \to \Box^1) \hatotimes \underline{(\braces{0} \to \Box^1)} \hatotimes (\braces{0^b} \cup \Gamma^b \to \bd \Box^b)
\]
if $b \geq 1$.

At this stage, we have defined the cube on
\[
\begin{split}
&(\bd \Box^a \otimes \Box^1 \otimes \Box^1 \otimes \Box^b)
\\\cup&
(\Box^a \otimes \braces{0} \otimes \Box^1 \otimes \Box^b)
\\\cup&
(\Box^a \otimes \Box^1 \otimes \Box^1 \otimes \bd \Box^b)
.\end{split}
\]
We now fill the open box
\[
(\bd \Box^a \to \Box^a) \hatotimes \underline{(\braces{0} \to \Box^1)} \hatotimes (\varnothing \to \braces{1}) \hatotimes (\bd \Box^b \to \Box^b)
,\]
noting that the critical edge $0^a 00 0^b \to 0^a 10 0^b$ is degenerate.
We then fill the open box
\[
(\bd \Box^a \to \Box^a) \hatotimes (\varnothing \to \braces{1}) \hatotimes \underline{(\braces{0} \to \Box^1)} \hatotimes (\bd \Box^b \to \Box^b)
,\]
noting that the critical edge $0^a 00 0^b \to 0^a 01 0^b$ is degenerate.
We finally fill the open box
\[
(\bd \Box^a \to \Box^a) \hatotimes (\bd \Box^1 \to \Box^1) \hatotimes \underline{(\braces{0} \to \Box^1)} \hatotimes (\bd \Box^b \to \Box^b)
,\]
noting that the critical edge $0^a 00 0^b \to 0^a 01 0^b$ is degenerate.
This defines the entire cube.
\end{proof}

\begin{lemma} \label{inner-fib-special-open-box}
Let $X \to Y$ be an inner fibration between cubical quasicategories. Then a lift exists for any diagram of the form

\centerline{
\xymatrix{
\sqcap^{n}_{i,\varepsilon} \ar[r] \ar@{>->}[d] & X \ar[d] \\
\Box^n \ar[r] & Y \\
}
}

in which $\sqcap^{n}_{i,\varepsilon}$ is a special open box in $X$.
\end{lemma}

\begin{proof}
Again we only consider positive filling problems; the negative case is dual.
Again we argue by induction on the dimension of the filling problem, with the case for dimension 2 being an exercise in filling three-dimensional open boxes, analogous to the base case of the previous proof. Once again, we will illustrate the argument for the case of a $(1,0)$-open box, with the case of a $(2,0)$-open box being similar.

Consider a $(1,0)$-open box in $X$ whose image in $Y$ admits a filler $\alpha$, as depicted below on the left and right, respectively:

\centerline{
\xymatrix{
x \ar[r]^{f}  & y \ar[d]^{g} & \overline{x} \ar[r]^{\overline{f}} \ar[d]_{h}  & \overline{y} \ar[d]^{\overline{g}} \\
w \ar[r]^{e} & z & \overline{w} \ar[r]^{\overline{e}} & \overline{z} \\ 
}
}

Once again, we assume that the critical edge $e$ is an equivalence in $X$. 

We begin by filling a 2-dimensional inner open box to obtain the following 2-cube in $X$:

\centerline{
\xymatrix{
x \ar@{..>}[d]_{gf} \ar[r]^{f} & y \ar[d]^{g} \\
z \ar@{=}[r] & z \\
}
}

We may extend the 2-cube $\alpha$ to a 3-dimensional $(2,0)$-open box in $Y$, as depicted below; here the right face is $\alpha$, the front is the image in $Y$ of the 2-cube in $X$ which witnesses $e^{-1}_{R}$ as a right inverse to $e$, and the bottom is the image in $Y$ of the 2-cube constructed above.

 \[
\xymatrix@!C{
 \overline{x}
 \ar@{=}[rrr]
 \ar@{=}[ddd]
 \ar[dr]^{\overline{gf}}
&&&
 \overline{x}
 \ar[ddd]|!{[dl];[dr]}{\hole}_{\overline{f}}
 \ar[dr]^{h}
\\&
 \overline{z}
 \ar[rrr]^{\overline{e}^{-1}_{R}}
 \ar@{=}[ddd]
&&&
 \overline{w}
 \ar[ddd]^{\overline{e}}
\\
\\
 \overline{x}
 \ar[rrr]|!{[uur];[dr]}{\hole}^{\overline{f}}
 \ar[dr]^{\overline{gf}}
&&&
 \overline{y}
 \ar[dr]^{\overline{g}}
\\&
 \overline{z}
 \ar@{=}[rrr]
&&&
 \overline{z}
}
\]

The critical edge of this open box is the equivalence $e$, hence it admits a filler by \cref{joyal-complex-special-open-box}. The top face of the cube thus obtained is a filler for an inner open box in $X$, hence it can be lifted along the inner fibration $X \to Y$; in particular we obtain an edge $\widetilde{h}$ mapping to $h$. We thus obtain the following $(1,1)$-inner open box in $X$:

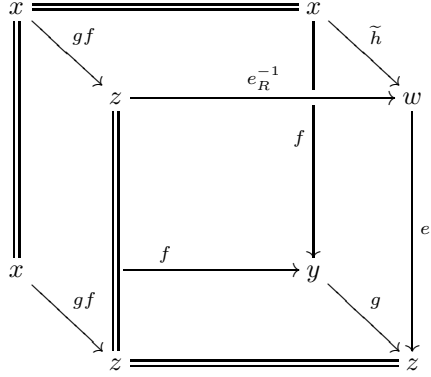
 \[
\xymatrix@!C{
 x
 \ar@{=}[rrr]
 \ar@{=}[ddd]
 \ar[dr]^{gf}
&&&
 x
 \ar[ddd]|!{[dl];[dr]}{\hole}_{f}
 \ar[dr]^{\widetilde{h}}
\\&
 z
 \ar[rrr]^{e^{-1}_{R}}
 \ar@{=}[ddd]
&&&
 w
 \ar[ddd]^{e}
\\
\\
 x
 \ar[rrr]|!{[uur];[dr]}{\hole}^{f}
 \ar[dr]^{gf}
&&&
 y
 \ar[dr]^{g}
\\&
 z
 \ar@{=}[rrr]
&&&
 z
}
\]

The critical edge of this open box is degenerate, and the previously constructed 3-cube in $Y$ is a filler for its image. Thus we may lift this filler along $X \to Y$; in particular, we obtain a filler for its right face which maps to $\alpha$.

Now assume that $X \to Y$ lifts against all special open box fillings of dimension less than or equal to $n$, and consider a lifting problem
\[
\xymatrix@C+2cm{
  (\bd \Box^a \otimes \Box^1 \otimes \Box^b) \cup (\Box^a \otimes \braces{0} \otimes \Box^b) \cup (\Box^a \otimes \Box^1 \otimes \bd \Box^b)
  \ar[r]
  \ar@{>->}[d]
&
  X
  \ar[d]
\\
  \Box^a \otimes \Box^1 \otimes \Box^b
  \ar[r]
  \ar@{.>}[ur]
&
  Y
}
\]
where $a + b = n$.
As before, we regard the codomain of the left map as a negative face of a larger cube via the map
\[
\xymatrix{
  \Box^a \otimes \Box^1 \otimes \Box^b
  \ar@{>->}[r]
&
  \Box^a \otimes \Box^1 \otimes \braces{0} \otimes \Box^b
}
\]
and the domain as the corresponding subobject $H$.
The critical edge is once again $0^a000^b \to 0^a100^b$.
Let $H'$ be the union of $H$ with the subobjects
\[
\begin{split}
&(\Gamma^a \otimes \Box^1 \otimes \Box^1 \otimes \Box^b)
\\\cup&
(\Box^a \otimes \braces{0} \otimes \Box^1 \otimes \Box^b)
\\\cup&
(\Box^a \otimes \Box^1 \otimes \Box^1 \otimes \Gamma^b)
\end{split}
\]
and $H''$ be the union of $H'$ with the square
\[
\braces{0^a} \otimes \Box^1 \otimes \Box^1 \otimes \braces{0^b}
.\]
We use degeneracies in the new direction to extend the map to $X$ from $H$ to $H'$:
\[
\xymatrix{
  H
  \ar[r]
  \ar@{^{(}->}[d]
&
  X
\rlap{.}\\
  H'
  \ar@{.>}[ur]  
}
\]
Since the critical edge is an equivalence in $X$, we extend the map to $X$ from $H'$ to $H''$ by filling the square
\[
\xymatrix{
  0^a 00 0^b
  \ar@{=}[r]
  \ar[d]
&
  0^a 01 0^b
  \ar@{=}[d]
\\
  0^a 10 0^b
  \ar@{.>}[r]
&
  0^a 11 0^b
}
\]
where the dotted edge is again an equivalence in $X$.
Note that the map $X \to Y$ preserves equivalences.

We construct the dotted arrow in the diagram
\[
\xymatrix{
  H
  \ar[r]
  \ar@{>->}[d]
&
  H''
  \ar[r]
  \ar[d]
&
  X
  \ar[d]
\\
  \Box^a \otimes \Box^1 \otimes \braces{0} \otimes \Box^b
  \ar[r]
  \ar@/_1.5em/[rr]
&
  \Box^a \otimes \Box^1 \otimes \Box^1 \otimes \Box^b
  \ar@{.>}[r]
&
  Y
}
\]
by solving a filling problem
\[
\xymatrix{
  (\Box^a \otimes \Box^1 \otimes \braces{0} \otimes \Box^b) \cup H''
  \ar[r]
  \ar[d]
&
  Y
\\
  \Box^a \otimes \Box^1 \otimes \Box^1 \otimes \Box^b
  \ar@{.>}[ur]
}
\]
as follows: the left map factors as a finite composite of open box inclusions of the form
\[
(\partial \Box^{a'} \to \Box^{a'}) \hatotimes (\braces{0} \to \Box^1) \hatotimes \underline{(\braces{0} \to \Box^1)} \hatotimes (\partial \Box^{b'} \to \Box^{b'})
\]
where $\Box^{a'}$ and $\Box^{b'}$ are faces of $\Box^a$ and $\Box^b$, respectively.
All critical edges are of the form $u v 0 w \to u v 1 w$ where $u, v, w$ are certain points of $\Box^a, \Box^1, \Box^b$, respectively.
All of these edges are degenerate in $Y$ except for the bottom edge in~\eqref{joyal-complex-special-open-box:0}, which is an equivalence.
We can thus fill these open boxes using the fact that $Y$ is a cubical quasicategory and \cref{joyal-complex-special-open-box}.

It remains to construct a lift
\[
\xymatrix{
  H''
  \ar[r]
  \ar[d]
&
  X
  \ar[d]
\\
  \Box^a \otimes \Box^1 \otimes \Box^1 \otimes \Box^b
  \ar[r]
  \ar@{.>}[ur]
&
  Y
\rlap{,}}
\]
which is done exactly as in the proof of \cref{joyal-complex-special-open-box} using that $X \to Y$ is an inner fibration.
\end{proof}

As with \cref{either-dir-comp}, the two proofs above make use of connections on 1-cubes (in this case, positive connections), but they can be adapted to $\cSet_\varnothing$ and $\cSet_0$ via \cref{cqcats-have-connections}.

\begin{lemma}\label{natural-marked-cubical-quasicat}
If $X$ is a cubical quasicategory, then $X^\natural$ is a marked cubical quasicategory.
\end{lemma}

\begin{proof}
Given a cubical quasicategory $X$, we have fillers for special open boxes in $X$ by \cref{joyal-complex-special-open-box}.
This implies that $X^\natural$ has fillers for marked open boxes. Furthermore, the definition of the natural marking implies that $X^\natural$ has the right lifting property with respect to the saturation map for any cubical set $X$. By \cref{marked-cubical-quasicat-3-of-4}, this suffices to show that $X^\natural$ is a marked cubical quasicategory.
\end{proof}

\begin{theorem}\label{fibrant-objects-unmarked}
The fibrant objects of the cubical Joyal model structure are given by cubical quasicategories.
The fibrations between fibrant objects are characterized by lifting against inner open box inclusions and endpoint inclusions $\Box^0 \hookrightarrow K$.
\end{theorem}

\begin{proof}
By \cref{KendEq}, every fibrant object is a cubical quasicategory and every fibration is an inner isofibration.

If $X$ is a cubical quasicategory, then $X^\natural$ is a marked cubical quasicategory by \cref{natural-marked-cubical-quasicat}. The forgetful functor $\cSet' \to \cSet$ preserves fibrant objects as a right Quillen adjoint, and the underlying cubical set of $X^\natural$ is $X$, thus $X$ is fibrant.

The case of fibrations between fibrant objects proceeds in an analogous way. Let $f \colon X \to Y$  be an inner isofibration between cubical quasicategories; we will show that $f^\natural$ is a fibration in $\cSet'$. Lifting against one-dimensional marked open box inclusions follows from the isofibration property; lifting against higher-dimensional marked open box inclusions follows from \cref{inner-fib-special-open-box}. To see that $f^\natural$ has the right lifting property with respect to the saturation and 3-out-of-4 maps, observe that any marked cubical quasicategory has the right lifting property with respect to these maps, hence so does any map between marked cubical quasicategories since the maps in question are epimorphisms. Since $X^\natural$ and $Y^\natural$ are marked cubical quasicategories, this implies that $f^\natural$ is a fibration by \cref{marked-fib-obs}.
\end{proof}

\begin{corollary}\label{weqIffHeq}
Let $f \colon X \to Y$ be a map between cubical quasicategories. Then $f$ is a weak categorical equivalence if and only if it is a categorical equivalence. \qed
\end{corollary}

\begin{proof}
This follows from \cref{K-cylinder,fibrant-objects-unmarked}.
\end{proof}

\begin{corollary}\label{cubical-quasicat-hom}
Let $X, Y \in \cSet$, with $Y$ a cubical quasicategory. Then $\ihom(X,Y)$ is a cubical quasicategory.
\end{corollary}

\begin{proof}
This follows from \cref{fibrant-exps-unmarked,fibrant-objects-unmarked}.
\end{proof}

Using \cref{fibrant-objects-unmarked}, we can see that this model structure can also be constructed using the Cisinski theory of Section \ref{Sec1Cisinski}.

\begin{proposition}
Let $\cSet_{K}$ denote the model structure given by applying \cref{CisinskiMS} to $\cSet$ with the following data:
\begin{itemize}
\item $I = - \otimes K$, with natural transformations $\bd_{\varepsilon}$ and $\sigma$ induced by the endpoint inclusions and the map $K \to \Box^0$;
\item $M = \{\bd \Box^n \hookrightarrow \Box^n | n \geq 0 \}$;
\item $S = \{\widehat{\sqcap}^{n}_{i,\varepsilon} \hookrightarrow \widehat{\Box}^{n}_{i,\varepsilon}|n \geq 2, 1 \leq i \leq n, \varepsilon = 0,1\}$. 
\end{itemize}

Then $\cSet_{K}$ coincides with the cubical Joyal model structure. 
\end{proposition}

\begin{proof}
The cofibrations in both model structures are the monomorphisms. Therefore, to show that the model structures coincide, it suffices to show that they have the same fibrant objects, i.e. that the objects having the right lifting property with respect to all maps in $\Lambda(S)$ are precisely the cubical quasicategories. For this, observe that all fibrant objects of $\cSet_K$ are cubical quasicategories, since $S \subseteq \Lambda(S)$ is precisely the set of inner open box inclusions. Furthermore, an inductive argument involving \cref{cof-tcof-pop-unmarked} shows that all maps in $\Lambda(S)$ are trivial cofibrations in the cubical Joyal model structure, so all cubical quasicategories are fibrant in $\cSet_K$.
\end{proof}

Our next goal will be to characterize the weak categorical equivalences in a manner similar to \cref{we-characterizations}.

\begin{lemma}\label{pi0-iso}
The following triangle of functors commutes:

\centerline{
\xymatrix{
\cSet \ar[dr]_{\pi_0} \ar[rr]^{(-)^\natural} && \cSet' \ar[dl]^{\pi_0} \\
& \Set
}
}
\end{lemma}

\begin{proof}
For $X \in \cSet$, $X$ and $X^\natural$ have the same set of vertices, and the equivalence relations defining $\pi_{0}X$ and $\pi_{0}X^\natural$ coincide.
\end{proof}

\begin{lemma}\label{hom-underlying}
Let $X, Y \in \cSet$, and let $Y'$ be a marked cubical set whose underlying cubical set is $Y$. The underlying cubical set of $\ihom(X^\flat,Y')$ is isomorphic to $\ihom(X,Y)$, and this isomorphism is natural in both $X$ and $Y$.
\end{lemma}

\begin{proof}
We will prove the statement for $\ihom_R$; the proof for $\ihom_L $ is similar. The $n$-cubes in the underlying cubical set of $\ihom_R(X^\flat,Y')$ are maps $X^\flat \otimes \Box^n \cong (X \otimes \Box^n)^\flat \to Y'$ (the isomorphism follows from \cref{flat-monoidal}). Under the adjunction $\cSet \rightleftarrows \cSet'$, these correspond to maps $X \otimes \Box^n \to Y$.
\end{proof}

\begin{proposition} \label{we-characterizations-unmarked}
  The following are equivalent for a cubical map $A \to B$:
  \begin{enumerate}
    \item \label{A-B-we} $A \to B$ is a weak categorical equivalence;
    \item \label{A-B-hom} for any cubical quasicategory $X$, the induced map $\ihom(B, X) \to \ihom(A, X)$ is a categorical equivalence;
    \item \label{A-B-pi} for any cubical quasicategory $X$, the induced map $\pi_0(\ihom(B, X)) \to \pi_0(\ihom(A, X))$ is a bijection.
  \end{enumerate}
\end{proposition}

\begin{proof}
To see that \ref{A-B-we} $\Rightarrow$ \ref{A-B-hom}, let $A \to B$ be a weak categorical equivalence in $\cSet$, and $X$ a cubical quasicategory. Then $X^\natural$ is a marked cubical quasicategory by \cref{natural-marked-cubical-quasicat}, so $\ihom(B^\flat,X^{\natural}) \to \ihom(A^\flat,X^\natural)$ is a homotopy equivalence by \cref{we-characterizations}. The underlying cubical set functor preserves weak equivalences between fibrant objects by Ken Brown's lemma, so $\ihom(B,X) \to \ihom(A,X)$ is a weak categorical equivalence by \cref{hom-underlying}. Hence it is a categorical equivalence by \cref{cubical-quasicat-hom,weqIffHeq}.

The implication \ref{A-B-hom} $\Rightarrow$ \ref{A-B-pi} is clear, so now we consider \ref{A-B-pi} $\Rightarrow$ \ref{A-B-we}. For this, let $X$ be the underlying cubical set of a marked cubical quasicategory $X'$, and note that by \cref{pi0-iso,hom-underlying}, we have the following commuting diagram in $\mathsf{Set}$:

\centerline{
\xymatrix{
\pi_{0} \ihom(B,X) \ar[r] \ar[d]_{\cong} & \pi_{0} \ihom(A,X) \ar[d]^{\cong} \\
\pi_{0} \ihom(B^{\flat},X') \ar[r] & \pi_{0} \ihom(A^{\flat},X') \\
}
}

Since the underlying cubical set functor preserves fibrant objects, $X$ is a cubical quasicategory. So if \ref{A-B-pi} holds then the top map is an isomorphism, hence so is the bottom map. Thus $A^\flat \to B^\flat$ is a weak equivalence in $\cSet'$ by \cref{we-characterizations}, meaning that $A \to B$ is a weak categorical equivalence.
\end{proof}

We now state two straightforward properties of the cubical Joyal model structure.

\begin{proposition}\label{straightforward-properties} \leavevmode
  \begin{enumerate}
    \item \label{Grothendieck-Joyal-localization} The Grothendieck model structure on $\cSet$ of \cref{Grothendieck-ms-cSet} is a localization of the cubical Joyal model structure.
    \item \label{tau-nerve-Quillen-adj} The adjunction $\tau_1 \colon \cSet \rightleftarrows \Cat : \! \Nerve$ is a Quillen adjunction between the canonical model structure on $\Cat$ and the cubical Joyal model structure.  \qed 
  \end{enumerate}  
\end{proposition}

The cubical Joyal model structure is clearly left proper, since all objects are cofibrant. However, it is not right proper. The proof of this fact is similar to the standard proof of the corresponding result for the Joyal model structure on $\sSet$, but requires an additional step due to the fact that inner cubes, unlike representable simplicial sets, are generally not fibrant.

\begin{proposition}
The cubical Joyal model structure is not right proper.
\end{proposition}

\begin{proof}
We will exhibit a fibration $X \to Z$ and a weak equivalence $Y \to Z$ such that the pullback map $X \times_{Z} Y \to X$ is not a weak equivalence. 

First consider the map $[1] \to [2]$ in $\Cat$ which picks out the morphism $0 \to 2$. This is an isofibration, hence its image under $N_{\Box}$ is a fibration by \cref{straightforward-properties} \ref{tau-nerve-Quillen-adj}.

We have a map $\widehat{\Box}^{2}_{2,0} \to N_{\Box} [2]$ given by the following 2-cube in $N_{\Box} [2]$:

\centerline{
\xymatrix{
0 \ar[r] \ar[d] & 2 \ar@{=}[d] \\
1 \ar[r] & 2 \\
}
}

Now consider the following commuting diagram in $\cSet$:

\centerline{
\xymatrix{
\bd \Box^1 \drpullback \ar[rr]^-{\{0,2\}} \ar[d] & & \widehat{\sqcap}^{2}_{2,0} \ar[d] \\
\Box^1 \drpullback \ar[d] \ar[rr]^-{0 \to 2} & & \widehat{\Box}^{2}_{2,0} \ar[d] \\
N_{\Box}[1] \ar[rr]^-{N_{\Box}(0 \to 2)} & & N_{\Box} [2] \\
}
}

Pullbacks of two monomorphisms in $\cSet$ are given by intersections; this is immediate from the corresponding result in $\Set$. From this, it follows that both of the squares in the diagram above are pullbacks. 

The middle horizontal map is a fibration, as a pullback of a fibration. So the inclusion $\bd \Box^1 \to \Box^1$ is the pullback of the trivial cofibration $\widehat{\sqcap}^{2}_{2,0} \to \widehat{\Box}^{2}_{2,0}$ along a fibration. However, it is not a weak equivalence by \cref{straightforward-properties} \ref{tau-nerve-Quillen-adj}, since its image under $\tau$ is not an equivalence of categories.
\end{proof}

Next we will study the interactions of the functors $(-)^\rmco$ and $(-)^\coop$ of Section \ref{subsection:cSet-basics} with the cubical Joyal model structure.

\begin{proposition}\label{op-Quillen-equiv}
The adjunctions $(-)^\rmco \adjoint (-)^\rmco$ and  $(-)^{\coop} \adjoint (-)^{\coop}$ are Quillen self-equivalences of the cubical Joyal model structure.
\end{proposition}

\begin{proof}
By \cref{QuillenEquivInvolution}, it suffices to show that the adjunctions are Quillen. We will prove the statement for $(-)^\rmco$; the proof for $(-)^\coop$ is identical.

To show that the adjunction $(-)^\rmco \adjoint (-)^\rmco$ is Quillen, we must show that $(-)^\rmco$ preserves cofibrations and trivial cofibrations. Unwinding the definitions, we must show that, given a map $f$ in $\cSet$, if $f^\flat$ is a (trivial) cofibration in $\cSet'$ then so is $(f^\rmco)^\flat$. We have the following commuting diagram:

\centerline{
\xymatrix{
\cSet \ar[r]^{(-)^\rmco} \ar[d]_{(-)^\flat} & \cSet \ar[d]^{(-)^\flat} \\
\cSet' \ar[r]^{(-)^\rmco} & \cSet'
}
}

The result thus follows from the fact that the map $(-)^\rmco \colon \cSet' \to \cSet'$ preserves (trivial) cofibrations by \cref{op-co-Quillen-equivalence-marked}.
\end{proof}

\begin{remark}
The statement of \cref{op-Quillen-equiv} concerning $(-)^\coop$ is not immediately adaptable to $\cSet_0$ and $\cSet_1$, as $(-)^\coop$ does not define an automorphism of either of these categories, but rather an isomorphism between them. Thus the corresponding result for cubical sets with only one connection states that for $\varepsilon \in \{0,1\}$, the adjunction $(-)^\coop : \cSet_{\varepsilon} \rightleftarrows \cSet_{1-\varepsilon} : (-)^\coop$ is a Quillen equivalence.
\end{remark}

The result above allows us to show that our set of pseudo-generating trivial cofibrations do not form a set of generating trivial cofirbations for the cubical Joyal model structure.

\begin{proposition}
The endpoint inclusions $\Box^0 \to K^\rmco$ have the right lifting property against all anodyne maps, but they are not fibrations.
\end{proposition}

\begin{proof}
Fix an endpoint inclusion $\Box^0 \to K^\rmco$; we must show that this map has the right lifting property against the inner open box inclusions and the endpoint inclusions $\Box^0 \to K$. Consider the following diagram in $\cSet$:

\centerline{
\xymatrix{
\widehat{\sqcap}^n_{i,\varepsilon} \ar[r] \ar[d] & \Box^0 \ar[d] \\
\widehat{\Box}^n_{i,\varepsilon} \ar[r] & K^\rmco \\
}
}

We may note that constant open boxes $\widehat{\sqcap}^n_{i,\varepsilon} \to K^\rmco$ for $n \geq 2$ have only constant fillers; thus the map $\widehat{\Box}^n_{i,\varepsilon} \to K^\rmco$ in this diagram factors through the unique map $\widehat{\Box}^n_{i,\varepsilon} \to \Box^0$, implying that the diagram admits a lift. Similarly, any map $K \to K^\rmco$ is constant, implying that $\Box^0 \to K^\rmco$ has the right lifting property against the maps $\Box^0 \to K$.

To see that $\Box^0 \to K^\rmco$ is not a fibration, observe that it is the image under $(-)^\rmco$ of one of the anodyne maps $\Box^0 \to K$, hence it is a trivial cofibration by \cref{op-Quillen-equiv}. Thus it cannot be a fibration, as it is not an isomorphism and therefore does not have the right lifting property against itself.
\end{proof}

We conclude this section with a proof of the following result, relating the cubical Joyal model structure to the Joyal model structure on simplicial sets via the triangulation functor.

\begin{proposition}\label{T-Quillen-adj}
The adjunction $T : \cSet \rightleftarrows \sSet : U$ is a Quillen adjunction between the cubical Joyal model structure and the Joyal model structure on $\sSet$.
\end{proposition}

Conceptually, this adjunction might be best understood at the level of marked simplicial and marked cubical sets. However, in order to avoid the burden of relying on the model structure on marked simplicial sets, we will compare the model structures on $\cSet$ and $\sSet$ directly.

\begin{lemma}\label{T-endpoints}
$T$ sends the endpoint inclusions $\Box^0 \to K$ to trivial cofibrations in the Joyal model structure.
\end{lemma}

\begin{proof}
We will construct a weak categorical equivalence from $TK$ to the simplicial set $J$ of \cref{sSet-J}. The simplicial set $TK$ is depicted below:

\[
\xymatrix{
  1 \ar[r] \ar@{=}[d] \ar[dr] & 0 \ar[d] \ar@{=}[r] \ar[dr] &0\ar@{=}[d] \\
  1 \ar@{=}[r] &1\ar[r] &0}
\]

Let $Z$ denote the simplicial set defined by the following pushout:

\centerline{
\xymatrix{
\Lambda^{2}_{1} \ar[r] \ar@{>->}[d] & \Delta^0 \ar[d] \\
\Delta^{2} \ar[r] & Z \pushoutcorner \\
}
}

The map $\Delta^0 \to Z$ is a trivial cofibration, as a pushout of an inner horn inclusion; thus $Z$ is contractible. We have a pair of cofibrations $Z \hookrightarrow TK$, picking out the bottom-left and top-right simplices in the illustration above; the induced map $Z \sqcup Z \to TK$ is a cofibration since these two simplices have no faces in common. We obtain $J$ as a quotient of $TK$ by contracting each of these two simplices to a point; in other words, we have the following pushout diagram:

\centerline{
\xymatrix{
Z \sqcup Z \ar[d] \ar@{>->}[r] & TK \ar[d] \\
\Delta^0 \sqcup \Delta^{0} \ar[r] & J \pushoutcorner \\
}
}

The left map is a weak equivalence since coproducts preserve weak equivalences in the Joyal model structure. Thus $TK \to J$ is a weak equivalence as a pushout of a weak equivalence along a cofibration. The composite of $\Delta^{0} \to TK$ with this quotient map is an endpoint inclusion $\Delta^0 \to J$, hence a weak equivalence; thus $\Delta^0 \to TK$ is a weak equivalence by 2-out-of-3.
\end{proof}

\begin{lemma}\label{U-fib-obs}
$T$ sends inner open box inclusions to trivial cofibrations.
\end{lemma}

\begin{proof}
By \cref{open-box-pop}, \cref{T-pres-pop}, and the symmetry of the cartesian product in $\sSet$, the triangulation of an open box inclusion $\sqcap^{m}_{i,\varepsilon} \hookrightarrow \Box^m$ is the pushout product $(T\bd \Box^{m-1} \hookrightarrow (\Delta^{1})^{m-1}) \hat{\times} (\{\varepsilon\} \hookrightarrow \Delta^1)$. Therefore, since $T$ preserves colimits, the triangulation of $\widehat{\sqcap}^{m}_{i,\varepsilon} \hookrightarrow \widehat{\Box}^m$ is the inclusion of the quotients of these simplicial sets in which the edge corresponding to the critical edge of $\sqcap^{m}_{i,\varepsilon}$ is collapsed to a vertex.

Since $T\bd \Box^{m-1} \hookrightarrow (\Delta^{1})^{m-1}$ is a monomorphism of simplicial sets, it can be written as a composite of boundary fillings. Since pushout products commute with composition, we can thus rewrite $T\sqcap^{m}_{i,\varepsilon} \hookrightarrow (\Delta^1)^m$ as a composite of pushouts of maps of the form $(\bd \Delta^{n} \to \Delta^{n}) \hat{\times} (\{\varepsilon\} \hookrightarrow \Delta^{1})$, i.e. open prism fillings. We can obtain $T\widehat{\Box}^n$ from $T\sqcap^{n}_{i,\varepsilon}$, therefore, by filling the corresponding open prisms in $T\sqcap^{n}_{i,\varepsilon}$. 

Each open prism filling can be explicitly written as a composite of horn fillings. Each of these horn fillings but one will be inner, and hence a trivial cofibration. However, the critical edge of the unique outer horn, i.e. the unique non-degenerate edge containing either the initial or the terminal vertices of both the horn and its missing face, corresponds to the critical edge of $\sqcap^{m}_{i,\varepsilon}$, hence it is degenerate. Thus this horn-filling is also a trivial cofibration by \cite[Thm. 2.2]{joyal:qcat-kan}.
\end{proof}

\begin{proof}[Proof of \cref{T-Quillen-adj}]
This follows from \cref{Quillen-adj-fib-obs}, together with \cref{T-pres-cof,T-endpoints,U-fib-obs}.
\end{proof}

\begin{corollary}\label{T-pres-we}
The triangulation functor preserves weak equivalences.
\end{corollary}

\begin{proof}
Since all cubical sets are cofibrant, this is immediate from \cref{T-Quillen-adj} and Ken Brown's lemma.
\end{proof}

\section{Cones in cubical sets}\label{sec:cones}

In this and the subsequent section, we will prove that the triangulation adjunction $T : \cSet \rightleftarrows \sSet : U$ is a Quillen equivalence between the cubical Joyal and the Joyal model structures.
Working directly with the this adjunction, however, is difficult, as it is hard to describe the counit $TU \Rightarrow \id$ explicitly.

To remedy this issue, we introduce a different adjunction $Q : \sSet \rightleftarrows \cSet : \int$, coming from the straightening-over-the-point functor, as studied in \cite{kapulkin-lindsey-wong,kapulkin-voevodsky:cubical-straightening}, also closely related to Lurie's straightening construction \cite[Ch.~2]{lurie:htt}.
We then show that $Q \adjoint \int$ is a Quillen equivalence and construct a natural weak equivalence $TQ \Rightarrow \id$, from which we derive our conclusion.

As in the case of triangulation, the key difficulty in proving that $Q \adjoint \int$ is a Quillen equivalence lies in understanding the counit map $Q \int X \to X$.
This however is a much more tractable problem as it was for instance shown in \cite{kapulkin-lindsey-wong} that it is a monomorphism.
Intuitively,  $Q \int X$ is the subcomplex of $X$ which is built out of cubes with sufficiently degenerate faces that they may be regarded as simplices, e.g., a square with a single edge collapsed to a point or a cube with one face collapsed to a point and another face collapsed to an edge.

For a cubical quasicategory $X$, we write the inclusion $Q \int X \to X$ as a transfinite composite of pushouts of inner open box fillings, thus establishing it as an anodyne map in the cubical Joyal model structure.
To determine its decomposition into individual open box fillings, we develop a theory of cones in cubical sets.
Roughly speaking, the decomposition proceeds by induction on the dimension of the base of a cone contained in $X$ and in particular, $Q \int X$ consists of cubes obtained by repeatedly taking cones only on the vertices of $X$ (rather than on cubes of arbitrary dimension).
To identify the open boxes needed to build $X$ from $Q \int X$, we introduce the notion of a coherent family of composites, a technical construction that picks out a distinguished cone on each cube of $X$.

The main purpose of this section is to set up the technical machinery needed for the proof that $Q \adjoint \int$ is a Quillen equivalence.
Then in \cref{sec:relation}, we define $Q$ using the theory of cones developed in this section and prove that it is a left Quillen equivalence.
Finally, as indicated above, we construct in \cref{sec:relation} a natural weak equivalence $TQ \Rightarrow \id$, and conclude that $T$ is also a left Quillen equivalence.

One can imagine an alternative approach that would instead proceed by establishing that for some left Quillen functor $F \colon \sSet \to \cSet$ the composite $FT$ is naturally weakly equivalent to the identity, bypassing the technical proof that $Q$ is a left Quillen equivalence. 
While we considered this approach, we were not able to identify a suitable functor $F$; in particular, we could not see a direct natural transformation $QT \Rightarrow \id$ and any zigzag we could think of would involve objects similar to coherent families of composites.
As a result, we opted for the approach presented below as it is both the most straightforward and provides insight into how a cubical quasicategory is built out of its maximal simplicial subcomplex.

We begin this section in \cref{left-positive-def} by defining a cone on a cubical set and showing in \cref{cone-monad} that taking cones defines a monad. 
We then proceed to analyze the faces and subcomplexes of iterated cones on standard cubes in \cref{CmnDef} through \cref{BEqv}.
In \cref{theta-construction}, we define coherent families of composites and show in \cref{theta-exists} that every cubical quasicategory admits such a family.

To define the constructions of this section, only one kind of connection is necessary; thus our arguments apply not only in $\cSet$, but also in $\cSet_0$ and $\cSet_1$, as displayed in the diagram \cref{eq:cube-cat-inclusions} at the end of \cref{sec:background}, after making suitable changes to the definitions.

\begin{definition}\label{left-positive-def}
The \emph{cone functor} $C \colon \cSet \to \cSet$ is defined by the following pushout diagram in $\mathrm{End} \, \cSet$:

\centerline{
\xymatrix{
\id \ar[d]_{\bd_{1,1} \otimes -} \ar[r] & \Box^{0} \ar[d] \\
\Box^{1} \otimes - \ar[r] & C \pushoutcorner \\
}
} 

For $m, n \geq 0$, the \emph{standard $(m,n)$-cone} is the object $C^{m,n} = C^{n}\Box^m$. In particular, for all $m$, $C^{m,0} = \Box^m$. We refer to the natural map $\Box^0 \Rightarrow C$ appearing in this diagram as the \emph{cone point}.
\end{definition}

A simple computation shows:

\begin{lemma}\label{Qcone}
For all $n \geq 1$, $C^{n} \varnothing \cong C^{0,n-1}$ and $C^{0,n} \cong C^{1,n-1}$. \qed
\end{lemma}

To develop our understanding of the cone construction, we consider certain examples of cones $CX$ for $X \in \cSet$. In all of our illustrations, we will denote the cone point of $CX$ by $c$. For our simplest example, we may observe that $C\Box^0 \cong \Box^1$:

\centerline{
\xymatrix{
0 \ar[r] & c \\
}
}

$C \Box^1$ is the quotient of $\Box^2$ depicted below:

\centerline{
\xymatrix{
0 \ar[r] \ar[d] & c \ar@{=}[d] \\
1 \ar[r] & c \\
}
}

For our final example, let $X$ denote the cubical set $0 \to 1 \to 2$. Then $CX$ is the cubical set depicted below:

\centerline{
\xymatrix{
0 \ar[r] \ar[d] & c \ar@{=}[d] \\
1 \ar[r] \ar[d] & c \ar@{=}[d] \\
2 \ar[r] & c \\
}
}

We define the natural transformation $\eta \colon \id \Rightarrow C$ to be the composite of $\bd_{1,0} \otimes - \colon \id \to \Box^1 \otimes -$ with the quotient map $\Box^1 \otimes - \Rightarrow C$. We also define a natural transformation $\mu \colon C^2 \Rightarrow C$ as follows. By the universal property of the pushout, such a natural transformation corresponds to a diagram of the form depicted below:

\centerline{
\xymatrix{
C \ar[r] \ar[d]_{\bd_{1,1} \otimes C} & \Box^0 \ar[d] \\
\Box^1 \otimes C \ar[r] & C \\
}
}

The only natural transformation $\Box^0 \to C$ is the cone point. Now note that $\Box^1 \otimes -$ preserves pushouts as a left adjoint. Thus we may define the map $\Box^1 \otimes C \to C$ as the map between pushouts induced by the following map between diagrams:

\centerline{
\xymatrix{
\Box^1 \otimes \Box^1 \otimes - \ar[d]_{\gamma_{1,0} \otimes -} && \Box^1 \otimes - \ar[ll]_{\bd_{2,1} \otimes -} \ar[d]^{\sigma_1 \otimes -} \ar[rr]^{\pi_{\Box^1}} && \Box^1 \ar[d] \\
\Box^1 \otimes - && \id \ar[ll]_{\bd_{1,1} \otimes -} \ar[rr] && \Box^0 \\
}
}

The commutativity of the left-hand square follows from the cubical identities.

We can also view the natural transformation $\mu$ more concretely, using \cref{geometric-product-description}. For $X \in \cSet, k \geq 0$, a $k$-cube of $\Box^1 \otimes X$ consists of a pair $(f \colon \Box^a \to \Box^1,x \colon \Box^b \to X)$ such that $a + b = k$. The quotient $CX$ is then obtained by identifying cubes $(f,x)$ and $(f',x')$ if $f = f' = \const_1$. Similarly, cubes of $C^{2}X$ consist of pairs $(f_1,f_2,x)$, with $(f_1,f_2,x)$ and $(f'_1,f'_2,x')$ identified if $f_1 = f'_1 = \const_1$ or $f_1 = f'_1$ and $f_2 = f'_2 = \const_1$. Since $\gamma_{1,0} \otimes X$ respects these identifications, it descends to a map $\mu \colon C^{2} X \to CX$.

\begin{proposition}\label{left-positive-monad}
The triple $(C,\eta,\mu)$ defines a monad on $\cSet$.
\end{proposition}

\begin{proof}
The monad laws follow from a straightforward calculation using the cubical identities.
\end{proof}

Given a cubical set $X$, the natural way to form a cone on $X$ is to take its geometric product with the interval $\Box^1$, and quotient one end of the cylinder to a vertex, as was done in \cref{left-positive-def}. This definition, however, involved certain choices: we chose to tensor on the left rather than on the right, and to quotient the subcomplex $\{1\} \otimes X$ rather than $\{0\} \otimes X$. Considering the alternative choices, we obtain four distinct cone functors. In general, we will work with the functor $C$ of \cref{left-positive-def}; when the potential for ambiguity arises, we will refer to this functor as $C_{L,1}$.

\begin{definition}\label{all-cones-def}
We define the \emph{left negative, left positive, right negative,} and \emph{right positive cone functors}, denoted $C_{L,0}, C_{L,1}, C_{R,0}, C_{R,1} \colon \cSet \to \cSet$, respectively, by the following pushout diagrams in $\mathrm{End} \, \cSet$:

\begin{footnotesize}
\centerline{
\xymatrix{
\id \ar[d]_{\bd_{1,0} \otimes -} \ar[r] & \Box^{0} \ar[d] & \id \ar[d]_{\bd_{1,1} \otimes -} \ar[r] & \Box^{0} \ar[d] & \id \ar[d]_{- \otimes \bd_{1,0}} \ar[r] & \Box^{0} \ar[d] & \id \ar[d]_{- \otimes \bd_{1,1}} \ar[r] & \Box^{0} \ar[d] \\
\Box^{1} \otimes - \ar[r] & C_{L,0} \pushoutcorner & \Box^{1} \otimes - \ar[r] & C_{L,1} \pushoutcorner &  - \otimes \Box^{1} \ar[r] & C_{R,0} \pushoutcorner & - \otimes \Box^{1}  \ar[r] & C_{R,1} \pushoutcorner \\
}
} 
\end{footnotesize}
\end{definition}

To understand the differences between these definitions, we illustrate the cubical sets $C_{W,\varepsilon} \Box^1$ for $W \in \{L,R\}, \varepsilon \in \{0,1\}$. These are the four quotients which can be obtained from $\Box^2$ by collapsing one of its faces to a vertex.

\centerline{
\xymatrix{
c \ar[r] \ar@{=}[d] & 0 \ar[d] & 0 \ar[d] \ar[r] & c \ar@{=}[d] & c \ar@{=}[r] \ar[d] & c \ar[d] & 0 \ar[d] \ar[r] & 1 \ar[d] \\
c \ar[r] & 1 & 1 \ar[r] & c & 0 \ar[r] & 1 & c \ar@{=}[r] & c \\
\ar@{}[r]^{C_{L,0} \Box^1} & & \ar@{}[r]^{C_{L,1}\Box^1} & & \ar@{}[r]^{C_{R,0}\Box^1} &  & \ar@{}[r]^{C_{R,1}\Box^1} & \\ 
}
}

Applying the involutions $(-)^\rmco, (-)^\coop, (-)^\op$ to the pushout diagrams of \cref{all-cones-def}, and using \cref{op-co-monoidal}, we obtain the following result, which shows that any one of these cone concepts suffices to describe all the others.

\begin{lemma}\label{cone-involutions}
The functors $C_{W,\varepsilon}$ for $W \in \{L,R\}, \varepsilon \in \{0,1\}$ are related by the following formulas:

\begin{itemize}
\item $C_{L,0} = (-)^\coop \circ C_{L,1} \circ (-)^\coop$;
\item $C_{R,0} = (-)^\op \circ C_{L,1} \circ (-)^\op$;
\item $C_{R,1} = (-)^\rmco \circ C_{L,1} \circ (-)^\rmco$. \qed
\end{itemize}
\end{lemma}

\begin{proposition}\label{cone-monad}
For $W \in \{L,R\}, \varepsilon \in \{0,1\}$, the functor $C_{W,\varepsilon} \colon \cSet \to \cSet$ admits the structure of a monad, with the unit $\eta$ and multiplication $\mu$ induced by natural transformations $\id \Rightarrow I_{W}$ and $I_{W}^2 \Rightarrow I_{W}$, where $I_L, I_R \colon \cSet \to \cSet$ are functors given by $I_L X = \Box^1 \otimes X$ and $I_R X = X \otimes \Box^1$, as follows:

\renewcommand{\arraystretch}{1.7}
\centerline{
$\begin{array}{c c c c c}
\mathrm{endofunctor} & &\mathrm{unit} & & \mathrm{multiplication} \\
\hline
C_{L,0} & & \bd_{1,1} \otimes - & & \gamma_{1,1} \otimes - \\
C_{L,1} & & \bd_{1,0} \otimes - & & \gamma_{1,0} \otimes - \\
C_{R,0} & & - \otimes \bd_{1,1} & & - \otimes \gamma_{1,1} \\
C_{R,1} & & - \otimes \bd_{1,0} & & - \otimes \gamma_{1,0} \\ 
\end{array}$
}
\end{proposition}

\begin{proof}
This follows from \cref{left-positive-monad,cone-involutions}.
\end{proof}

In order to express the counit $Q \int X \to X$ (for a cubical quasicategory $X$) as a transfinite composite of anodyne maps, we will need to analyze the images of standard cones in cubical quasicategories.

For the remainder of this section, we will work exclusively with left positive cones, with the understanding that our results may be adapted to any of the other three varieties of cones using the formulas of \cref{cone-involutions}.

\begin{definition} \label{CmnDef}
For $m, n \geq 0$, an \emph{$(m,n)$-cone} in a cubical set $X$ is a map $C^{m,n} \to X$. 
\end{definition}

Observe that each cone $C^{m,n} \to X$ corresponds to a unique $(m+n)$-cube of $X$ by pre-composition with the quotient map $\Box^{m+n} \to C^{m,n}$. Thus we will also use the term ``$(m,n)$-cone'' to refer to a map $\Box^{m+n} \to X$ which factors through this quotient map. In particular, when we refer to the $(i,\varepsilon)$-face of a cone $x$, this means the $(i,\varepsilon)$-face of the corresponding cube: $\Box^{m+n-1} \xrightarrow{\partial_{i,\varepsilon}} \Box^{m+n} \to C^{m,n} \xrightarrow{x} X$.

For $m,n,k \geq 0$, recall that $\Box^{m+n}_{k}$ is the set of maps $[1]^{k} \to [1]^{m+n}$ in the box category $\Box$; thus we may write such a $k$-cube $f$ as $(f_{1},...,f_{m+n})$ where each $f_{i}$ is a map $[1]^{k} \to [1]$. This allows us to describe $C^{m,n}$ explicitly as a quotient of $\Box^{m+n}$.

\begin{lemma}\label{ConeDesc}
For all $m, n \geq 0, C^{m,n}$ is the quotient of $\Box^{m+n}$ obtained by identifying two $k$-cubes $f, g$ if there exists $j$ with $1 \leq j \leq n$ such that $f_{i} = g_{i}$ for $i \leq j$ and $f_{j} = g_{j} = \mathrm{const}_{1}$ (the constant map $[1]^{k} \to [1]$ with value 1).
\end{lemma}

\begin{proof}
We fix $m$ and proceed by induction on $n$. For the base case $n = 0$, there cannot exist any $j$ satisfying the given criteria, thus no identifications are to be made; and indeed we have $C^{m,0} = \Box^{m}$ by definition.

Now suppose that the given description holds for $C^{m,n}$, and let $q$ denote the quotient map $\Box^{m+n} \to C^{m,n}$. Then because the functor $\Box^{1} \otimes -$ preserves colimits, $\Box^{1} \otimes C^{m,n}$ is a quotient of $\Box^{1+m+n}$ with quotient map $\Box^{1} \otimes q$. From this description we see that $\Box^{1} \otimes C^{m,n}$ is obtained from $\Box^{1+m+n}$ by identifying two $k$-cubes $f,g$ whenever $f_{1} = g_{1}$ and the cubes $(f_{2},...,f_{n+1})$ and $(g_{2},...,g_{n+1})$ are identified in $C^{m,n}$. In other words, we obtain $\Box^{1} \otimes C^{m,n}$ from $\Box^{1+m+n}$ by identifying $f$ and $g$ if there exists $j$ with $2 \leq j \leq n + 1$ such that $f_{i} = g_{i}$ for all $i \leq j$ and $f_{j} = g_{j} = \mathrm{const}_{1}$. Taking the pushout of the inclusion $\partial_{1,1} \otimes C^{m,n} \colon C^{m,n} \hookrightarrow \Box^{1} \otimes C^{m,n}$ along the unique map $C^{m,n} \to \Box^{0}$, we then see that $C^{m,n+1}$ is the quotient of $\Box^{1} \otimes C^{m,n}$ obtained by identifying cubes $f,g$ whenever $f_{1} = g_{1} = \mathrm{const}_{1}$. Thus the description holds for $C^{m,n+1}$.
\end{proof}

\begin{corollary}\label{ConeWLOG}
For $k \leq n$, the quotient map $\Box^{m+n} \to C^{m,n}$ factors through $C^{m+k,n-k}$. In particular, if $x \colon \Box^{m+n} \to X$ is an $(m,n)$-cone, then $x$ is also an $(m+k,n-k)$-cone for all $k \leq n$. \qed
\end{corollary}

Using the characterization of cones given above, we can show that any face of a given cone is a cone of a specified degree.

\begin{lemma}\label{FaceIso}
For $i \leq n$, the image of the composite map $\Box^{m+n-1} \xrightarrow{\partial_{i,0}} \Box^{m+n} \to C^{m,n}$ is isomorphic to $C^{m,n-1}$. For $i \geq n + 1$, $\varepsilon \in \{0,1\}$, the image of the composite map $\Box^{m+n-1} \xrightarrow{\partial_{i,\varepsilon}} \Box^{m+n} \to C^{m,n}$ is isomorphic to $C^{m-1,n}$.
\end{lemma}

\begin{proof}
First consider the composite map $\Box^{m+n-1} \xrightarrow{\partial_{i,0}} \Box^{m+n} \to C^{m,n}$. Let $f = (f_{1},...,f_{m+n-1})$ denote a $k$-cube of $\Box^{m+n-1}$, as in the proof of \cref{ConeDesc}. We denote the image of this cube under $\partial_{i,0}$ by $f' = (f'_{1},...,f'_{m+n-1})$, where $f'_{j} = f_{j}$ for $j < i$, $f'_{i} = \mathrm{const}_{0}$, and $f'_{j} = f_{j-1}$ for $j > i$. By \cref{ConeDesc}, given two $k$-cubes $f$ and $g$ in $\Box^{m+n-1}$, their images under $\partial_{i,0}$ will be identified in the quotient $C^{m,n}$ if and only if there exists $j \leq n$ such that $f'_{l} = g'_{l}$ for $l \leq j$ and $f'_{j} = g'_{j} = \mathrm{const}_{1}$ -- in other words, if there exists $j \leq n - 1$ such that $f_{l} = g_{l}$ for $l \leq j$ and $f_{j} = g_{j} = \mathrm{const}_{1}$. The desired isomorphism thus follows from \cref{ConeDesc}.

The analysis of $\partial_{i,\varepsilon}$ where $i \geq n + 1$, $\varepsilon \in \{0,1\}$ is similar, except that in that case we have $f'_{j} = f_{j}$ for all $j \leq i$. Thus we conclude that the images of $f$ and $g$ in the quotient $C^{m,n}$ are equal if and only if there exists $j \leq n$ such that $f_{l} = g_{l}$ for $l \leq j$ and $f_{j} = g_{j} = \mathrm{const}_{1}$.
\end{proof}

Using \cref{FaceIso} and further computations, we can analyze the effect of cubical structure maps on cones.

\begin{lemma}\label{ConeFaceDeg}
Let $x$ be an $(m,n)$-cone in a cubical set $X$. Then:

\begin{enumerate}
\item\label{Low0FaceOfCone} for $1 \leq i \leq n$, $x\bd_{i,0}$ is an $(m,n-1)$-cone;
\item\label{High0FaceOfCone} for $n+1 \leq i \leq m + n$, $x\bd_{i,0}$ is an $(m-1,n)$-cone;
\item\label{1FaceOfCone} if $m \geq 1$, then for all $i$, $x\bd_{i,1}$ is an $(m-1,n)$-cone;
\item\label{DegenOfCone} for $n+1 \leq i \leq m + n + 1$, $x\sigma_{i} $ is an $(m+1,n)$-cone;
\item\label{Low0ConOfCone} for $1 \leq i \leq n$, $x\gamma_{i,0}$ is an $(m,n+1)$-cone;
\item\label{HighConOfCone} for $n+1 \leq i \leq m + n$, $x\gamma_{i,\varepsilon}$ is an $(m+1,n)$-cone.
\end{enumerate}
\end{lemma} 

\begin{proof}
First consider \cref{Low0FaceOfCone}. By \cref{FaceIso}, we have a commuting diagram as shown below:

\centerline{
\xymatrix{
\Box^{m+n-1} \ar[d] \ar@{^(->}[r]^-{\partial_{i,0}} & \Box^{m+n} \ar[d] \\
C^{m,n-1} \ar@{^(->}[r] & C^{m,n} \\
}
}

Now, for an $(m+n)$-cube $x \in X_{m+n}$ to be an $(m,n)$-cone means precisely that the corresponding map $x \colon \Box^{m+n} \to X$ factors through $C^{m,n}$. So the face $x\bd_{i,0}$ can be written as $\Box^{m+n-1} \xrightarrow{\partial_{i,0}} \Box^{m+n} \to C^{m,n} \xrightarrow{x} X$; by the diagram above we can rewrite this as $\Box^{m+n-1} \to C^{m,n-1} \to C^{m,n} \xrightarrow{x} X$. So $x\bd_{i,0}$ factors through $C^{m,n-1}$, meaning that it is an $(m,n-1)$-cone.

Similar commuting diagrams can be used to prove the remaining statements. For \cref{High0FaceOfCone} we may again apply \cref{FaceIso}; the other statements require new computations. We will show these computations for \cref{1FaceOfCone}; the others are similar.

Let $ m \geq 1, i \leq n$ and consider the composite $\Box^{m+n-1} \xrightarrow{\partial_{i,1}} \Box^{m+n} \to C^{m,n}$. As in the proof of \cref{FaceIso}, we let $f$ denote an arbitrary $k$-cube of $\Box^{m+n-1}$ and let $f'$ denote its image under $\partial_{i,1}$; then once again we have $f'_{j} = f_{j}$ for $j \leq i - 1$, but now $f'_{i} = \mathrm{const}_{1}$. So let $f$ and $g$ be two $k$-cubes of $\Box^{m+n-1}$, and suppose that there exists $j \leq n$ such that $f_{l} = g_{l}$ for $l \leq j$ and $ f_{j} = g_{j} = \mathrm{const}_{1}$. Then there exists $j' \leq n$ such that $f'_{l} = g'_{l}$ for $l \leq j'$ and $f'_{j'} = g'_{j'} = \mathrm{const}_{1}$: if $j < i$ then $j' = j$, while if $j \geq i$ then $j' = i$. So $f'$ and $g'$ are identified in $C^{m,n}$. Thus the composite map factors through $C^{m-1,n}$, i.e. we have a commuting diagram:

\centerline{
\xymatrix{
\Box^{m+n-1} \ar[d] \ar@{^(->}[r]^-{\partial_{i,1}} & \Box^{m+n} \ar[d] \\
C^{m-1,n} \ar@{^(->}[r] & C^{m,n} \\
}
}

So for any $(m,n)$-cone $x$, $x\bd_{i,1}$ is an $(m-1,n)$-cone.
\end{proof}

\begin{corollary}\label{QFace}
For $n \geq 1$, every face of a $(0,n)$-cone is a $(0,n-1)$-cone.
\end{corollary}

\begin{proof}
Let $x \colon \Box^n \to X$ be a $(0,n)$-cone, and consider a face $x \bd_{i,\varepsilon}$. If $i \leq n, \varepsilon = 0$, then $x \bd_{i,\varepsilon}$ is a $(0,n-1)$-cone by \cref{ConeFaceDeg} \ref{Low0FaceOfCone}. Otherwise, we may note that $x$ is a $(1,n-1)$-cone by \cref{ConeWLOG}, and therefore $x \bd_{i,\varepsilon}$ is a $(0,n-1)$-cone by \cref{ConeFaceDeg} \ref{High0FaceOfCone} or \ref{1FaceOfCone}. 
\end{proof}

\begin{corollary}\label{cn}
For $m \geq 1, n \geq 0$, let $x$ be an $(m+n-1)$-cube in a cubical set $X$. If $x\gamma_{n,0}$ is an $(m,n)$-cone, then it is also an $(m-1,n+1)$-cone.
\end{corollary}

\begin{proof}
By \cref{ConeFaceDeg} \ref{High0FaceOfCone}, $x\gamma_{n,0}\bd_{n+1,0} = x$ is an $(m-1,n)$-cone. Therefore, $x\gamma_{n,0}$ is an $(m-1,n+1)$-cone by \cref{ConeFaceDeg} \ref{Low0ConOfCone}.
\end{proof}

In some cases it will be more convenient to characterize cones in a cubical set by a set of conditions on their faces. By a direct analysis of the cubes of $C^{m,n}$, or by an inductive argument similar to that used in the proof of \cref{ConeDesc}, we have the following characterization of $(m,n)$-cones in $X$.

\begin{lemma}\label{FaceCond}
For $m, n$ with $n \geq 1$, and $X \in cSet$, a cube $x \colon\Box^{m+n} \to X$ is an $(m,n)$-cone if and only if for all $i$ such that $1 \leq i \leq n$ we have 

$$x\bd_{i,1} = x\bd_{m+n,0}\bd_{m+n-1,0}...\bd_{i+1,0}\bd_{i,1}\sigma_{i}\sigma_{i+1}...\sigma_{m+n-2}\sigma_{m+n-1}$$

(In the case $m = 0, i = n$ we interpret this statement as the tautology $x\bd_{n,1} = x\bd_{n,1}$). \qed
\end{lemma}

This formula is admittedly somewhat ad hoc; various other choices of face maps would give equivalent conditions. The intuition behind it is that for $1 \leq i \leq n$, the $(i,1)$-face of an $(m,n)$-cone is degenerate in dimensions $i + 1$ to $m + n$.

We will also have use for the following result, which shows that the standard cones contain many inner open boxes.

\begin{lemma}\label{ConeEdge}
For $n \geq 1, 2 \leq i \leq m+n$, the quotient map $\Box^{m+n} \to C^{m,n}$ sends the critical edge with respect to the face $\partial_{i,0}$ to a degenerate edge. 
\end{lemma} 

\begin{proof}
The critical edge in question corresponds to the function $f \colon [1] \to [1]^{m+n}$ with $f_{i} = \mathrm{id}_{[1]}$, $f_{j} = \mathrm{const}_{1}$ for $j \neq i$. In particular, $f_{1} = \mathrm{const}_{1}$, so $f$ is equivalent, under the equivalence relation of \cref{ConeDesc}, to the map $[1] \to [1]^{m+n}$ which is constant at $(1,...,1)$.
\end{proof}

We now prove a lemma regarding the standard forms of cones.

\begin{lemma}\label{sa1}
Let $m \geq 1$, and let $x \colon C^{m,n} \to X$ be a degenerate $(m,n)$-cone.

\begin{enumerate}
\item\label{sa1-degen} If the standard form of $x$ is $z\sigma_{a_{p}}$, then $a_{p} \geq n + 1$.

\item\label{sa1-con} If the standard form of $x$ is $z\gamma_{b_{q},1}$, then $b_{q} \geq n+1$.
\end{enumerate}
\end{lemma}

\begin{proof}
For $n = 0$ these statements are trivial, so assume $n \geq 1$. We will prove \cref{sa1-degen}; the proof for \cref{sa1-con} is similar. 

Towards a contradiction, suppose that $a_{p} \leq n$, and let $z = y\gamma_{b_{1},\varepsilon_{1}}...\gamma_{b_{q},\varepsilon_{q}}\sigma_{a_{1}}...\sigma_{a_{p-1}}$, so that $z\sigma_{a_{p}} = x$. Taking the $(a_{p},1)$-faces of both sides of this equation, and applying \cref{FaceCond}, we see that:

\begin{align*}
z & = x\partial_{m+n,0}...\partial_{a_{p}+1,0}\partial_{a_{p},1}\sigma_{a_{p}}...\sigma_{m+n-1} \\
\therefore z\sigma_{a_{p}} & = x\partial_{m+n,0}...\partial_{a_{p}+1,0}\partial_{a_{p},1}\sigma_{a_{p}}...\sigma_{m+n-1}\sigma_{a_{p}} \\
\therefore x & = x\partial_{m+n,0}...\partial_{a_{p}+1,0}\partial_{a_{p},1}\sigma_{a_{p}}...\sigma_{m+n} \\
\end{align*}

In the last step, we have repeatedly used the identity $\sigma_{j}\sigma_{i} = \sigma_{i}\sigma_{j+1}$ for $i \leq j$ to rearrange the string $\sigma_{a_{p}}...\sigma_{m+n-1,1}\sigma_{a_{p}}$ into one whose indices are in strictly increasing order. (We can do this because, by our assumption on $m$, $m+n-1 \geq n \geq a_p$.) Now let $y'\gamma_{b'_{1},\varepsilon'_{1}}...\gamma_{b'_{q'},\varepsilon'_{q'}}\sigma_{a'_{1}}...\sigma_{a'_{p'}}$ be the standard form of $x\partial_{m+n,0}...\partial_{a_{p}+1,0}\partial_{a_{p},1}$; then we have:

$$
x = y'\gamma_{b'_{1},\varepsilon'_{1}}...\gamma_{b'_{q'},\varepsilon'_{q'}}\sigma_{a'_{1}}...\sigma_{a'_{p'}}\sigma_{a_{p}}...\sigma_{m+n}
$$

We can apply further identities to re-order the maps on the right-hand side of this equation, obtaining a standard form for $x$ in which the rightmost degeneracy map has index greater than or equal to $m+n$. But as the standard form of $x$ is unique, this contradicts our assumption that $a_{p} \leq n$.
\end{proof}

\begin{corollary}\label{Q-stand-form}
Let $x \colon \Box^n \to X$. If $x$ is a $(0,n)$-cone, then the standard form of $x$ contains no positive connection maps.
\end{corollary}

\begin{proof}
Let $x = y\gamma_{b_{1},\varepsilon_{1}}...\gamma_{b_{q},\varepsilon_{q}}\sigma_{a_{1}}...\sigma_{a_{p}}$ in standard form. Towards a contradiction, suppose that there exists $1 \leq i \leq q$ such that $\varepsilon_i = 1$. By repeatedly applying face maps and using \cref{QFace}, we see that $y\gamma_{b_{1},\varepsilon_{1}}...\gamma_{b_{i},1}$ is a $(0,n-p-q+i)$-cone. \cref{sa1} \ref{sa1-con} thus implies that  $b_i \geq n - p - q + i + 1$. But $\gamma_{b_{i},1}$ is a map $[1]^{n-p-q+i} \to [1]^{n-p-q+i-1}$, implying $b_i \leq n-p-q+i-1$.
\end{proof}

Before turning our attention to coherent families of composites, we introduce certain subcomplexes of the standard cones, which will be useful in constructing coherent families of composites.


\begin{definition}
For $m, n \geq 0, n \leq k \leq m + n - 1$, $B^{m,n,k}$ is the subcomplex of $C^{m,n}$ consisting of the images of the faces $\partial_{1,0}$ through $\partial_{k,0}$, as well as all all faces $\partial_{i,1}$, under the quotient map $\Box^{m+n} \to C^{m,n}$.
\end{definition}

In order to characterize maps out of $B^{m,n,k}$, we will need to prove a couple of lemmas concerning the faces of $C^{m,n}$.

\begin{lemma}\label{FaceInt}
For $m, n \geq 0, 1 \leq i_{1} < i_{2} \leq m + n, \varepsilon_{1}, \varepsilon_{2} \in \{0,1\}$, where $i_{j} \geq n+ 1$ if $\varepsilon_{j} = 1$, the intersection of the images of the faces $\partial_{i_{1},\varepsilon_{1}}$ and $\partial_{i_{2},\varepsilon_{2}}$ of $\Box^{m+n}$ under the quotient map $\Box^{m+n} \to C^{m,n}$ is exactly the image of the face $\partial_{i_{2},\varepsilon_{2}}\partial_{i_{1},\varepsilon_{1}} = \partial_{i_{1},\varepsilon_{1}}\partial_{i_{2}-1,\varepsilon_{2}}$.
\end{lemma}

\begin{proof}
That the intersection of the images of $\partial_{i_{1},\varepsilon_{1}}$ and $\partial_{i_{2},\varepsilon_{2}}$ contains the image of $\partial_{i_{2},\varepsilon_{2}}\partial_{i_{1},\varepsilon_{1}}$ follows from the fact that this face is the intersection of $\partial_{i_{1},\varepsilon_{1}}$ and $\partial_{i_{2},\varepsilon_{2}}$ in $\Box^{m+n}$. Now we will verify the opposite containment, using description of $C^{m,n}$ from \cref{ConeDesc}. 

To this end, consider a map $f \colon [1]^{k} \to [1]^{m+n}$ such that the equivalence class $[f] \in C^{m,n}_{k}$ is contained in the images of faces $(i_{1},\varepsilon_{1})$ and $(i_{2},\varepsilon_{2})$. We will  construct $f' \colon [1]^{k} \to [1]^{m+n}$ such that $f \sim f'$ and $f'$ is contained in the intersection of faces $(i_{1},\varepsilon_{1})$ and $(i_{2},\varepsilon_{2})$, thereby showing that $[f] = [f']$ is contained in the image of this intersection under the quotient map.

Since $f$ is in the image of face $(i_{1},\varepsilon_{1})$, $f \sim g$ for some $g \colon [1]^{k} \to [1]^{m+n}$ such that $g_{i_{1}} = \mathrm{const}_{\varepsilon_{1}}$. Therefore, at least one of the following holds:

\begin{enumerate}
\item $f_{i_{1}} = \mathrm{const}_{\varepsilon_{1}}$;
\item $f_{j} = g_{j} = \mathrm{const}_{1}$ for some $j \leq \mathrm{min}(i_{1}-1,n)$.
\end{enumerate} 

If (ii) holds, then $f$ is equivalent to any $f'$ such that $f'_{l} = f_{l}$ for $l \leq j$; in particular, we can choose such an $f'$ satisfying $f'_{i_{1}} = \mathrm{const}_{\varepsilon_{1}}, f'_{i_{2}} = \mathrm{const}_{\varepsilon_{2}}$.

Now suppose that (i) holds, but (ii) does not. Then because $f$ is in the image of face $(i_{2},\varepsilon_{2})$, $f \sim h$ for some $h \colon [1]^{k} \to [1]^{m+n}$ such that $h_{i_{2}} = \mathrm{const}_{\varepsilon_{2}}$. Therefore, at least one of the following holds:

\begin{enumerate}
\item $f_{i_{2}} = \mathrm{const}_{\varepsilon_{2}}$;
\item $f_{j} = h_{j} = \mathrm{const}_{1}$ for some $i_{1}+1 \leq j \leq \mathrm{min}(i_{2}-1,n)$.
\end{enumerate}

In case (i), we have $f_{i_{1}} = \mathrm{const}_{\varepsilon_{1}}, f_{i_{2}} = \mathrm{const}_{\varepsilon_{2}}$, so we can simply choose $f' = f$. In case (ii), $f$ is equivalent to any $f'$ such that $f'_{l} = f_{l}$ for $l \leq j$ (which implies $f'_{i_{1}} = \mathrm{const}_{\varepsilon_{1}}$); in particular, we can choose such an $f'$ satisfying $f'_{i_{2}} = \mathrm{const}_{\varepsilon_{2}}$.
\end{proof}

\begin{lemma}\label{Contain}
For $i \leq n$, the image of the face $\partial_{i,1}$ under the quotient map $\Box^{m+n} \to C^{m,n}$ is contained in the image of $\partial_{m+n,1}$. 
\end{lemma}

\begin{proof}
Let $f \colon [1]^{k} \to [1]^{m+n}$ be a $k$-cube of $\Box^{m+n}$ which factors through $\partial_{i,1}$. Then $f_{i} = \mathrm{const}_{1}$. Thus $f$ is equivalent to any $f' \colon [1]^{k} \to [1]^{m+n}$ such that $f'_{j} = f_{j}$ for all $j \leq i$; in particular, we may choose such an $f'$ with $f'_{m+n} = \mathrm{const}_{1}$. So $f'$ factors through $\partial_{m+n,1}$; thus $[f] = [f']$ is contained in the image of $\partial_{m+n,1}$ under the quotient map.
\end{proof}

Note that \cref{Contain} can also be seen as a consequence of \cref{FaceCond}.

\begin{lemma}\label{BMap}
For a cubical set $X$, a map $x \colon B^{m,n,n} \to X$ is determined by a set of $(m,n-1)$-cones $x_{i,0} \colon C^{m,n-1} \to X$ for $1 \leq i \leq n$ and a set of $(m-1,n)$-cones $x_{i,1}$ for $n+1 \leq i \leq m+n$ such that for all $i_{1} < i_{2}, \varepsilon_{1}, \varepsilon_{2} \in \{0,1\}, x_{i_{2},\varepsilon_{2}}\bd_{i_{1},\varepsilon_{1}} = x_{i_{1},\varepsilon_{1}}\bd_{i_{2}-1,\varepsilon_{2}}$, with $x_{i,\varepsilon}$ being the image of $\partial_{i,\varepsilon}$ under $x$. 
\end{lemma}

\begin{proof}
To define a map $x \colon B^{m,n,k} \to X$, it suffices to assign the values of $x$ on the faces $[\partial_{i,\varepsilon}]$ of $C^{m,n}$ for which $i \leq k$ or $\varepsilon = 1$, provided that these choices are consistent on the intersections of faces. By \cref{Contain}, it suffices to consider only those faces for which $i \leq k, \varepsilon = 0$ or $i \geq n + 1, \varepsilon = 1$. These faces are isomorphic to $C^{m,n-1}$ or $C^{m-1,n}$, respectively, by \cref{FaceIso}. By \cref{FaceInt}, to show that these choices  are consistent on the intersections of faces, it suffices to show that they satisfy the cubical identity for composites of face maps.
\end{proof}

\begin{proposition}\label{BEqv}
For all $m, n \geq 1, n \leq k \leq m + n - 1,$ the inclusion $B^{m,n,k} \hookrightarrow C^{m,n}$ is a trivial cofibration.
\end{proposition}

\begin{proof}
We proceed by induction on $m$. In the base case $m = 1$, the only relevant value of $k$ is $k = n$. The only face of $C^{1,n}$ which is missing from $B^{1,n,n}$ is $[\partial_{n+1,0}]$, so the inclusion $B^{1,n,n} \hookrightarrow C^{1,n}$ is an $(n+1,0)$-open box filling. By \cref{ConeEdge}, the critical edge for this open box filling is degenerate, so the inclusion is a trivial cofibration.

Now let $m \geq 2$, and suppose the statement holds for $m-1$. For $n \leq k \leq m + n - 2$, consider the intersection of the $(k+1,0)$-face of $C^{m,n}$, $[\partial_{k,0}]$, with the subcomplex $B^{m,n,k}$. By \cref{FaceInt} and \cref{Contain}, this intersection consists of faces $(1,0)$ through $(k,0)$ and $(1,1)$ through $(m+n-1,1)$ of $[\partial_{k+1,0}]$. By \cref{FaceIso}, it is thus isomorphic to $B^{m-1,n,k}$.

Thus we can express $B^{m,n,k+1}$ as the following pushout:

\centerline{
\xymatrix{
B^{m-1,n,k} \ar@{^(->}[r] \ar@{^(->}[d] & B^{m,n,k} \ar@{^(->}[d] \\
C^{m-1,n} \ar@{^(->}[r] & B^{m,n,k+1} \pushoutcorner \\
}
}

By the induction hypothesis, $B^{m,n,k} \hookrightarrow C^{m-1,n}$ is a trivial cofibration, since $n \leq k \leq m + n - 2$. Thus $B^{m,n,k} \hookrightarrow B^{m,n,k+1}$ is a trivial cofibration, as a pushout of a trivial cofibration. From this we can see that for any $n \leq k \leq m + n - 2$, the composite inclusion $B^{m,n,k} \hookrightarrow B^{m,n,k+1} \hookrightarrow ... \hookrightarrow B^{m,n,m+n-1}$ is a trivial cofibration.

Thus it suffices to prove that $B^{m,n,m+n-1} \hookrightarrow C^{m,n}$ is a trivial cofibration. Here, as in the base case, the subcomplex $B^{m,n,m+n-1}$ is only missing the face $[\partial_{m+n,0}]$, so the inclusion is an $(m+n,0)$-open box filling. The critical edge of this open box is degenerate by \cref{ConeEdge}, so the inclusion is indeed a trivial cofibration.

Thus we see that the inclusion $B^{m,n,k} \hookrightarrow C^{m,n}$ is a trivial cofibration for any $m,n,k$ satisfying the constraints given in the statement.
\end{proof}

We now turn our attention to coherent families of composites, a technical tool needed to build a cubical quasicategory out of its maximal simplicial subcomplex via inner open box fillings.
To this end, we begin by defining coherent families of composites and then show that every cubical quasicategory admits such a family.

\begin{definition}\label{theta-construction}
A \emph{coherent family of composites} $\theta$ in a cubical quasicategory $X$ consists of a family of functions $\theta^{m,n} \colon \cSet(C^{m,n},X) \to \cSet(C^{m,n+1},X)$ satisfying the following identities:

\begin{enumerate}
\myitem[$(\Theta 1)$]\label{ThetaFace0} for $x \colon C^{m,n} \to X$ and $i \leq n$, $\theta^{m,n}(x)\bd_{i,0} = \theta^{m,n-1}(x\bd_{i,0})$;
\myitem[$(\Theta 2)$]\label{ThetaFace0Id} for $x \colon C^{m,n} \to X$, $\theta^{m,n}(x)\bd_{n+1,0} = x$;
\myitem[$(\Theta 3)$]\label{ThetaFace1} for $x \colon C^{m,n} \to X$ and $i \geq n + 2$, $\theta^{m,n}(x)\bd_{i,1} = \theta^{m-1,n}(x\bd_{i-1,1})$;
\myitem[$(\Theta 4)$]\label{ThetaDegen} for $x \colon C^{m-1,n} \to X$ and $i \geq n + 1$, $\theta^{m,n}(x\sigma_{i}) = \theta^{m-1,n}(x)\sigma_{i+1}$;
\myitem[$(\Theta 5)$]\label{ThetaLowCon} for $x \colon C^{m,n-1} \to X$ and $i \leq n - 1$, $\theta^{m,n}(x\gamma_{i,0}) = \theta^{m,n-1}(x)\gamma_{i,0}$;
\myitem[$(\Theta 6)$]\label{ThetaHighCon} for $x \colon C^{m-1,n} \to X$ and $i \geq n + 1$, then $\theta^{m,n}(x\gamma_{i,\varepsilon}) = \theta^{m-1,n}(x)\gamma_{i+1,\varepsilon}$;
\myitem[$(\Theta 7)$]\label{ThetaTheta} for $x \colon C^{m,n-1} \to X$, $\theta^{m,n}(\theta^{m,n-1}(x)) = \theta^{m,n-1}(x)\gamma_{n,0}$; 
\myitem[$(\Theta 8)$]\label{ThetaConeWLOG} for $x \colon C^{m-1,n+1} \to X$, $\theta^{m,n}(x) = x\gamma_{n+1,0}$.
\end{enumerate}
\end{definition}

The rough intuition behind \cref{theta-construction} is this: thinking of cubes in a cubical quasicategory $X$ as representing diagrams commuting up to homotopy, constructing a coherent family of composites on $X$ amounts to coherently choosing a specific composite edge for each $x \colon \Box^n \to X$. For instance, consider a 2-cube $x$ as depicted below, witnessing $gf \sim qp$:

 \[
\xymatrix{
  a
  \ar[r]^f
  \ar[d]_p
&
  b
  \ar[d]^g
\\
  c
  \ar[r]^q
&
  d
}
\]

Then the identities \ref{ThetaFace0} through \ref{ThetaConeWLOG} imply that $\theta^{2,0}(x)$ is a 3-cube of the form depicted below:

 \[
\xymatrix@!C{
 a
 \ar[rrr]^{s}
 \ar[ddd]_{f}
 \ar[dr]^{p}
&&&
 d
 \ar@{=}[ddd]|!{[dl];[dr]}{\hole}
 \ar@{=}[dr]
\\&
 c
 \ar[rrr]^{q}
 \ar[ddd]_{q}
&&&
 d
 \ar@{=}[ddd]
\\
\\
 b
 \ar[rrr]|!{[uur];[dr]}{\hole}^{g}
 \ar[dr]^{g}
&&&
 d
 \ar@{=}[dr]
\\&
 d
 \ar@{=}[rrr]
&&&
 d
}
\]

The edge $s$ from $a$ to $d$ is homotopic to both composites $gf$ and $qp$.

The remainder of this section is dedicated to proving the following theorem.

\begin{theorem}\label{theta-exists}
Every cubical quasicategory admits a coherent family of composites.
\end{theorem}

To prove this, we will construct the family of functions $\theta^{m,n}$ by induction on $m$ and $n$.

\begin{definition}[Base case]\label{theta-base-def}
For a cubical quasicategory $X$ and $x \colon C^{0,n} \to X$, let $\theta^{0,n}(x) = x \sigma_{n+1}$. For $x \colon C^{1,n} \to X$, let $\theta^{1,n}(x) = x \gamma_{n+1,0}$. 
\end{definition}

These define $(0,n+1)$-cones and $(1,n+1)$-cones, respectively, by \cref{ConeFaceDeg}.

\begin{remark}
While it may appear that these definitions of $\theta^{0,n}$ and $\theta^{1,n}$ were chosen arbitrarily, in fact they are implied by the identities of \cref{theta-construction}. Specifically, the given definition of $\theta^{1,n}$ is implied by \ref{ThetaConeWLOG} and \cref{Qcone}. This, together with \ref{ThetaFace1} and \cref{ConeFaceDeg} \ref{DegenOfCone}, then implies the given definition of $\theta^{0,n}$.
\end{remark}

\begin{lemma}\label{theta-base-case}
For a cubical quasicategory $X$, the families of functions $\theta^{0,n}$ and $\theta^{1,n}$ satisfy the identities of \cref{theta-construction}.
\end{lemma}

\begin{proof}
We first verify the identities for $\theta^{0,n}$. The hypotheses of \ref{ThetaFace1}, \ref{ThetaDegen} and \ref{ThetaHighCon} are vacuous here, as there are no cubical structure maps satisfying the given constraints on their indices; \ref{ThetaConeWLOG} similarly does not apply in this case. The remaining identities follow easily from the cubical identities:

\begin{itemize}
\item For \ref{ThetaFace0}, let $i \leq n$. Then $\theta^{0,n}(x)\bd_{i,0} = x\sigma_{n+1}\bd_{i,0} = x\bd_{i,0}\sigma_{n} = \theta^{0,n-1}(x)\sigma_{n}$.
\item For \ref{ThetaFace0Id}, we have $\theta^{0,n}(x)\bd_{n+1,0} = x\gamma_{n+1,0}\bd_{n+1,0} = x$.
\item For \ref{ThetaLowCon}, let $1 \leq i \leq n - 1$. Then $\theta^{0,n}(x\gamma_{i,0}) = x\gamma_{i,0}\sigma_{n+1} = x\sigma_{n}\gamma_{i,0} = \theta^{0,n-1}(x)\gamma_{i,0}$.
\item For \ref{ThetaTheta}, we have $\theta^{0,n+1}(\theta^{0,n}(x)) = x\sigma_{n+1}\sigma_{n+2} = x\sigma_{n+1}\gamma_{n+1,0} = \theta^{0,n}(x)\gamma_{n+1,0}$. 
\end{itemize}

Next we will verify the identities for $\theta^{1,n}$. Here \ref{ThetaConeWLOG} holds by definition, while the hypothesis of \ref{ThetaHighCon} is still vacuous, as there are no connection maps $\gamma_{i,\varepsilon} \colon [1]^{n} \to [1]^{n-1}$ with $i \geq n + 1$. Once again, we can verify the remaining identities using the cubical identities:

\begin{itemize}
\item For \ref{ThetaFace0}, let $i \leq n$. Then $\theta^{1,n}(x)\bd_{i,0} = x\gamma_{n+1,0}\bd_{i,0} = x\bd_{i,0}\gamma_{n,0} = \theta^{1,n-1}(x\bd_{i,0})$.
\item For \ref{ThetaFace0Id}, we have $\theta^{1,n}(x)\bd_{n+1,0} = x\gamma_{n+1,0}\bd_{n+1,0} = x$.
\item For \ref{ThetaFace1}, we need only consider the case $m' = 1, i = n + 2$. For this case we have $\theta^{1,n}(x)\bd_{n+2,1}  = x\gamma_{n+1,0}\bd_{n+2,1} = x\bd_{n+1,1}\sigma_{n+1}  = \theta^{0,n}(x\bd_{n+1,1})$. 
\item For \ref{ThetaDegen}, the only relevant degeneracy is $\sigma_{n+1}$, and we have $\theta^{1,n}(x\sigma_{n+1}) = x\sigma_{n+1}\gamma_{n+1,0} = x\sigma_{n+1}\sigma_{n+2} = \theta^{0,n}(x)\sigma_{n+2}$.
\item For \ref{ThetaLowCon}, let $1 \leq i \leq n - 1$. Then $\theta^{1,n}(x\gamma_{i,0}) = x\gamma_{i,0}\gamma_{n+1,0} = x\gamma_{n,0}\gamma_{i,0} = \theta^{1,n-1}(x)\gamma_{i,0}$.
\item For \ref{ThetaTheta}, we have $\theta^{1,n+1}(\theta^{1,n}(x)) = x\gamma_{n+1,0}\gamma_{n+2,0} = x\gamma_{n+1,0}\gamma_{n+1,0} = \theta^{1,n}(x)\gamma_{n+1,0}$. \qedhere
\end{itemize}
\end{proof}

The following lemma will be used in defining $\theta^{m,n}$ in the inductive case.

\begin{lemma}\label{theta-lift}
Let $m \geq 2, n \geq 0$, and let $X$ be a cubical quasicategory equipped with functions $\theta^{m,n}$ satisfying the identities of \cref{theta-construction} for all pairs $(m',n')$ such that $m' \leq m$, $n' \leq n$, and at least one of these two inequalities is strict. Then for any $x \colon C^{m,n} \to X$, there exists an $(m,n+1)$-cone $\widetilde{\theta}^{m,n}(x) \colon C^{m,n+1} \to X$ satisfying \ref{ThetaFace0}, \ref{ThetaFace0Id}, and \ref{ThetaFace1}.
\end{lemma}

\begin{proof}
For each $i \leq n$, the face $x\bd_{i,0}$ is an $(m,n-1)$-cone by \cref{ConeFaceDeg} \ref{Low0FaceOfCone}; thus $X$ contains an $(m,n)$-cone $\theta^{m,n-1}(x\bd_{i,0})$. Similarly, for each $i \geq n + 2$, the face $x\bd_{i-1,1}$ is an $(m-1,n)$-cone, and so $X$ contains an $(m-1,n+1)$-cone $\theta^{m-1,n}(x\bd_{i-1,1})$, and these cones satisfy the identities of \cref{theta-construction}. Using \cref{BMap}, we will define a map $y \colon B^{m,n+1,n+1} \to X$ with $y_{i,0} = \theta^{m,n-1}(x\bd_{i,0})$ for $1 \leq i \leq n$, $y_{n+1,0} = x$, and $y_{i,1} = \theta^{m-1,n}(x\bd_{i-1,1})$ for $i \geq n + 2$.

To show that we can define such a map, we must verify that our choices of $y_{i,\varepsilon}$ satisfy the cubical identity for composing face maps.

For $i_{1} < i_{2} \leq n, \varepsilon_{1} = \varepsilon_{2} = 0$, we have:

\begin{align*}
y_{i_{2},0}\bd_{i_{1},0} & = \theta^{m,n-1}(x\bd_{i_{2},0})\bd_{i_{1},0} \\
& = \theta^{m,n-2}(x\bd_{i_2,0}\bd_{i_1,0}) \\
& = \theta^{m,n-2}(x\bd_{i_1,0}\bd_{i_{2}-1,0}) \\
& = \theta^{m,n-1}(x\bd_{i_1,0})\bd_{i_{2}-1,0} \\
& = y_{i_1,0}\bd_{i_{2}-1,0} \\
\end{align*}

For $i_{1} < i_{2} = n + 1$, we have:

\begin{align*}
y_{n+1,0}\bd_{i_{1},0} & = x\bd_{i_1,0} \\
& = \theta^{m,n-1}(x\bd_{i_1,0})\bd_{n,0} \\
& = y_{i_{1},0}\bd_{n,0} \\
\end{align*}

For $n + 1 = i_{1} < i_{2}$ we have:

\begin{align*}
y_{i_{2},1}\bd_{n+1,0} & = \theta^{m-1,n}(x\bd_{i_{2}-1,1})\bd_{n+1,0} \\
& = x\bd_{i_{2}-1,1} \\
& = y_{n+1,0}\bd_{i_{2}-1,1} \\
\end{align*}

Finally, for $n + 2 \leq i_{1} < i_{2}$, we have:

\begin{align*}
y_{i_{2},1}\bd_{i_{1},1} & = \theta^{m-1,n}(x\bd_{i_{2}-1,1})\bd_{i_{1},1} \\
& = \theta^{m-2,n}(x\bd_{i_{2}-1,1}\bd_{i_{1}-1,1}) \\
& = \theta^{m-2,n}(x\bd_{i_{1}-1,1}\bd_{i_{2}-2,1}) \\
& = \theta^{m-1,n}(x\bd_{i_{1}-1,1})\bd_{i_{2}-1,1} \\
& = y_{i_{1},1}\bd_{i_{2}-1,1} \\
\end{align*}

Thus the $(n+1)$-tuple $y$ does indeed define a map $B^{m,n+1,n+1} \to X$. Now consider the following commuting diagram:

\centerline{
\xymatrix{
B^{m,n+1,n+1} \ar[r]^{\hspace{2pc} y} \ar@{^(->}[d]_{\sim} & X \ar[d] \\
C^{m,n+1} \ar[r] & \Box^{0} \\
}
}

The left-hand map is a trivial cofibration by \cref{BEqv}, while the right-hand map is a fibration by assumption. Thus there exists a lift of this diagram, i.e. an $(m,n+1)$-cone $\widetilde{\theta}^{m,n}(x) \colon C^{m,n+1} \to X$ such that for $i \leq n, \widetilde{\theta}^{m,n}(x)\bd_{i,0} = \theta^{m,n-1}(x\bd_{i,0})$, $\widetilde{\theta}^{m,n}(x)\bd_{n+1,0} = x$, and for $i \geq n + 2$, $\widetilde{\theta}^{m,n}(x)\bd_{i,1} = \theta^{m-1,n}(x\bd_{i-1,1})$.
\end{proof}

Although \cref{theta-lift} applies for an arbitrary $(m,n)$-cone $x$ with $m \geq 2$, we will not use it to construct $\theta^{m,n}$ for all such cones, as the arbitrary lift used in its proof may not satisfy \ref{ThetaDegen} through \ref{ThetaConeWLOG}. Instead, we define $\theta^{m,n}$ for $m \geq 2,n \geq 0$ by the following case analysis.

\begin{definition}[Inductive case]\label{ThetaDef}
Let $m \geq 2, n \geq 0$, and let $X$ be a cubical quasicategory equipped with functions $\theta^{m,n}$ satisfying the identities of \cref{theta-construction} for all pairs $(m',n')$ such that $m' \leq m$, $n' \leq n$, and at least one of these two inequalities is strict. Let $x \colon C^{m,n} \to X$ be an $(m,n)$-cone. Then $\theta^{m,n}(x) \colon \Box^{m+n+1} \to X^{m,n}$ is defined as follows:

\begin{enumerate}[label=(\arabic*)]
\item If the standard form of $x$ is $z\sigma_{a_{p}}$ for some $a_{p} \geq n + 1$, then $\theta^{m,n}(x) = \theta^{m-1,n}(z)\sigma_{a_{p}+1}$;
\item If the standard form of $x$ is $z\gamma_{b_{q},0}$ for some $b_{q} \leq n - 1$, then $\theta^{m,n}(x) = \theta^{m,n-1}(z)\gamma_{b_{q},0}$;
\item If the standard form of $x$ is $z\gamma_{b_{q},\varepsilon}$ for some $b_{q} \geq n + 1$, then $\theta^{m,n}(x) = \theta^{m-1,n}(z)\gamma_{b_{q}+1,\varepsilon}$;
\item If $x$ is an $(m-1,n+1)$-cone not covered under any of cases (1) through (3), then $\theta^{m,n}(x) = x\gamma_{n+1,0}$;
\item If $x = \theta^{m,n-1}(x')$ for some $x' \colon C^{m,n-1} \to X$ and $x$ is not covered under any of cases (1) through (4) then $\theta^{m,n}(x) = x\gamma_{n,0}$;
\item If $x$ is not convered under any of cases (1) through (5), then $\theta^{m,n}(x)$ is the cone $\widetilde{\theta}^{m,n}(x)$ constructed in \cref{theta-lift}.
\end{enumerate}
\end{definition}

That each of the constructions of \cref{ThetaDef} produces an $(m,n+1)$-cone can be seen from \cref{ConeWLOG,ConeFaceDeg,theta-lift}. 

Before proving that this definition satisfies the identities of \cref{theta-construction}, we prove some simple lemmas about its cases.

\begin{lemma}\label{theta-cases-degenerate}
Every degenerate cone in a cubical quasicategory $X$ falls under one of cases (1) to (4) of \cref{ThetaDef}.
\end{lemma}

\begin{proof}
This follows from \cref{cn,sa1}.
\end{proof}

\begin{corollary}\label{theta-case-6}
Case (6) of \cref{ThetaDef} consists precisely of those $(m,n)$-cones of $X$ which are:

\begin{itemize}
\item Non-degenerate;
\item Not $(m-1,n+1)$-cones;
\item Not equal to $\theta^{m,n-1}(x)$ for any $x \colon C^{m,n-1} \to X$. \qed
\end{itemize}
\end{corollary}

\begin{lemma}\label{Theta-T}
Let $X$ be a cubical quasicategory, and let $m, n \geq 0$ for which we have defined $\theta^{m,n}$ satisfying the identities of \cref{theta-construction}. Then $x \colon C^{m,n} \to X$ is covered under case (6) of \cref{ThetaDef}, i.e.:

\begin{itemize}
\item $x$ is non-degenerate;
\item $x$ is not an $(m-1,n+1)$-cone;
\item $x$ is not equal to $\theta^{m,n-1}(x')$ for any $x' \colon C^{m,n-1} \to X$;
\end{itemize}

if and only if $\theta^{m,n}(x)$ is covered under case (5), i.e. it is non-degenerate and is not an $(m-1,n+2)$-cone.
\end{lemma}

\begin{proof}
First suppose $x$ is covered under case (6). The cubical identities show that if a degenerate cube $y$ has a non-degenerate face $z$, then $z$ appears as at least two distinct faces of $y$. We have $\theta^{m,n}(x) \bd_{n+1,0} = x$, and $x$ is non-degenerate by assumption, so if $\theta^{m,n}(x)$ is degenerate, then $x$ must appear as at least one other face of $\theta^{m,n}(x)$. However, for $i \leq n$ we have $\theta^{m,n}(x) \bd_{i,0} = \theta^{m,n-1}(x \bd_{i,0})$, while for $i \geq n + 2$ or $\varepsilon = 1$, $\theta^{m,n}(x) \bd_{i,\varepsilon}$ is an $(m-1,n+1)$-cone by \cref{ConeFaceDeg}. Thus none of these faces are equal to $x$, showing that $\theta^{m,n}(x)$ is non-degenerate. Furthermore, $\theta^{m,n}(x)$ is not an $(m-1,n+2)$-cone, as this would imply that $\theta^{m,n}(x) \bd_{n+1,0} = x$ was an $(m-1,n+1)$-cone by \cref{ConeFaceDeg} \ref{Low0FaceOfCone}.

On the other hand, if $x$ is not covered under case (6), then $\theta^{m,n}(x)$ is degenerate, hence covered under one of cases (1) to (4) by \cref{theta-cases-degenerate}.
\end{proof}

The proof that the construction $\theta$ of \cref{ThetaDef} satisfies all of the identities of \cref{theta-construction} involves many elaborate case analyses; for brevity, these calculations have been relegated to Appendix \ref{appendix:calculations}.

\begin{proof}[Proof of \cref{theta-exists}]
The functions $\theta^{m,n}$ are defined inductively by \cref{theta-base-def,ThetaDef}. That this definition satisfies all the given identities is proven in \cref{AB,IdC,CDE,IdF,IdG}.
\end{proof}

The following lemma will be useful in various proofs involving coherent families of composites.

\begin{lemma}\label{ThetaDegEdge}
Let $X$ be a cubical quasicategory equipped with a coherent family of composites $\theta$. For $m \geq 0$ and $x \colon \Box^{m} \to X$, the critical edge of $\theta^{m,0}(x) \colon \Box^{m+1} \to X$ with respect to its $(1,0)$-face is degenerate.
\end{lemma}

\begin{proof}
We proceed by induction on $m$. For $m = 0$, we have $\theta^{0,0}(x) = x\sigma_{1}$; so $\theta^{0,0}(x)$ is a degeneracy of a vertex, thus its unique edge is degenerate. 

Now let $m \geq 1$, and suppose that the statement holds for $m - 1$. The edge in question may be written as $\theta^{m,0}(x)\bd_{m+1,1}...\bd_{3,1}\bd_{2,1}$. By \ref{ThetaFace1}, this is equal to $\theta^{m-1,0}(x\bd_{m,1})\bd_{m,1}...\bd_{2,1}$, which is degenerate by the induction hypothesis.
\end{proof}

\section{Comparison with the Joyal model structure} \label{sec:relation}

In this section we use the theory of cones developed in \cref{sec:cones} to compare the cubical Joyal model structure with the Joyal model structure on $\sSet$, showing that the model structures constructed in \cref{section:marked,sec:structurally-marked,sec:joyal-cset} present the theory of $(\infty,1)$-categories. As in \cref{sec:cones}, our results can be adapted to $\cSet_0$ and $\cSet_1$ by replacing the constructions of \cref{sec:cones} with their analogues in the appropriate settings.

Our main goal is to prove the following:

\begin{theorem}\label{T-Quillen-equivalence}
The adjunction $T : \cSet \rightleftarrows \sSet : U$ is a Quillen equivalence between the cubical Joyal model structure on $\cSet$ and the Joyal model structure on $\sSet$. 
\end{theorem}

Throughout this section, $\sSet$ and $\cSet$ will be equipped with the Joyal and cubical Joyal model structures, respectively, unless otherwise noted.  

As indicated in \cref{sec:cones}, due to the difficulty of working directly with the triangulation functor, we first establish a second Quillen adjunction $Q : \sSet \rightleftarrows \cSet : \int$; this adjunction, known as straightening-over-the-point,  was previously studied in \cite{kapulkin-lindsey-wong,kapulkin-voevodsky:cubical-straightening} (for specific choices of the box category), but here we will construct it using the theory of cones developed in \cref{sec:cones}. 
In the simplicial setting, a similar adjunction was developed by Lurie \cite[Ch.~2]{lurie:htt}.

We will prove that $Q \adjoint \int$ is a Quillen equivalence, and that the left derived functor of $Q$ is an inverse to that of $T$.
To give a construction of $Q$ in terms of cones, we first recall a folklore result about constructing cosimplicial objects out of monads.

\begin{proposition}\label{monad-cosimplicial-object}
Let $M$ be a monad on a category $\catC$. Then $M$ induces an augmented cosimplicial object $\Delta_{\mathrm{aug}} \to \mathrm{End} \, \catC$, defined as follows:

\begin{itemize}
\item For $n \geq -1$, $[n] \mapsto M^{n+1}$;
\item $(\bd_{i} \colon [n-1] \to [n]) \mapsto M^{n-i}\eta_{M^i}$;
\item $(\sigma_{i} \colon [n] \to [n-1]) \mapsto M^{n-i-1}\mu_{M^i}$.
\end{itemize}

In particular, for any $c \in \catC$ there is an augmented cosimplicial object $\Delta_{\mathrm{aug}} \to \catC$ given by instantiating this construction at $c$.
\end{proposition}

\begin{proof}
This follows from the characterization of $\Delta_{\mathrm{aug}}$ as the universal monoidal category equipped with a monoid, together with the characterization of monads on a category $\catC$ as monoids in $\mathrm{End} \, \catC$.
\end{proof}

For $n \geq 0$, let $Q^n$ denote the cubical set $C^{n+1}\varnothing = C^{0,n}$. Likewise, for $W \in \{L,R\}, \varepsilon \in \{0,1\}$, let $Q^{n}_{W,\varepsilon} = C^{n+1}_{W,\varepsilon} \varnothing$.

\begin{proposition}\label{left-positive-Q}
The assignment $[n] \mapsto Q^n$ extends to a cosimplicial object $Q \colon \Delta \to \cSet$, with simplicial structure maps defined as follows:

$\begin{array}{l|ccccccc}
\text{a map } Q^{n-1} \to Q^n & 
0^\text{th} \text{ face} & 1^\text{st} \text{ face} & 2^\text{nd} \text{ face} & 
\cdots & j^\text{th} \text{ face} & \cdots & n^\text{th} \text{ face} \\ \hline

\text{is induced by a map } \Box^{n-1} \to \Box^n & 
\partial_{n, 1} & \partial_{n, 0} & \partial_{n-1,0} &
\cdots & \partial_{n-j+1, 0} & \cdots & \partial_{1, 0} \\
& &&&&&& \\
\text{a map } Q^n \to Q^{n-1} & 
0^\text{th} \text{ deg.} & 1^\text{st} \text{ deg.} & 2^\text{nd} \text{ deg.} &
\cdots & j^\text{th} \text{ deg.} & \cdots & (n-1)^\text{st} \text{ deg.} \\ \hline

\text{is induced by a map } \Box^n \to \Box^{n-1} & 
\sigma_n & \gamma_{n-1,0}  & \gamma_{n-2,0} &
\cdots & \gamma_{n-j,0} & \cdots & \gamma_{1,0}
\end{array}$
\end{proposition}

\begin{proof}
This follows from applying \cref{monad-cosimplicial-object} to the monad of \cref{left-positive-monad} and the object $\varnothing \in \cSet$. For a direct construction, see \cite[Prop. 2.3]{kapulkin-lindsey-wong}.
\end{proof}

Taking the left Kan extension of this cosimplicial object along the Yoneda embedding, we obtain a functor $Q \colon \sSet \to \cSet$.

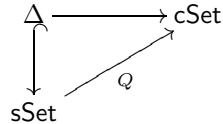
\[
\xymatrix@C+0.5cm{
  \Delta
  \ar[r]
  \ar@{^{(}->}[d]
&
  \cSet
\\
  \sSet
  \ar[ru]_{Q}
&
}
\]

This functor has a right adjoint $\int \colon \cSet \to \sSet$, given by $(\int X)_{n} = \cSet(Q^{n},X)$.

\begin{remark}
Viewing $\sSet$ as the slice category $\sSet \downarrow \Delta^{0}$ and $\cSet$ as the functor category $\cSet^{[0]}$, the adjunction $Q \adjoint \int$ coincides with the cubical straightening-unstraightening adjunction developed in \cite{kapulkin-voevodsky:cubical-straightening}.
\end{remark}

The alternative cone monads described in \cref{cone-monad} admit similar constructions.

\begin{proposition}\label{Q-cosimp-ob}
For $W \in \{L,R\}$, $\varepsilon \in \{0,1\}$, the assignment $[n] \mapsto Q^{n}_{W,\varepsilon}$ extends to a cosimplicial object $Q_{W,\varepsilon} \colon \Delta \to \cSet$. For $(W, \varepsilon) \neq (L,1)$ the simplicial structure maps are defined as follows:

$\begin{array}{l|ccccccc}
\text{a map } Q_{L,0}^{n-1} \to Q_{L,0}^n & 
0^\text{th} \text{ face} & 1^\text{st} \text{ face} & 2^\text{nd} \text{ face} & 
\cdots & j^\text{th} \text{ face} & \cdots & n^\text{th} \text{ face} \\ \hline

\text{is induced by a map } \Box^{n-1} \to \Box^n & 
\partial_{n, 0} & \partial_{n, 1} & \partial_{n-1,1} &
\cdots & \partial_{n-j+1, 1} & \cdots & \partial_{1, 1} \\
& &&&&&& \\
\text{a map } Q_{L,0}^n \to Q_{L,0}^{n-1} & 
0^\text{th} \text{ deg.} & 1^\text{st} \text{ deg.} & 2^\text{nd} \text{ deg.} &
\cdots & j^\text{th} \text{ deg.} & \cdots & (n-1)^\text{st} \text{ deg.} \\ \hline

\text{is induced by a map } \Box^n \to \Box^{n-1} & 
\sigma_n & \gamma_{n-1,1}  & \gamma_{n-2,1} &
\cdots & \gamma_{n-j,1} & \cdots & \gamma_{1,1}
\end{array}$

\vspace{2pc}

$\begin{array}{l|ccccccc}
\text{a map } Q_{R,0}^{n-1} \to Q_{R,0}^n & 
0^\text{th} \text{ face} & 1^\text{st} \text{ face} & 2^\text{nd} \text{ face} & 
\cdots & j^\text{th} \text{ face} & \cdots & n^\text{th} \text{ face} \\ \hline

\text{is induced by a map } \Box^{n-1} \to \Box^n & 
\partial_{1, 0} & \partial_{1, 1} & \partial_{2,1} &
\cdots & \partial_{j, 1} & \cdots & \partial_{n, 1} \\
& &&&&&& \\
\text{a map } Q_{R,0}^n \to Q_{R,0}^{n-1} & 
0^\text{th} \text{ deg.} & 1^\text{st} \text{ deg.} & 2^\text{nd} \text{ deg.} &
\cdots & j^\text{th} \text{ deg.} & \cdots & (n-1)^\text{st} \text{ deg.} \\ \hline

\text{is induced by a map } \Box^n \to \Box^{n-1} & 
\sigma_1 & \gamma_{1,1}  & \gamma_{2,1} &
\cdots & \gamma_{j,1} & \cdots & \gamma_{n-1,1}
\end{array}$

\vspace{2pc}

$\begin{array}{l|ccccccc}
\text{a map } Q_{R,1}^{n-1} \to Q_{R,1}^n & 
0^\text{th} \text{ face} & 1^\text{st} \text{ face} & 2^\text{nd} \text{ face} & 
\cdots & j^\text{th} \text{ face} & \cdots & n^\text{th} \text{ face} \\ \hline

\text{is induced by a map } \Box^{n-1} \to \Box^n & 
\partial_{1, 1} & \partial_{1, 0} & \partial_{2,0} &
\cdots & \partial_{j, 0} & \cdots & \partial_{n, 0} \\
& &&&&&& \\
\text{a map } Q_{R,1}^n \to Q_{R,1}^{n-1} & 
0^\text{th} \text{ deg.} & 1^\text{st} \text{ deg.} & 2^\text{nd} \text{ deg.} &
\cdots & j^\text{th} \text{ deg.} & \cdots & (n-1)^\text{st} \text{ deg.} \\ \hline

\text{is induced by a map } \Box^n \to \Box^{n-1} & 
\sigma_1 & \gamma_{1,0}  & \gamma_{2,0} &
\cdots & \gamma_{j,0} & \cdots & \gamma_{n-1,0}
\end{array}$
\end{proposition}

\begin{proof}
This follows from applying \cref{monad-cosimplicial-object} to the monads of \cref{cone-monad}.
\end{proof}

\begin{remark}\label{monad-op}
 We could instead have chosen to define the cosimplicial object of \cref{monad-cosimplicial-object} by $\bd_i \mapsto M^{i} \eta_{M^{n-i}}$, $\sigma_i \mapsto M^{i} \mu_{M^{n-i-1}}$; this amounts to pre-composing the cosimplicial object as we have defined it with the involution $(-)^\op \colon \sSet \to \sSet$. If we had made this choice, we would have obtained a different set of cosimplicial objects $Q_{W,\varepsilon}$.
\end{remark}

As with the cosimplicial object constructed using left positive cones, each of these Kan extends to a functor $Q_{W,\varepsilon} \colon \sSet \to \cSet$ having a right adjoint $\int\limits_{W,\varepsilon} \colon \cSet \to \sSet$.

\begin{lemma}\label{Q-involutions}
The functors $Q_{W,\varepsilon}$ and $\int\limits_{W,\varepsilon}$, for $W \in \{L,R\}, \varepsilon \in \{0,1\}$, are related by the following formulas:

\begin{multicols}{2}
\begin{itemize}
\item $Q_{L,0} = (-)^\coop \circ Q_{L,1}$;
\item $Q_{R,0} = (-)^\op \circ Q_{L,1}$;
\item $Q_{R,1} = (-)^\rmco \circ Q_{L,1}$; \newline
\item $\int\limits_{L,0} = \int\limits_{L,1} \circ (-)^\coop$;
\item $\int\limits_{R,0} = \int\limits_{L,1} \circ (-)^\op$;
\item $\int\limits_{R,1} = \int\limits_{L,1} \circ (-)^\rmco$;
\end{itemize}
\end{multicols}
\end{lemma}

\begin{proof}
It suffices to prove the first three items, which follow from \cref{cone-involutions}.
\end{proof}

As we did in \cref{sec:cones}, from here on we will work exclusively with left positive cones except where noted, using the subscript $(L,1)$ only where the potential for ambiguity arises.

The analogous functor $Q \colon \sSet \to \cSet_0$, which we will denote $Q_{0}$, was previously studied in \cite{kapulkin-lindsey-wong}. In that paper, the objects $Q^{n}_{0}$ were described as quotients of $\Box^n_{0}$ under a certain equivalence relation; this relation is precisely that of \cref{ConeDesc} in the case $m = 0$. We begin by recalling some of the theory developed in that paper, and adapting it to the present setting.

\begin{lemma}\label{Q-i-composite}
We have a commuting triangle of adjunctions:

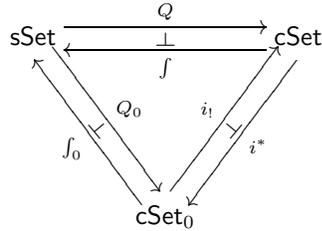
\centerline{
\xymatrix{
\sSet \ar@<1ex>[rr]^{Q} \ar@{}|{\rotatebox{-90}{$\adjoint$}}[rr] \ar@<1ex>[ddr]^{Q_0} \ar@{}|{\rotatebox{-140}{$\adjoint$}}[ddr] && \cSet \ar@<1ex>[ll]^{\int} \ar@<1ex>[ddl]^{i^\ast} \\
\\
& \cSet_0 \ar@<1ex>[uul]^{\int_0} \ar@<1ex>[uur]^{i_!} \ar@{}|{\rotatebox{-40}{$\adjoint$}}[uur]
}
}
\end{lemma} 

\begin{proof}
By definition, $i_{!} Q^{n}_{0} \cong Q^{n}$; the general result follows from this, using the fact that $i_{!}$ preserves colimits as a left adjoint.
\end{proof}

\begin{lemma}
For any $X \in \cSet$, the counit $Q \int X \to X$ is a monomorphism. 
\end{lemma}

\begin{proof}
By \cref{Q-i-composite}, the counit of $Q \adjoint \int$ is the natural map $i_{!} Q_0 \int_0 i^{\ast} X \to X$ for $X \in \cSet$. Applying \cref{mono-characterization}, we wish to show that this map sends non-degenerate cubes to non-degenerate cubes, and does so injectively. 

By \cref{i-adjunction-non-degenerate}, the action of this map on non-degenerate cubes coincides with that of its adjunct map $Q_0 \int_0 i^{\ast} X \to i^{\ast} X$. This map is a monomorphism by \cite[Lem. 4.2]{kapulkin-lindsey-wong}, hence it maps the non-degenerate cubes of $Q_0 \int_0 i^{\ast} X$ injectively to non-degenerate cubes of $i^{\ast} X$ by the analogue of \cref{mono-characterization} in $\cSet_0$. The non-degenerate cubes of $i^{\ast} X$ consist of the non-degenerate cubes of $X$, together with the degenerate cubes of $X$ which cannot be expressed as degeneracies or negative connections. Therefore, to show that $Q \int X \to X$ is a monomorphism, it suffices to show that it does not send any non-degenerate cubes of $Q \int X$ to degenerate cubes of $X$ whose standard forms contain positive connection maps.

Note that every non-degenerate $n$-cube of $Q \int X$ is a $(0,n)$-cone -- this follows from \cref{colim-factor,ConeFaceDeg}. Thus the images of such cones under any cubical set map will also be $(0,n)$-cones. By \cref{Q-stand-form}, their standard forms will therefore contain no positive connection maps. 
\end{proof}

This lemma shows that for any cubical set $X$, $Q \int X$ is a subcomplex of $X$. Specifically, it is the subcomplex whose non-degenerate $n$-cubes, for each $n$, are those which factor through $Q^{n}$ -- in other words, they are the non-degenerate $(0,n)$-cones in $X$.

\begin{theorem}\label{Q-fully-faithful}
The functor $Q \colon \sSet \to \cSet$ is fully faithful.
\end{theorem}

\begin{proof}
That $Q_0$ is fully faithful follows from \cite[Thm. 3.9]{kapulkin-lindsey-wong}. Since $i_!$ is faithful, it follows from \cref{Q-i-composite} that $Q$ is faithful as well. In general, $i_!$ is not full; we will show, however, that it is full on the image of $Q_0$, which suffices to show that the composite $Q$ is fully faithful.

Let $X, Y \in \sSet$, and consider a map $f \colon QX \to QY$. By \cref{map-on-non-degen}, $f$ is determined by its action on the non-degenerate cubes of $QX$. Let $x \colon \Box^n \to QX$ be non-degenerate; then $x$ is a $(0,n)$-cone, hence so is $fx$. Therefore, by \cref{Q-stand-form}, the standard form of $fx$ contains no positive connection maps; thus $fx$ corresponds to a cube of $Q_0 Y$.

Thus we see that the action of $f$ on the non-degenerate cubes of $QX$ (which coincide with those of $Q_0 X$) defines a map $Q_0 X \to Q_0 Y$; the image of this map under $i_!$ is precisely $f$. 
\end{proof}

Our next goal is to show the following:

\begin{proposition}\label{QAdj}
The adjunction $Q \adjoint \int$ is Quillen.
\end{proposition}

To prove this, we will show that this adjunction satisfies the hypotheses of \cref{Quillen-adj-fib-obs}.

\begin{lemma}\label{Q-pres-cof}
$Q$ preserves monomorphisms.
\end{lemma}

\begin{proof}
$Q_0$ preserves monomorphisms by \cite[Lem. 4.5]{kapulkin-lindsey-wong}. The stated result thus follows from \cref{Q-i-composite} and the fact that $i_!$ preserves monomorphisms.
\end{proof}

\begin{lemma}\label{openBoxInnerHornPushout} 
The image under $Q$ of an inner horn inclusion $\Lambda^n_i \subseteq \Delta^n$ is a trivial cofibration.
\end{lemma}

\begin{proof}  
Because $Q$ preserves colimits, $Q\Lambda^{n}_{i}$ is the subcomplex of $Q^{n}$ consisting of the images of the maps $Q \partial_{j} \colon  Q^{n-1} \to Q^{n}$ for which $j \neq i$. By \cref{left-positive-Q} we can see that this subcomplex is the image of $\sqcap^{n}_{n-i+1,0}$ under the quotient map $\Box^{n} \to Q^{n}$. We thus have the following pushout:

\[
  \xymatrix{
    \sqcap^n_{n-i+1,0} \ar[r]\ar[d] & Q\Lambda^n_i \ar[d] \\
    \Box^n \ar[r] & Q^n \pushoutcorner
  }
\]

Furthermore, the critical edge of the open box $\sqcap^n_{n-i+1,0} \to Q\Lambda^n_i$ is degenerate by \cref{ConeEdge}. Thus $Q\Lambda^n_i \hookrightarrow Q^{n}$ is a trivial cofibration, as an inner open box filling.
\end{proof}

\begin{lemma}\label{QJ} 
$QJ \cong K$. \qed
\end{lemma}

\begin{proof}[Proof of \cref{QAdj}]
By \cref{Q-pres-cof}, $Q$ preserves cofibrations. By \cref{openBoxInnerHornPushout}, the image under $Q$ of an inner-horn inclusion is a trivial cofibration. By \cref{QJ}, the image under $Q$ of an endpoint inclusion $\Delta^0 \to J$ is an endpoint inclusion $\Box^0 \to K$, hence a trivial cofibration by \cref{KendEq}. Thus the adjunction is Quillen by \cref{Quillen-adj-fib-obs}.
\end{proof}

\begin{corollary}\label{QPres}
$Q$ preserves weak equivalences.
\end{corollary}

\begin{proof}
Since all simplicial sets are cofibrant in the Quillen model structure, this follows from \cref{QAdj} and Ken Brown's lemma.
\end{proof}

Next we will concern ourselves with the relationship between $Q$ and the triangulation functor. Our goal will be to develop a natural weak categorical equivalence $TQ \Rightarrow \mathrm{id}_{\mathsf{sSet}}$.

For $n \geq 0$, we have a poset map $G \colon [n] \to [1]^n$ defined by:

$$
(Ga)_i = 
\begin{cases} 
0 & i \leq n - a \\
1 & i \geq n - a + 1 \\
\end{cases}
$$
For a given $n \geq 0$, let $F \colon [1]^n \to [n]$ be defined via $Fb = n -i + 1$,where $i \in \{1,...,n+1\}$ is maximal such that $b_j = 0$ for all $j < i$. 

\begin{lemma}\label{F-characterization}
For any $n \geq 0$, the functors $F : [1]^n \rightleftarrows [n] : G$ are adjoint.
\end{lemma}

\begin{proof}
  Let $a \in [n]$ and $b \in [1]^n$.
  We have that $b \leq Ga$ if and only if $b_j = 0$ for all $j \leq n - a$ -- in other words, if and only if $i \geq n - a + 1$. Rearranging, we obtain that this is equivalent to $n - i + 1 \leq a$, i.e., $Fb \leq a$.
\end{proof}

Similarly, one can show that $G$ also has a right adjoint, although it will not play any role in this paper.

\begin{proposition}
For all $n$, $F$ induces a map of simplicial sets $TQ^{n} \to \Delta^n$.
\end{proposition}

\begin{proof}
First, observe that by applying the nerve functor $N \colon \mathsf{Cat} \to \mathsf{sSet}$, we get an induced map $NF \colon (\Delta^{1})^{n} \to \Delta^{n}$. 

The simplicial set $TQ^{n}$ is a quotient of $T\Box^{n} = (\Delta^{1})^{n}$. Specifically, since $N$ is fully faithful, we may regard $n$-simplices $\Delta^{n} \to (\Delta^{1})^{n}$ as poset maps $[n] \to [1]^{n}$. Then by an argument analogous to the proof of \cref{ConeDesc}, using \cref{T-prod} and the fact that $T$ preserves colimits, $TQ^{n}$ is obtained by identifying two such maps $f, g$ if there exists $i$ such that $f_{j} = g_{j}$ for $j \leq i$ and $f_{i} = g_{i} = \mathrm{const}_{1}$. $NF$ then acts on such maps by post-composition with $F$. By \cref{F-characterization}, $F$ depends only on the position of the first $1$ in an object of $[1]^{n}$; therefore, maps which are identified in $TQ^{n}$ agree after post-composition with $F$. Thus $NF$ factors through the quotient $TQ^{n}$.
\end{proof}

Let $\bar{F} \colon TQ^{n} \to \Delta^{n}$ denote the map constructed above. Then we can show:

\begin{lemma}\label{FCos}
The maps $\bar{F} \colon TQ^{n} \to \Delta^{n}$ form a natural transformation of co-simplicial objects in $\mathsf{sSet}$. That is, for any map $\phi \colon [m] \to [n]$ in $\Delta$, we have a commuting diagram:

\centerline{
\xymatrix{
TQ^{m} \ar[r]^{TQ\phi} \ar[d]_{\bar{F}} & TQ^{n} \ar[d]^{\bar{F}} \\
\Delta^{m} \ar[r]^{\phi} & \Delta^{n} \\
}
}
\end{lemma}

\begin{proof}
It suffices to show that this holds for the generating morphisms of $\Delta$, namely the face and degeneracy maps. For each such map $\phi \colon [m] \to [n]$ we have a corresponding map $\phi' \colon [1]^{m} \to [1]^{n}$ in $\Box$, as described in \cref{left-positive-Q}:

\begin{itemize}
\item For $\partial_{0} \colon [n-1] \to [n]$, $\partial_{0}' = \partial_{n,1}$;
\item For $i \geq 1, \partial_{i} \colon [n-1] \to [n]$, $\partial_{i}' = \partial_{n-i+1,0}$;
\item For $\sigma_{0} \colon [n] \to [n-1]$, $\sigma_{0}' = \sigma_{n}$;
\item For $\sigma_{i} \colon [n] \to [n-1]$, $\sigma_{i}' = \gamma_{n-i,0}$.
\end{itemize}

For every such $\phi$ we have a commuting diagram in $\mathsf{cSet}$, where the vertical maps $\Box^{m} \to Q^{m}$ are the quotient maps:

\begin{equation}\label{F-diagram-cSet}
\begin{gathered}
\xymatrix{
\Box^{m} \ar[d] \ar[r]^{\phi'} & \Box^{n} \ar[d] \\
Q^{m} \ar[r]^{Q\phi} & Q^{n} \\
}
\end{gathered}
\end{equation}

Furthermore, by direct computation using \cref{F-characterization}, we have commuting diagrams in $\mathsf{Cat}$:

\begin{equation}\label{F-diagram-Cat}
\begin{gathered}
\xymatrix{
[1]^{m} \ar[d]_{F} \ar[r]^{\phi'} & [1]^{n} \ar[d]^{F} \\
[m] \ar[r]^{\phi} & [n] \\
}
\end{gathered}
\end{equation}

Now consider the following diagram in $\mathsf{sSet}$:

\centerline{
\xymatrix{
(\Delta^{1})^{m} \ar[r]^{T\phi'} \ar[d] & (\Delta^{1})^{n} \ar[d] \\
TQ^{m} \ar[r]^{TQ \phi} \ar[d]_{\bar{F}} & TQ^{n} \ar[d]^{\bar{F}} \\
\Delta^{m} \ar[r]^{\phi} & \Delta^{n} \\
}
}

The top square commutes, as it is obtained by applying $T$ to Diagram \ref{F-diagram-cSet}; the outer rectangle also commutes, as it is obtained by applying $N$ to Diagram \ref{F-diagram-Cat}. We wish to show that the bottom square commutes, i.e. that $\phi \circ \bar{F} = \bar{F} \circ TQ\phi$. Since the quotient map $(\Delta^{1})^{m} \to TQ^{m}$ is an epimorphism, we can show the desired equality by pre-composing with this map and performing a simple diagram chase.
\end{proof}

\begin{corollary}\label{FTrans}
$\bar{F}$ extends to a natural transformation $\bar{F} \colon TQ \Rightarrow \mathrm{id}_{\mathsf{sSet}}$.
\end{corollary}

\begin{proof}
Immediate from \cref{FCos} and the fact that $T$ and $Q$ preserve colimits.
\end{proof}

Note that this is precisely the map considered by Lurie in \cite[Prop.~2.2.2.7]{lurie:htt}, since Lurie's straightening construction can be recovered from its cubical analogue by composing with triangulation \cite[Thm.~3.8]{kapulkin-voevodsky:cubical-straightening}.

\begin{proposition}\label{FWE}
For every simplicial set $X$, the map $\bar{F} \colon TQX \to X$ is a weak categorical equivalence.
\end{proposition}

\begin{proof}
We begin by proving the statement for the case where $X$ is $m$-skeletal for some $m \geq 0$, proceeding by induction on $m$. For $m = 0$ or $m = 1$, the map in question is an isomorphism.

Now let $m \geq 2$, and suppose that the statement holds for any $(m-1)$-skeletal $X$. Then in particular, it holds for any horn $\Lambda^{m}_{i}$. For any $0 < i < n$, consider the following commuting diagram:

\centerline{
\xymatrix{
TQ \Lambda^{m}_{i} \ar@{^(->}[r]^{\sim} \ar[d]_{\sim} & TQ^{m} \ar[d] \\
\Lambda^{m}_{i} \ar@{^(->}[r]^{\sim} & \Delta^{m}
}
}

The left-hand map is a weak equivalence by the induction hypothesis; the bottom map is a trivial cofibration as an inner horn inclusion; and the top map is a trivial cofibration by \cref{QAdj} and \cref{T-Quillen-adj}. Thus $\bar{F} \colon TQ^{m} \to \Delta^{m}$ is a weak equivalence by the two-out-of-three property. Extending this result to an arbitrary $m$-skeletal simplicial set $X$ is an application of the gluing lemma, using the fact that both $T$ and $Q$ preserve colimits.

Now let $X$ be an arbitrary simplicial set; then $\bar{F}$ is a weak equivalence on the $n$-skeleton of $X$ for each $n \geq 0$. Thus $\bar{F} \colon TQX \to X$ is a weak equivalence, using the fact that sequential colimits of cofibrations preserve weak equivalences.
\end{proof}

\begin{proposition}\label{QRef}
$Q$ reflects weak categorical equivalences.
\end{proposition}

\begin{proof}
Let $f \colon X \to Y$ be a map of simplicial sets, such that $Qf$ is a weak categorical equivalence. We have a commuting diagram:

\centerline{
\xymatrix{
TQX \ar[r]^{TQf} \ar[d]_{\bar{F}} & TQY \ar[d]^{\bar{F}} \\
X \ar[r]^{f} & Y \\
}
}

The top horizontal map is a weak categorical equivalence by \cref{T-pres-cof}, as are the vertical maps by \cref{FWE}. Thus $f$ is a weak categorical equivalence by the 2-out-of-3 property.
\end{proof}

We have shown that the adjunction $Q \adjoint \int$ satisfies the hypotheses of \cref{QuillenEquivCreate} \ref{QuillenEquivCounit}. To show that it is a Quillen equivalence, therefore, we will prove the following:

\begin{proposition}\label{Counit}
For any cubical quasicategory $X$, the counit $\varepsilon \colon Q \int X \hookrightarrow X$ is anodyne.
\end{proposition}

\begin{proof}
Let $X$ be a cubical quasicategory. By \cref{theta-exists}, we may equip $X$ with a coherent family of composites $\theta$. We will build $X$ from $Q \int X$ via successive inner open box fillings, thereby showing that the inclusion of $Q \int X$ into $X$ is anodyne. 

For $m \geq 2, n \geq -1$, let $X^{m,n}$ denote the smallest subcomplex of $X$ containing all $(m',n')$-cones of $X$, as well as all cones of the form $\theta^{m',n'}(x)$, for $m' < m$ or $m' = m, n' \leq n$. In particular, this means $X^{2,-1} = Q \int X$, since all cubes in the image of $\theta^{0,n}$ or $\theta^{1,n}$ are degenerate.

For $m < m'$ or $m = m', n \leq n'$, we have $X^{m,n} \subseteq X^{m',n'}$. Thus we obtain a sequence of inclusions:

$$
Q \textstyle{\int} X = X^{2,-1} \hookrightarrow X^{2,0} \hookrightarrow ... \hookrightarrow X^{3,-1} \hookrightarrow X^{3,0} \hookrightarrow ... \hookrightarrow X^{m,n} \hookrightarrow ...
$$

Observe that the colimit of this sequence is $X$. Furthermore, for each $m$, $X^{m,-1}$ is the union of all subcomplexes $X^{m',n}$ for $m' < m$, i.e. the colimit of the sequence of inclusions $Q \int X \hookrightarrow ... \hookrightarrow X^{m',n} \hookrightarrow ...$. So to show that $Q \int X \hookrightarrow X$ is anodyne, it suffices to show that each map $X^{m,n-1} \hookrightarrow X^{m,n}$ for $n \geq 0$ is anodyne.

Fix $m \geq 2, n \geq 0$, and Let $S$ denote the set of non-degenerate $(m,n)$-cones of $X$ which are not $(m-1,n+1)$-cones, and are not in the image of $\theta^{m,n-1}$ -- in other words, those $(m,n)$-cones which fall under case (6) of \cref{ThetaDef}. To construct $X^{m,n}$ from $X^{m,n-1}$, we must adjoin to $X^{m,n-1}$ all $(m,n)$-cones of $X$, and images of such cones under $\theta^{m,n}$, which are not already present in $X^{m,n}$. Using \cref{sa1,ConeFaceDeg,Theta-T}, and the identities \ref{ThetaFace0} to \ref{ThetaConeWLOG}, we can see that these are precisely the cones in $S$ and their images under $\theta^{m,n}$.

Let $x \in S$; we will analyze the faces of $\theta^{m,n}(x)$ to determine which of them are contained in $X^{m,n-1}$. For $i \leq n$ we have $\theta^{m,n}(x) \bd_{i,0} = \theta^{m,n-1}(x\bd_{i,0})$ by \ref{ThetaFace0}, while for $i \geq n + 2$ or $\varepsilon = 1$, $\theta^{m,n}(x) \bd_{i,\varepsilon}$ is an $(m-1,n+1)$-cone by \cref{ConeFaceDeg}. Thus we see that the only face of $\theta^{m,n}(x)$ which is not contained in $X^{m,n-1}$ is $\theta^{m,n}(x)\bd_{n+1,0} = x$. Furthermore, the critical edge of $\theta^{m,n}(x)$ with respect to the $(n+1,0)$-face  is degenerate; for $n = 0$ this follows from \cref{ThetaDegEdge}, while for $n \geq 1$ it follows from \cref{ConeEdge}. Thus the faces of $\theta^{m,n}(x)$ which are contained in $X^{m,n-1}$ form an $(n+1,0)$-inner open box.

Constructing $X^{m,n}$ from $X^{m,n-1}$ amounts to filling all of these inner open boxes; in other words, we have a pushout diagram:

\centerline{
\xymatrix{
\bigsqcup\limits_{S} \widehat{\sqcap}^{m+n+1}_{n+1,0} \ar[r] \ar@{^(->}[d] & X^{m,n-1} \ar@{^(->}[d] \\
\bigsqcup\limits_{S} \widehat{\Box}^{m+n+1} \ar[r] & X^{m,n} \pushoutcorner
}
}

Thus $X^{m,n-1} \hookrightarrow X^{m,n}$ is anodyne, as a pushout of a coproduct of anodyne maps.
\end{proof}

\begin{theorem} \label{Q-Quillen-equivalence}
The adjunction $Q : \sSet \rightleftarrows \cSet : \int$ is a Quillen equivalence.
\end{theorem}

\begin{proof}
The adjunction is Quillen by \cref{QAdj}. $Q$ preserves and reflects the weak equivalences of the Quillen model structure on $\mathsf{sSet}$ by \cref{QPres} and \cref{QRef}. Thus $Q \dashv \int$ satisfies the hypotheses of \cref{QuillenEquivCreate}, \cref{QuillenEquivCounit} and we can apply \cref{Counit} to conclude that it is a Quillen equivalence.
\end{proof}

\begin{corollary}
For all $W \in \{L,R\}, \varepsilon \in \{0,1\}$, the adjunction $Q_{W,\varepsilon} \adjoint \int\limits_{W,\varepsilon}$ is a Quillen equivalence.
\end{corollary}

\begin{proof}
Immediate from \cref{op-Quillen-equiv,Q-involutions,Q-Quillen-equivalence}.
\end{proof}

\begin{proof}[Proof of \cref{T-Quillen-equivalence}]
The adjunction $T \adjoint U$ is Quillen by \cref{T-Quillen-adj}. To see that it is a Quillen equivalence, note that because all objects in both $\cSet$ and $\sSet$ are cofibrant, the left derived functor $L(TQ)$ is the composite of the left derived functors $LT$ and $LQ$, while the left derived functor of the identity is the identity; this can easily be seen from \cite[Def.~1.36]{hovey:book}. By \cref{FTrans}, we have a natural weak equivalence $TQ \Rightarrow \mathrm{id}_{\sSet}$. In the homotopy category $\Ho \, \sSet$, this natural weak equivalence becomes a natural isomorphism $LT \circ LQ \cong \mathrm{id}_{\Ho \, \sSet}$. By \cref{Q-Quillen-equivalence}, $LQ$ is an equivalence of categories, thus $LT$ is an equivalence of categories as well. 
\end{proof}

The proofs in this section can easily be adapted to show that $Q \adjoint \int$ is a Quillen equivalence between the Quillen and Grothendieck model structures for $\infty$-groupoids on $\sSet$ and $\cSet$, respectively. (The analogue of this result for $\cSet_0$ was essentially stated as \cite[Prop. 5.3]{kapulkin-lindsey-wong}, but the proof supplied there only shows that $Q_0$ and $\int_0$ form a Quillen adjunction.)

\begin{proposition}\label{Quillen-Grothendieck-equiv}
The adjunction $Q : \sSet \rightleftarrows \cSet : \int$ is a Quillen equivalence between the Quillen model structure on $\sSet$ and the Grothendieck model structure on $\cSet$.
\end{proposition}

\begin{proof}
\cref{T-Quillen-adj} and \cref{QAdj} both have natural analogues, showing that $T \adjoint U$ and $Q \adjoint \int$ are Quillen adjunctions between these model structures (implying in particular that $Q$ preserves weak equivalences). Since every weak equivalence in the Joyal model structure is also a weak equivalence in the Quillen model structure, $\overline{F}$ is a natural weak equivalence in the Quillen model structure as well. Thus the proof of \cref{QRef} adapts to show that $Q$ reflects the weak equivalences of the Quillen model structure. \cref{QuillenEquivCreate}, \cref{QuillenEquivCounit} and \cref{Counit} then imply the analogue of \cref{Q-Quillen-equivalence}, since every cubical Kan complex is a cubical quasicategory and every weak equivalence in the cubical Joyal model structure is a weak equivalence in the Grothendieck model structure.
\end{proof}

We thus obtain a new proof of the following result, previously shown in \cite[Prop.~8.4.30]{CisinskiAsterisque} for cubical sets without connections:

\begin{theorem}\label{Quillen-Grothendieck-equiv-T}
$T \adjoint U$ is a Quillen equivalence between the Grothendieck model structure on $\cSet$ and the Quillen model structure on $\sSet$. \qed
\end{theorem}

\section{Mapping spaces in cubical quasicategories}\label{sec:fund-thm}

In this section we introduce the concept of a mapping space in a cubical quasicategory, and show that categorical equivalences can be characterized in terms of these mapping spaces and the homotopy category construction of \cref{section:marked}.

Our arguments in this section, up to the proof of our main result, \cref{fundamental-theorem}, do not rely on connections, thus they are valid in all of the categories of cubical sets shown in the diagram \cref{eq:cube-cat-inclusions} at the end of \cref{sec:background}. Some of the subsequent results, which relate cubical mapping spaces and homotopy categories to their simplicial analogues, make use of the adjunction $Q \dashv \int$ developed in \cref{sec:relation}, so those results which involve these functors can only be adapted to $\cSet_0$ or $\cSet_1$.

\begin{definition}\label{cubical-mapping-space}
For $X \in \cSet$ and $x_0,x_1 \colon \Box^0 \to X$, we define the \emph{mapping space} $\Map_{X}(x_0,x_1)$ by the following pullback:

\centerline{
\xymatrix{
\Map_{X}(x_0,x_1) \drpullback \ar[d] \ar[rr] && \ihom_{L}(\Box^1,X) \ar[d] \\
\Box^0 \ar[rr]^{(x_0,x_1)} && \ihom_{L}(\bd \Box^1, X) \\
}
}
\end{definition}

From this definition, we can derive a more concrete description of the cubical mapping space. For $X \in \cSet$, $x_0, x_1 \colon \Box^0 \to X$, we have:

\[ \Map_X(x_0, x_1)_n = \left\{  \Box^{n+1} \overset{s}{\to} X \ \big| \ s \partial_{n+1, \varepsilon} = x_\varepsilon  \right\}\text{,} \]

with the cubical operations given by those of $X$. Note that $x_{\varepsilon}$ here refers to the degeneracy of the vertex $x_{\varepsilon}$ in the appropriate dimension.


There is a clear geometric intuition behind this definition, as the example below shows.

\begin{example}
  Given a cubical set $X$ and $0$-cubes $x_0, x_1 \colon \Box^0 \to X$, we have that:
  \begin{itemize}
    \item a $0$-cube in $\Map_X(x_0, x_1)$ is a $1$-cube from $x_0$ to $x_1$ in $X$;
	\item a $1$-cube in $\Map_X(x_0,x_1)$ is a $2$-cube in $X$ of the form
	\[
\xymatrix{
  x_0 \ar@{=}[r] \ar[d] & x_0 \ar[d] \\
  x_1 \ar@{=}[r] & x_1 }
\]
  \end{itemize}
\end{example}

Given a cubical set map $f \colon X \to Y$, for any $x_{0}, x_{1} \colon \Box^0 \to X$ there is a natural map $f_{\ast} \colon \Map_{X}(x_{0},x_{1}) \to \Map_{Y}(fx_{0},fx_{1})$ induced by a natural map between the pullbacks of \cref{cubical-mapping-space}. Thus the mapping space construction defines a functor $\Map \colon \bd \Box^1 \downarrow \cSet \to \cSet$. In fact, this functor has a left adjoint, which we will now describe.

\begin{definition}\label{Sigma-def}
For $X \in \cSet$, the \emph{suspension} of $X$ is the bi-pointed cubical set $\Sigma X \in \bd \Box^1 \downarrow \cSet$ defined by the following pushout diagram:

\centerline{
\xymatrix{
X \sqcup X \ar[d] \ar[r] & \bd \Box^1 \ar[d] \\
X \otimes \Box^1 \ar[r] & \Sigma X \pushoutcorner \\
}
}

The chosen map $\bd \Box^1 \to \Sigma X$ is that which appears in the diagram above. We denote the basepoints of $\Sigma X$, i.e. the images under this map of the vertices $0, 1 \in \bd \Box^1$, by $0$ and $1$, respectively. For $f \colon X \to Y$, we define $\Sigma f \colon \Sigma X \to \Sigma Y$ to be the natural map between pushouts induced by $f$.
\end{definition}

\begin{proposition}\label{Sigma-Map-adjoint}
The functor $\Sigma \colon \cSet \to \bd \Box^1 \downarrow \cSet$ is left adjoint to $\Map \colon \bd \Box^1 \downarrow \cSet \to \cSet$.
\end{proposition}

\begin{proof}
Let $X, Y \in \cSet, y_0,y_1 \colon \Box^0 \to Y$. By the universal property of the pullback, maps $X \to \Map_{Y}(y_0,y_1)$ correspond to diagrams of the form

\centerline{
\xymatrix{
X \ar[d] \ar[rr]^-{f} && \ihom_{L}(\Box^1,Y) \ar[d] \\
\Box^0 \ar[rr]^-{(y_0,y_1)} && \ihom_{L}(\bd \Box^1, Y) \\
}
}

The map $\ihom_{R}(\Box^1,Y) \to \ihom_{R}(\bd \Box^1,Y)$ is the pullback exponential $(\bd \Box^1 \hookrightarrow \Box^1) \triangleright (Y \to \Box^0)$. Using the duality between pushout products and pullback exponentials, and observing that the pushout object $X \otimes \Box^1 \cup_{X \otimes \bd \Box^1} \Box^0 \otimes \bd \Box^1$ is precisely $\Sigma X$, we have a natural bijection between such diagrams and diagrams of the form

\centerline{
\xymatrix{
\Sigma X \ar[r]^{\overline{f}} \ar[d] & Y \ar[d] \\
\Box^1 \ar[r] & \Box^0 \\
}
}

in which $\overline{f}$ maps the basepoints $0 \mapsto y_0, 1 \mapsto y_1$. In other words, cubical set maps $X \to \Map_{Y}(y_0,y_1)$ are in natural bijection with bi-pointed cubical set maps $(\Sigma X, 0, 1) \to (Y,y_0,y_1)$.
\end{proof}

By using $\ihom_{R}$ rather than $\ihom_{L}$ in the pullback diagram of \cref{cubical-mapping-space}, we obtain the \emph{left mapping space} functor $\Map^L \colon \bd \Box^1 \downarrow \cSet \to \cSet$.

\centerline{
\xymatrix{
\Map^L_{X}(x_0,x_1) \drpullback \ar[d] \ar[rr] && \ihom_{R}(\Box^1,X) \ar[d] \\
\Box^0 \ar[rr]^-{(x_0,x_1)} && \ihom_{R}(\bd \Box^1, X) \\
}
}

This functor admits the following explicit description:

\[ \Map^{L}_X(x_0, x_1)_n = \left\{  \Box^{n+1} \overset{s}{\to} X \ \big| \ s \partial_{1, \varepsilon} = x_\varepsilon  \right\}\text{,} \]

Once again, cubical operations are given by those of $X$. In this case, that means that each face map $\bd_{i,\varepsilon}$ of $\Map^L_{X}(x_0,x_1)$ is induced by the face map $\bd_{1+i,\varepsilon}$ of $X$, and similarly for degeneracies and connections. 

$\Map^L$ also has a left adjoint, the \emph{left suspension} $\Sigma_L \colon \cSet \to \bd \Box^1 \downarrow \cSet$, with $\Sigma_{L} X$ defined as a quotient of $\Box^1 \otimes X$. Where the potential for confusion may arise, we will refer to $\Map$ and $\Sigma$ as the \emph{right mapping space} and \emph{right suspension}, denoting them by $\Map^{R}$ and $\Sigma_{R}$.

\begin{lemma}\label{Map-involutions}
The left and right mapping space constructions are related by the following formulas:

\begin{itemize}
\item $\Map^{L}_{X}(x_0,x_1)^\rmco \cong \Map^{R}_{X^\rmco}(x_0,x_1)$, $\Map^{R}_{X}(x_0,x_1)^\rmco \cong \Map^{L}_{X^\rmco}(x_0,x_1)$;
\item $\Map^{L}_{X}(x_0,x_1)^\coop \cong \Map^{L}_{X^\coop}(x_1,x_0)$, $\Map^{R}_{X}(x_0,x_1)^\coop \cong \Map^{R}_{X^\coop}(x_1,x_0)$;
\item $\Map^{L}_{X}(x_0,x_1)^\op \cong \Map^{R}_{X^\op}(x_1,x_0)$, $\Map^{R}_{X}(x_0,x_1)^\op \cong \Map^{L}_{X^\op}(x_1,x_0)$;
\end{itemize}
\end{lemma}

\begin{proof}
This follows from applying the involutions $(-)^\rmco, (-)^\coop$, and $(-)^\op$ to the pullback diagrams defining $\Map^L$ and $\Map^R$, and applying \cref{op-co-hom}.
\end{proof}

From here on, we will work exclusively with right mapping spaces unless otherwise noted, omitting the superscript $R$, with the understanding that our results may be adapted to left mapping spaces using the formulas of \cref{Map-involutions}.

\begin{proposition}\label{Sigma-Map-Quillen-adj}
$\Sigma \adjoint \Map$ is a Quillen adjunction between the Grothendieck model structure on $\cSet$ and the cubical Joyal model structure on $\bd \Box^1 \downarrow \cSet$.
\end{proposition}

\begin{proof}
That $\Sigma$ preserves cofibrations follows from the description of the geometric product in \cref{geometric-product-description}. To show that $\Sigma$ preserves trivial cofibrations, it suffices to show that $\Sigma$ sends all open box inclusions to trivial cofibrations in the cubical Joyal model structure. To see this, observe that $\Sigma \Box^n$ is the quotient of $\Box^{n+1}$ in which the faces $\bd_{n+1,0}, \bd_{n+1,1}$ are quotiented down to vertices, while $\Sigma \sqcap^{n}_{i,\varepsilon}$ is the corresponding quotient of $\sqcap^{n+1}_{i,\varepsilon}$. For $i \leq n, \varepsilon \in \{0,1\}$, the critical edge of $\Box^{n+1}$ with respect to the face $\bd_{i,\varepsilon}$ is an edge of the face $\bd_{n+1,1-\varepsilon}$, hence its image in $\Sigma \Box^{n+1}$ is degenerate. Thus the inclusion $\Sigma \sqcap^{n}_{i,\varepsilon} \to \Sigma \Box^{n}$ is a trivial cofibration. 
\end{proof}

\begin{corollary} \label{Map-right-Quillen}
If $f \colon X \to Y$ is a (trivial) fibration in the cubical Joyal model structure, then each induced map $f_{\ast} \colon \Map_{X}(x_0,x_1) \to \Map_{Y}(fx_0,fx_1)$ is a (trivial) fibration in the Grothendieck model structure.

In particular, if $X$ is a cubical quasicategory then all mapping spaces $\Map_{X}(x_0,x_1)$ are cubical Kan complexes. \qed
\end{corollary}

We can characterize categorical equivalences in terms of these mapping spaces and the homotopy categories defined in Section \ref{sec:marked-cubical-quasicat-theory}.

\begin{definition}
Let $X$ be a cubical quasicategory. We define the \emph{homotopy category} $\Ho X$ to be the homotopy category of the marked cubical quasicategory $X^{\natural}$.
\end{definition}

\begin{lemma}\label{Ho-involutions}
For a cubical quasicategory $X$, we have the following natural isomorphisms:

\begin{itemize}
\item $\Ho \, X^\rmco \cong \Ho \, X$;
\item $\Ho \, X^\coop \cong (\Ho \, X)^\op$;
\item $\Ho \, X^\op \cong (\Ho \, X)^\op$. \qed
\end{itemize}
\end{lemma}

Our next goal is to prove the following:

\begin{theorem}\label{fundamental-theorem}
Let $f \colon X \to Y$ be a map between cubical quasicategories. Then $f$ is a categorical equivalence if and only if the following two conditions are satisfied:

\begin{itemize}
\item $\Ho f \colon \Ho X \to \Ho Y$ is an equivalence of categories;
\item for all pairs of vertices $x_{0}, x_{1} \colon \Box^0 \to X$, the induced map $\Map_{X}(x_{0},x_{1}) \to \Map_{Y}(fx_{0},fx_{1})$ is a homotopy equivalence in the Grothendieck model structure.
\end{itemize}
\end{theorem}

The proof of this statement will require several steps. We begin by defining certain quotients of standard cubes which will be used in the proof.

\begin{definition}\label{K-def}
For $n \geq 0$, we define $K^n$ to be the quotient of $\Box^{n+2}$ in which:

\begin{itemize}
\item $\Box^1 \otimes \{0\} \otimes \Box^n$, i.e. the $(2,0)$-face, is degenerate in the first dimension, i.e. the corresponding map  $\Box^{1+n} \to K^{n}$ factors through $\sigma_1$;
\item $\{1\} \otimes \{1\} \otimes \Box^n$ is a degeneracy of a vertex;
\item the edge $\Box^1 \otimes \{1,...,1\}$ is degenerate. 
\end{itemize}

Let $\overline{K}^{n}$ denote the image of $\Box^1 \otimes \bd \Box^{n+1}$ in $\overline{K}^n$.
  For $\varepsilon \in \{0, 1\}$, let $K^n_{\varepsilon}$ denote the image of $\{\varepsilon\} \otimes \Box^{n+1}$ in $K^n$.
  Let $\overline{K}^n_{\varepsilon}$ denote the intersection of $\overline{K}^n$ and $K_{\varepsilon}^n$, i.e. the image in $K^n$ of the boundary of the $(1,\varepsilon)$-face. 
\end{definition}

Note that the inclusion $\overline{K}_{0}^n \hookrightarrow K_0^n$ is isomorphic to $\bd \Box^{n+1} \hookrightarrow \Box^{n+1}$. Similarly, $\overline{K}_{1}^n \hookrightarrow K_1^n$ is isomorphic to the quotient of $\bd \Box^{n+1} \hookrightarrow \Box^{n+1}$ where the $(1,1-\varepsilon)$-face is a degeneracy of a vertex.

\begin{lemma}\label{K-boundary-anodyne}
For all $n \geq 0$, the inclusion $\overline{K}^n_0 \hookrightarrow \overline{K}^n$ is anodyne.
\end{lemma}

\begin{proof}
Let $E$ denote the union of $\overline{K}_0^n$ with the image of $\Box^1 \otimes \{1\} \otimes \Box^n$ in $\overline{K}^n$. We first show that the inclusion $\overline{K}_0^n \hookrightarrow E$ is anodyne. To see this, observe that the intersection of $\overline{K}_0^n$ with $E$ is the image in $K^n$ of $\{0\} \otimes \{1\} \otimes \Box^n$. This coincides with the image in $K^n$ of $\bd \Box^1 \otimes \{1\} \otimes \Box^n \cup \Box^1 \otimes \{1\} \otimes \{(1,...,1)\}$. Thus $\overline{K}_0^n \hookrightarrow E$ is a pushout of the image in $K^n$ of the inclusion $\bd \Box^1 \otimes \{1\} \otimes \Box^n \cup \Box^1 \otimes \{1\} \otimes \{(1,...,1)\} \hookrightarrow \Box^1 \otimes \{1\} \otimes \Box^n$. This map can be written as a composite of open box fillings; in $K^n$, the critical edges of each of these open boxes will be degenerate, hence the map will be anodyne.
    
To see that $E \hookrightarrow \overline{K}^n$ is anodyne, observe that $E$ consists of the images in $K^n$ of the boundary of the $(1,0)$-face together with the $(2,\varepsilon)$-faces. For $2 \leq i \leq n+2$, let $E_i$ consist of the images in $K^n$ of the boundary of the $(1,0)$-face together with the $(j,\varepsilon)$-faces for $j \leq i$; thus $E_2 = E$ while $E_{n+1} = \overline{K}^n$. So it suffices to show that each map $E_i \hookrightarrow E_{i+1}$ is anodyne. To see this, observe that for $i \geq 2, \varepsilon \in \{0,1\}$, the intersection of the image in $K^n$ of the $(i+1,\varepsilon)$-face with $E_i$ consists of the images of its $(1,0)$-face and its $(j,\varepsilon')$-faces for $j \leq i$ (this can be seen from the cubical identities). In other words, this intersection is the image in $K^n$ of $\sqcap^{i}_{1,1} \otimes \{\varepsilon\} \otimes \Box^{(n+2)-(i+1)}$. Thus the inclusion $E_i \hookrightarrow E_{i+1}$ is a pushout of $\sqcap^{i}_{1,1} \otimes \bd \Box^1 \otimes \Box^{(n+2)-(i+1)} \hookrightarrow \Box^{i} \otimes \bd \Box^1 \otimes \Box^{(n+2)-(i+1)}$. Moreover, the image in $K^n$ of $\Box^1 \otimes \{0,...,0\} \otimes \bd \Box^1 \otimes \Box^{(n+2)-(i+1)}$ is degenerate in the first dimension. Thus this map is a pushout of the anodyne map $\widehat{\sqcap}^{i}_{1,1} \otimes \bd \Box^1 \otimes \Box^{(n+2)-(i+1)} \hookrightarrow \widehat{\Box}^{i}_{1,1} \otimes \bd \Box^1 \otimes \Box^{(n+2)-(i+1)}$.
\end{proof}

\begin{lemma}\label{K-anodyne}
For $n \geq 0$, the inclusion $K^n_0 \hookrightarrow K^n$ is anodyne.
\end{lemma}

\begin{proof}
Consider the following diagram:

\centerline{
\xymatrix{
\overline{K}^n_0 \ar[r]^-{\sim} \ar[d] & \overline{K}^n \ar[d] \ar@/^1.5pc/[ddr]\\
K^n_0 \ar[r]^-{\sim} \ar@/_1.5pc/[drr] & K^n_0 \cup_{\overline{K}^n_0} \overline{K}^n \ar[dr] \pushoutcorner \\
&& K^n \\
}
}

The inclusion $K^n_0 \hookrightarrow K^n_0 \cup_{\overline{K}^n_0} \overline{K}^n$ is anodyne as a pushout of an anodyne map. The inclusion $K^n_0 \cup_{\overline{K}^n_0} \overline{K}^n \hookrightarrow K^n$ is a $(1,1)$-open box filling; as the critical edge is the degenerate edge ${\Box^1 \otimes \{0,...,0\}}$, this is an inner open box filling.
\end{proof}

\begin{lemma}\label{lift-replace-degen-face}
Let $X \to Y$ be a fibration between cubical quasicategories. Let $x : \Box^n \to X$, for $n \geq 0$, and $\varepsilon \in \{0, 1\}$. Consider all diagrams of the form:

\centerline{
\xymatrix{
\bd \Box^{n+1} \ar[r] \ar[d] & X \ar[d] \\
\Box^{n+1} \ar[r] & Y \\ 
}
}

for which the $(1,\varepsilon)$-face of the boundary $\Box^{n+1} \to X$ is $x$. A lift exists in every such diagram if and only if a lift exists in every such diagram for which the $(1,1-\varepsilon)$-face of $\bd \Box^{n+1} \to X$ is a degeneracy of a vertex.
\end{lemma}

\begin{proof}
Fix $x$ and a diagram of the form depicted above; we will obtain a lift in the given diagram under the assumption that a lift exists for all such diagrams in which the $(1,1-\varepsilon)$-face of $\bd \Box^{n+1} \to X$ is a degeneracy of a vertex. By duality, it suffices to consider the case $\varepsilon = 0$. 

By \cref{K-boundary-anodyne,K-anodyne}, we have an injective trivial cofibration from $\overline{K}^n_0 \hookrightarrow K^n_0$ to $\overline{K}^n \hookrightarrow K^n$, regarding these maps as objects in the morphism category $\cSet^{\to}$. (Note that the injective model structure on $\cSet^{\to}$ coincides with the Reedy model structure by \cref{Box-Reedy}.) Furthermore, the map $X \to Y$ is injective fibrant, as a fibration between fibrant objects. Therefore, identifying $\bd \Box^{n+1} \hookrightarrow \Box^{n+1}$ with $\overline{K}^n_0 \hookrightarrow K^n_0$, the given diagram factors as:

\centerline{
\xymatrix{
\overline{K}^n_0 \ar@{^(->}[d] \ar@{^(->}[r] & \overline{K}^n \ar[r] \ar@{^(->}[d] & X \ar[d] \\
K^n_0 \ar@{^(->}[r] & K^n \ar[r] & Y \\
}
}

Thus, to obtain a lift in the original diagram, it suffices to obtain a lift in the right-hand diagram above. For this, observe that the inclusion $K^n_{1} \cup_{\overline{K}^n_{1}} \overline{K}^n \hookrightarrow K^n$ is a $(1,0)$-inner open box filling, whose critical edge is the degenerate edge $\Box^1 \otimes \{1,...,1\}$; thus this map is anodyne. Therefore, it suffices to obtain a lift in the diagram:

\centerline{
\xymatrix{
\overline{K}^n \ar@{^(->}[d] \ar[rr] && X \ar[d] \\
K^n_{1} \cup_{\overline{K}^n_{1}} \overline{K}^n \ar[r] & K^n \ar[r] & Y \\ 
}
}

For this, in turn, it suffices to obtain a lift in the diagram:

\centerline{
\xymatrix{
\overline{K}^n_1 \ar@{^(->}[d] \ar[r] & X \ar[d] \\
K^n_{1} \ar[r] & Y \\ 
}
}

The boundary $\bd \Box^{n+1} \to \overline{K}^n_1 \to X$ has $x$ as its $(1,0)$-face; this follows from the fact that the image in $K^n$ of the $(2,0)$-face is degenerate in the first dimension. Similarly, the $(1,1)$-face of this boundary is a degeneracy of a vertex, as it is precisely the image of $\{1\} \otimes \{1\} \otimes \Box^n$. Thus this diagram admits a lift by assumption.
\end{proof}

\begin{corollary}\label{lift-replace-Sigma}
For $n \geq 0$, a fibration between cubical quasicategories has the right lifting property against $\bd \Box^{n+1} \hookrightarrow \Box^{n+1}$ if and only if it has the right lifting property against $\Sigma \bd \Box^n \hookrightarrow \Sigma \Box^n$.
\end{corollary}

\begin{proof}
The forward implication follows from the fact that $\Sigma \bd \Box^n \hookrightarrow \Sigma \Box^{n+1}$ is a pushout of $\bd \Box^{n+1} \hookrightarrow \Box^{n+1}$. For the reverse implication, observe that $\Sigma \bd \Box^n$ (resp. $\Sigma \Box^n$) is precisely the quotient of $\bd \Box^{n+1}$ (resp. $\Box^{n+1}$) in which the $(1,0)$ and $(1,1)$-faces are degeneracies of vertices. The result then follows from applying \cref{lift-replace-degen-face} twice, once with $\varepsilon = 0$ and once with $\varepsilon = 1$.
\end{proof}

\begin{proof}[Proof of \cref{fundamental-theorem}]
First let $f \colon X \to Y$ be a categorical equivalence between cubical quasicategories. That $\Ho f$ is an equivalence of categories follows from \cref{tau-ho-equiv,straightforward-properties} \ref{tau-nerve-Quillen-adj}. That each map $\Map_{X}(x_{0},x_{1}) \to \Map_{Y}(fx_{0},fx_{1})$ is a homotopy equivalence follows from \cref{Sigma-Map-Quillen-adj}.

Now let $f \colon X \to Y$ be a map between cubical quasicategories inducing an equivalence on homotopy categories and homotopy equivalences on all mapping spaces. We will show that $f$ is a categorical equivalence. By factoring an arbitrary map as a composite of a trivial cofibration with a fibration and applying the implication which we have already proven, we may assume $f$ is a fibration. By \cref{Sigma-Map-Quillen-adj}, this implies that $f$ induces fibrations on all cubical mapping spaces. Thus we wish to show that, given a fibration of cubical quasicategories $f \colon X \to Y$ such that $\Ho f$ is an equivalence of categories and each map $\Map_{X}(x_0,x_1) \to \Map_{Y}(fx_0,fx_1)$ is a trivial fibration, $f$ is a trivial fibration.

 We begin by showing that $f$ has the right lifting property with respect to the map $\varnothing \to \Box^0$ -- in other words, that $f$ is surjective on vertices. To this end, let $y$ be a vertex of $Y$. Then since $\Ho f$ is essentially surjective, there is a vertex $x \colon \Box^0 \to X$ such that $fx \cong y$ in $\Ho Y$. Thus we have a commuting diagram in $\cSet$:
 
\centerline{
\xymatrix{
\Box^0 \ar[d]_{0} \ar[r] & X \ar[d]^{f} \\
K \ar[r] & Y
}
}

Since $f$ is a fibration, this diagram has a lift; the restriction of this lift to the endpoint $1 \colon \Box^0 \to K$ gives a vertex $x' \colon \Box^{0} \to X$ with $fx' = y$.

To complete the proof, we must show that $f$ has the right lifting property with respect to all boundary inclusions $\bd \Box^{n+1} \hookrightarrow \Box^{n+1}$ for $n \geq 0$. By \cref{lift-replace-Sigma}, it suffices to show that $f$ has the right lifting property with respect to all maps $\Sigma \bd \Box^{n} \hookrightarrow \Sigma \Box^{n}$. But by \cref{Sigma-Map-adjoint}, this is equivalent to our assumption that $f$ induces trivial fibrations on all mapping spaces.
\end{proof}

The following result shows that, in verifying the conditions of \cref{fundamental-theorem} for a map $f \colon X \to Y$, it suffices to show that $\Ho \, X$ is essentially surjective and that $f$ induces homotopy equivalences on all mapping spaces.

\begin{proposition}\label{fundamental-theorem-essentially-surjective}
Let $f \colon X \to Y$ be a map between cubical quasicategories. If $f$ induces homotopy equivalences on all mapping spaces, then $\Ho \, X \to \Ho \, Y$ is fully faithful.
\end{proposition}

\begin{proof}
Factoring an arbitrary map as a trivial cofibration followed by a fibration and applying \cref{fundamental-theorem}, we see that it suffices to consider the case where $f$ is a fibration. By \cref{Sigma-Map-Quillen-adj}, this implies that each map $\Map_{X}(x_0,x_1) \to \Map_{Y}(fx_0,fx_1)$ is a trivial fibration.

For $x_0, x_1 \colon \Box^0 \to X$, $\Map_{X}(x_0,x_1) \to \Map_{Y}(fx_0,fx_1)$ is surjective on vertices, implying that every edge of $Y$ from $fx_0$ to $fx_1$ is the image under $f$ of some edge of $X$ from $x_0$ to $x_1$. Thus $\Ho \, f$ is full. To see that it is faithful, let $p, q \colon \Box^1 \to X$ be a pair of edges from $x_0$ to $x_1$, such that the morphisms in $\Ho \, Y(fx_0,fx_1)$ corresponding to $fp$ and $fq$ are equal. Applying \cref{either-dir-comp}, this implies that there is a $2$-cube in $Y$ of the form:

\centerline{
\xymatrix{
fx_0 \ar[d]_{fp} \ar@{=}[r] & fx_0 \ar[d]^{fq} \\
fx_1 \ar@{=}[r] & fx_1 
}
}

This $2$-cube corresponds to an edge from $fp$ to $fq$ in $\Map_{Y}(fx_0,fx_1)$; thus we have a commuting diagram:

\centerline{
\xymatrix{
\bd \Box^1 \ar[rr]^-{(p,q)} \ar@{^(->}[d] && \Map_{X}(x_0,x_1) \ar[d] \\
\Box^1 \ar[rr] && \Map_{Y}(fx_0,fx_1) \\
}
}

Since $\Map_{X}(x_0,x_1) \to \Map_{Y}(fx_0,fx_1)$ is a trivial fibration, this diagram has a lift, implying that $p = q$ in $\Ho \, X(x_0,x_1)$. Thus we see that $\Ho \, f$ is faithful.
\end{proof}

The Quillen equivalences $T \adjoint U$ and $Q_{W,\varepsilon} \adjoint \int\limits_{W,\varepsilon}$ relate the cubical homotopy category and mapping space constructions to their simplicial analogues. Before examining these relationships, we recall the corresponding constructions in $\sSet$, developed in \cite{joyal:theory-of-quasi-cats} and \cite{lurie:htt}.

\begin{definition}
Let $X \in \sSet$ be a quasicategory. The \emph{homotopy category} of $X$, denoted $\Ho X$, is defined as follows:

\begin{itemize}
  \item the objects of $\Ho X$ are the $0$-simplices of $X$;
  \item the morphisms from $x$ to $y$ in $\Ho X$ are the equivalence classes of edges $X_1(x,y)/\sim_X$, where $\sim$ denotes the equivalence relation defined by $f \sim g$ if and only if there exists a 2-simplex in $X$ of the form:

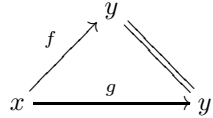
\centerline{
\xymatrix{
& y \ar@{=}[dr] \\
x \ar[ur]^{f} \ar[rr]^{g} && y \\
}
}  
  
  \item the identity map on $x \in X_0$ is given by $x \sigma_0$;
  \item the composition of $f \colon x \to y$ and $g \colon y \to z$ is given by filling the horn
  
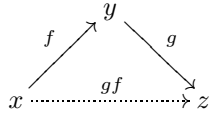
\centerline{
\xymatrix{
& y \ar[dr]^{g} \\
x \ar[ur]^{f} \ar@{.>}[rr]^{gf} && z\
}
}
\end{itemize}
\end{definition}

\begin{definition}
  Let $x_0$ and $x_1$ be $0$-simplices in a simplicial set $X$. The \emph{mapping space} from $x_0$ to $x_1$ is the simplicial set $\Hom_X(x_0, x_1)$ given by the following pullback.

\centerline{
\xymatrix{
\Hom_{X}(x_0,x_1) \drpullback \ar[d] \ar[rr] && X^{\Delta^1} \ar[d] \\
\Delta^0 \ar[rr]^{(x_0,x_1)} && X^{\bd \Delta^1} \\
}
}
\end{definition}

As in the cubical case, the simplicial mapping space admits the following concrete description:

  \[ \Hom_X(x_0, x_1)_n = \left\{  \Delta^{n} \times \Delta^1 \overset{s}{\to} X \ \big| \ s  \circ (\Delta^n \times \bd_{1-\varepsilon}) = x_\varepsilon  \right\}\text{,} \]
 with simplicial operations induced by those of $X$.

From this description we can see that, in contrast to the cubical case, the simplices of $\Hom_X(x_0,x_1)$ are not simplices of $X$. Thus it is often preferable to work with the \emph{left} and \emph{right mapping spaces} in a simplicial set, defined below.

\begin{definition}
Let $X \in \sSet, x_0, x_1 \colon \Delta^0 \to X$. The \emph{left mapping space} $\Hom^L_X(x_0,x_1)$ is defined by:

  \[ \Hom^L_X(x_0, x_1)_n = \left\{  \Delta^{n+1} \overset{s}{\to} X \ \big| s|_{\Delta^{\{0\}}} = x_0, \ s  \bd_{0} = x_1  \right\}\text{,} \]
  
  with simplicial operations induced by those of $X$, meaning that the face map $\bd_{i}$ of $\Hom^L_X(x_0,x_1)$ corresponds to the face map $\bd_{i+1}$ of $X$, and similarly for degeneracies.
  
  Similarly, the \emph{right mapping space} $\Hom^R_X(x_0,x_1)$ is defined by:

  \[ \Hom^R_X(x_0, x_1)_n = \left\{  \Delta^{n+1} \overset{s}{\to} X \ \big|   s  \bd_{n+1} = x_0, \ s|_{\Delta^{\{n+1\}}} = x_1 \right\}\text{,} \]
  
  with simplicial operations induced by those of $X$, meaning that the face map $\bd_{i}$ of $\Hom^R_X(x_0,x_1)$ corresponds to the face map $\bd_{i}$ of $X$, and similarly for degeneracies.
\end{definition}

A routine calculation shows:

\begin{lemma}\label{simplicial-hom-op}
For $X \in \sSet, x_0, x_1 \colon \Box^0 \to X$, we have a natural isomorphism $\Hom^{L}_{X}(x_0,x_1)^\op \cong \Hom^{R}_{X^\op}(x_1,x_0)$. \qed
\end{lemma}

\begin{lemma}\label{quasicat-U-Ho}
We have the following natural isomorphisms relating the homotopy categories of quasicategories and cubical quasicategories:

\begin{enumerate}
\item \label{Ho-cat-U} For a quasicategory $X$, $\Ho \, X \cong \Ho \, UX$;
\item \label{Ho-cat-int-1} For a cubical quasicategory $X$ and $W \in \{L,R\}$, $\Ho \, X \cong \Ho \, \int\limits_{W,1} X$;
\item \label{Ho-cat-int-0} For a cubical quasicategory $X$ and $W \in \{L,R\}$, $\Ho \, X \cong (\Ho \, \int\limits_{W,0} X)^\op$.
\end{enumerate}
\end{lemma}

\begin{proof}
For \ref{Ho-cat-U}, first note that $X$ and $UX$ have the same edges and vertices. The equivalence relations defining the morphisms of $\Ho X$ and $\Ho UX$ coincide by a simple argument involving \cref{either-dir-comp} and its simplicial analogue. A similar argument proves \ref{Ho-cat-int-1}, and \ref{Ho-cat-int-0} then follows from \cref{Q-involutions,Ho-involutions}. 
\end{proof}

\begin{lemma}\label{quasicat-U-Map}
For $X \in \sSet, x_0, x_1 \colon \Delta^0 \to X$, and $W \in \{L,R\}$, we have a natural isomorphism $U\Map^{W}_{X}(x_0,x_1) \cong \Hom_{UX}(x_0,x_1)$. 
\end{lemma}

\begin{proof}
Observe that the simplicial mapping space construction defines a functor $\Hom  \colon \bd {\Delta^1 \downarrow \sSet} \to \sSet$. An argument similar to the proof of \cref{Sigma-Map-adjoint} shows that this functor has a left adjoint $\Sigma \colon \sSet \to \bd \Delta^1 \downarrow \sSet$, given by the following pushout diagram:

\centerline{
\xymatrix{
X \sqcup X \ar[d] \ar[r] & \bd \Delta^1 \ar[d] \\
\Delta^1 \times X \ar[r] & \Sigma X \pushoutcorner \\
}
}

Thus we have the following square of adjunctions:

\centerline{
\xymatrix{
\bd \Box^1 \downarrow \cSet \ar@<1ex>[r] \ar@<1ex>[d] & \bd \Delta^1 \downarrow \sSet \ar@<1ex>[l] \ar@<1ex>[d] \\
\cSet \ar@<1ex>[u] \ar@<1ex>[r] & \sSet \ar@<1ex>[u] \ar@<1ex>[l] \\
}
}

We wish to show that the square of right adjoints commutes (up to natural isomorphism); for this, it suffices to show that the square of left adjoints commutes, i.e. that $T \Sigma_W \cong \Sigma T$. To see this, we may apply $T$ to the pushout square which defines the (left or right) suspension of a cubical set. Using \cref{T-prod} and the fact that $T$ preserves pushouts, we obtain a natural isomorphism $T \Sigma_{W} X \cong \Sigma T X$ for $X \in \cSet$.
\end{proof}

\begin{lemma}\label{int-Map}
For $X \in \cSet$, $x_0,x_1 \colon \Box^0 \to X$, we have the following natural isomorphisms:

\begin{multicols}{2}
\begin{itemize}
\item $\int\limits_{L,0} \Map^{L}_{X}(x_0,x_1) \cong \Hom^{R}_{\int\limits_{L,0} X}(x_1,x_0)$;
\item $\int\limits_{L,1} \Map^{L}_{X}(x_0,x_1) \cong \Hom^{R}_{\int\limits_{L,1} X}(x_0,x_1)$;
\item $\int\limits_{R,0} \Map^{R}_{X}(x_0,x_1) \cong \Hom^{R}_{\int\limits_{R,0} X}(x_1,x_0)$;
\item $\int\limits_{R,1} \Map^{R}_{X}(x_0,x_1) \cong \Hom^{R}_{\int\limits_{R,1} X}(x_0,x_1)$.
\end{itemize}
\end{multicols}
\end{lemma}

\begin{proof}
It suffices to prove the identity for $\int\limits_{R,1}$; the others then follow from \cref{Q-involutions,Map-involutions}. Observe that maps $\Delta^n \to \int\limits_{R,1} \Map^R_{X}(x_0,x_1)$ correspond to maps $\Sigma^{R} Q_{R,1}^{n} \to X$ mapping the basepoints $0 \mapsto x_0, 1 \mapsto x_1$. By the universal property of the pushout, these correspond to commuting diagrams of the form:

\centerline{
\xymatrix{
Q^{n}_{R,1} \sqcup Q^{n}_{R,1} \ar[r] \ar@{^(->}[d] & \bd \Box^1 \ar[d] \\
Q^{n}_{R,1} \otimes \Box^1 \ar[r] & X \\
}
}

In other words, these are maps $Q^{n}_{R,1} \otimes \Box^1 \to X$ such that for $\varepsilon \in \{0,1\}$, the subcomplex $Q^{n}_{R,1} \otimes \{\varepsilon\}$ is mapped to $x_{\varepsilon}$.

On the other hand, maps $s \colon \Delta^{n+1} \to \int\limits_{R,1} X$, i.e. $Q^{n+1}_{R,1} \to X$, which map the terminal vertex to $x_1$ correspond to commuting diagrams of the form:

\centerline{
\xymatrix{
Q^{n}_{R,1} \ar[r] \ar@{^(->}[d]_{Q^{n}_{R,1} \otimes \bd_{1,1}} & \Box^0 \ar[d]^{x_1} \\
Q^{n}_{R,1} \otimes \Box^1 \ar[r] & X \\
}
}

In other words, these are maps $Q^{n}_{R,1} \to X$ such that $Q^{n}_{R,1} \otimes \{1\}$ is mapped to $x_1$. By \cref{Q-cosimp-ob}, the condition $s \bd_{n+1} = x_0$ corresponds to the condition that $Q^{n}_{R,1} \otimes \{0\}$ is mapped to $x_0$.
\end{proof}

\begin{remark}
One may observe that applying a functor $\int\limits_{W,\varepsilon}$ to a compatible cubical mapping space always produces a simplicial right mapping space, regardless of the values of $W$ and $\varepsilon$. \cref{simplicial-hom-op} shows that the alternative definitions of $Q_{W,\varepsilon}$ discussed in \cref{monad-op} would instead produce formulas relating cubical mapping spaces to simplicial left mapping spaces.
\end{remark}

These results allow us to transfer \cref{fundamental-theorem} along the Quillen equivalence $T \adjoint U$, obtaining a new proof of the analogous result for the Joyal model structure on $\sSet$ that can be found, e.g., in \cite{rezk:quasicats}.

\begin{theorem}[{\cite[Props. 34.2 and 43.2]{rezk:quasicats}}]\label{simplicial-fundamental-theorem}
Let $f \colon X \to Y$ be a map between quasicategories. Then $f$ is a categorical equivalence if and only if the following two conditions are satisfied:

\begin{itemize}
\item $\Ho f \colon \Ho X \to \Ho Y$ is an equivalence of categories;
\item for all pairs of vertices $x_{0}, x_{1} \colon \Delta^0 \to X$, the induced map $\Hom_{X}(x_{0},x_{1}) \to \Hom_{Y}(fx_{0},fx_{1})$ is a homotopy equivalence in the Quillen model structure.
\end{itemize}

\begin{proof}
By \cref{QuillenEquivCreate-original,T-Quillen-equivalence}, $f$ is a categorical equivalence if and only if $Uf \colon UX \to UY$ is a categorical equivalence. Similarly, by \cref{QuillenEquivCreate-original,Quillen-Grothendieck-equiv-T}, each map $\Hom_{X}(x_{0},x_{1}) \to \Hom_{Y}(fx_{0},fx_{1})$ is a homotopy equivalence if and only if $U\Hom_{X}(x_{0},x_{1}) \to U\Hom_{Y}(fx_{0},fx_{1})$ is a homotopy equivalence. The stated result thus follows from \cref{fundamental-theorem}, together with \cref{quasicat-U-Ho,quasicat-U-Map}.
\end{proof}
\end{theorem} 

\begin{remark}
One can similarly prove \cref{fundamental-theorem} from \cref{simplicial-fundamental-theorem}, using \cref{Q-Quillen-equivalence,int-Map}, and the natural trivial cofibration $\Hom^{R}_{X}(x_0,x_1) \hookrightarrow \Hom_{X}(x_0,x_1)$ for quasicategories $X$.
\end{remark}

\appendix

\section{Verification of identities on $\theta$}\label{appendix:calculations}

Here we prove that the construction $\theta^{m,n}$ of \cref{ThetaDef} satisfies the identities of \cref{theta-construction}. Fix a cubical quasicategory $X$, $m \geq 2$, and $n \geq 0$, and assume that we have defined $\theta^{m',n'}$ satisfying all necessary identities for all pairs $m' \leq m, n' \leq n$ for which at least one of these inequalities is strict. Then we may define $\theta^{m,n}$ by the case analysis of \cref{ThetaDef}. We will show that a function $\theta^{m,n}$ defined in this way satisfies all identities of \cref{theta-construction}, starting with the identities involving faces.

Throughout these proofs we will conduct a number of case analyses, many of which will involve the standard forms of cones; for these computations we will often reduce the number of cases to be considered using \cref{sa1}.

\begin{proposition}\label{AB}
$\theta^{m,n}$ satisfies \ref{ThetaFace0} and \ref{ThetaFace0Id}; that is, for $x \colon C^{m,n} \to X$ and $i \leq n$, $\theta^{m,n}(x)\partial_{i,0} = \theta^{m,n-1}(x\partial_{i,0})$, while $\theta^{m,n}(x)\partial_{n+1,0} = x$.
\end{proposition}

\begin{proof}
We will prove this via a case analysis, based on the six cases of \cref{ThetaDef}. First, let $x = z\sigma_{a_{p}}$ in standard form, for $a_{p} \geq n+1$. By the induction hypotheses, for $m' < m$ or $m' = m$ and $n' < n$, $\theta^{m',n'}$ satisfies all the identities of \cref{theta-construction} (in future computations we will often use this assumption without comment). So for $i \leq n$ we have:

\begin{align*}
\theta^{m,n}(x)\partial_{i,0} & = \theta^{m-1,n}(z)\sigma_{a_{p}+1}\partial_{i,0} \\
& = \theta^{m-1,n}(z)\partial_{i,0}\sigma_{a_{p}} \\
& = \theta^{m-1,n-1}(z\partial_{i,0})\sigma_{a_{p}} \\
& = \theta^{m,n-1}(z\partial_{i,0}\sigma_{a_{p}-1}) \\
& = \theta^{m,n-1}(z\sigma_{a_{p}}\partial_{i,0}) \\
& = \theta^{m,n-1}(x\partial_{i,0}) \\
\end{align*}

And for $i = n+1$ we have:

\begin{align*}
\theta^{m,n}(x)\partial_{n+1,0} & = \theta^{m-1,n}(z)\sigma_{a_{p}+1}\partial_{n+1,0} \\
& = \theta^{m-1,n}(z)\partial_{n+1,0}\sigma_{a_{p}} \\
& = z\sigma_{a_{p}} \\
& = x
\end{align*}

Now suppose that the standard form of $x$ is $z\gamma_{b_{q},0}$, where $b_{q} \leq n-1$. Note that we must have $b_{q} \geq 1$, so this case can only occur when $n \geq 2$. Now for $i \leq b_{q} - 1$ we have:

\begin{align*}
\theta^{m,n}(x)\partial_{i,0} & = \theta^{m,n-1}(z)\gamma_{b_{q},0}\partial_{i,0} \\
& = \theta^{m,n-1}(z)\partial_{i,0}\gamma_{b_{q}-1,0} \\
& = \theta^{m,n-2}(z\partial_{i,0})\gamma_{b_{q}-1,0} \\
& = \theta^{m,n-1}(z\partial_{i,0}\gamma_{b_{q}-1,0}) \\
& = \theta^{m,n-1}(z\gamma_{b_{q},0}\partial_{i,0}) \\
& = \theta^{m,n-1}(x\partial_{i,0}) \\
\end{align*}
 
For $i = b_{q}$ or $i = b_{q} + 1$ we have:

\begin{align*}
\theta^{m,n}(x)\partial_{i,0} & = \theta^{m,n-1}(z)\gamma_{b_{q},0}\partial_{i,0} \\
& = \theta^{m,n-1}(z) \\
& = \theta^{m,n-1}(z\gamma_{b_{q},0}\partial_{i,0}) \\
& = \theta^{m,n-1}(x\partial_{i,0}) \\
\end{align*}
 
For $b_{q} + 2 \leq i \leq n$ we have:

\begin{align*}
\theta^{m,n}(x)\partial_{i,0} & = \theta^{m,n-1}(z)\gamma_{b_{q},0}\partial_{i,0} \\
& = \theta^{m,n-1}(z)\partial_{i-1,0}\gamma_{b_{q},0} \\
& = \theta^{m,n-2}(z\partial_{i-1,0})\gamma_{b_{q},0} \\
& = \theta^{m,n-1}(z\partial_{i-1,0}\gamma_{b_{q},0}) \\
& = \theta^{m,n-1}(z\gamma_{b_q,0}\partial_{i,0}) \\
& = \theta^{m,n-1}(x\partial_{i,0}) \\
\end{align*}

And for $i = n + 1$ we have:

\begin{align*}
\theta^{m,n}(x)\partial_{n+1,0} & = \theta^{m,n-1}(z)\gamma_{b_{q},0}\partial_{n+1,0} \\
& = \theta^{m,n-1}(z)\partial_{n,0}\gamma_{b_{q},0} \\
& = z\gamma_{b_{q},0} \\
& = x
\end{align*}

Next we consider the case where the standard form of $x$ is $z\gamma_{b_{q},\varepsilon}, b_{q} \geq  n + 1$. Then for $1 \leq i \leq n$ we have:

\begin{align*}
\theta^{m,n}(x)\partial_{i,0} & = \theta^{m-1,n}(z)\gamma_{b_{q}+1,\varepsilon}\partial_{i,0} \\
& = \theta^{m-1,n}(z)\partial_{i,0}\gamma_{b_{q},\varepsilon} \\
& = \theta^{m-1,n-1}(z\partial_{i,0})\gamma_{b_{q},\varepsilon} \\
& = \theta^{m,n-1}(z\partial_{i,0}\gamma_{b_{q}-1,\varepsilon}) \\
& = \theta^{m,n-1}(z\gamma_{b_{q},\varepsilon}\partial_{i,0}) \\
& = \theta^{m,n-1}(x\partial_{i,0}) \\
\end{align*}

And for $i = n + 1$ we have:

\begin{align*}
\theta^{m,n}(x)\partial_{n+1,0} & = \theta^{m-1,n}(z)\gamma_{b_{q}+1,\varepsilon}\partial_{n+1,0} \\
& = \theta^{m-1,n}(z)\partial_{n+1,0}\gamma_{b_{q},\varepsilon} \\
& = z\gamma_{b_{q},\varepsilon} \\
& = x
\end{align*}

Next, we consider case (4) of \cref{ThetaDef}: let $x$ be an $(m-1,n+1)$-cone not falling under any of cases (1)-(3). By \cref{ConeFaceDeg} \ref{Low0FaceOfCone}, every face $x\partial_{i,0}$ for $i \leq n$ is an $(m-1,n)$-cone, and therefore $\theta^{m,n-1}(x\partial_{i,0}) = x\partial_{i,0}\gamma_{n,0}$ by the induction hypothesis. Now, for $i \leq n$, we can compute:

\begin{align*}
\theta^{m,n}(x)\partial_{i,0} & = x\gamma_{n+1,0}\partial_{i,0} \\
& = x\partial_{i,0}\gamma_{n,0} \\
& = \theta^{m,n-1}(x\partial_{i,0}) \\
\end{align*}

And $\theta^{m,n}(x)\partial_{n+1,0} = x\gamma_{n+1,0}\partial_{n+1,0} = x$.

Next, we consider case (5): consider an $(m,n)$-cone $\theta^{m,n-1}(x')$ not falling under any of cases (1) through (4). Then for $1 \leq i \leq n - 1$ we have:

\begin{align*}
\theta^{m,n}(\theta^{m,n-1}(x'))\partial_{i,0} & = \theta^{m,n-1}(x')\gamma_{n,0}\partial_{i,0} \\
& = \theta^{m,n-1}(x')\partial_{i,0}\gamma_{n-1,0} \\
& = \theta^{m,n-2}(x'\partial_{i,0})\gamma_{n-1,0} \\
& = \theta^{m,n-1}(\theta^{m,n-2}(x'\partial_{i,0})) \\
& = \theta^{m,n-1}(\theta^{m,n-1}(x')\partial_{i,0}) \\
\end{align*}

For $i = n$ we have:

\begin{align*}
\theta^{m,n}(\theta^{m,n-1}(x))\partial_{n,0} & = \theta^{m,n-1}(x')\gamma_{n,0}\partial_{n,0} \\
& = \theta^{m,n-1}(x') \\
& = \theta^{m,n-1}(\theta^{m,n-1}(x')\partial_{n,0}) \\
\end{align*}

And for $i = n + 1$ we have $\theta^{m,n}(\theta^{m,n-1}(x'))\partial_{n+1,0} = \theta^{m,n-1}(x')\gamma_{n,0}\partial_{n+1,0} = \theta^{m,n-1}(x')$.

Finally, we consider case (6); in this case the identities hold by \cref{theta-lift}.
\end{proof}

\begin{proposition}\label{IdC}
$\theta^{m,n}$ satisfies \ref{ThetaFace1}; that is, for $x \colon C^{m,n} \to X$ and $i \geq n + 2$, we have $\theta^{m,n}(x)\partial_{i,1} = \theta^{m-1,n}(x\partial_{i-1,1})$. 
\end{proposition} 

\begin{proof}
Throughout the proof, we fix $i \geq n + 2$. First we consider case (1) of \cref{ThetaDef}. Suppose that the standard form of $x$ is $z\sigma_{a_{p}}$, for some $a_{p} \geq n+1$. Here we must consider various cases based on a comparison of $i$ with $a_{p}$. First suppose that $i \leq a_{p}$; note that this implies $a_{p} \geq n + 2$. Then we have: 

\begin{align*}
\theta^{m,n}(x)\partial_{i,1} & = \theta^{m-1,n}(z)\sigma_{a_{p}+1}\partial_{i,1} \\
& = \theta^{m-1,n}(z)\partial_{i,1}\sigma_{a_{p}} \\
& = \theta^{m-2,n}(z\partial_{i-1,1})\sigma_{a_{p}} \\
& = \theta^{m-1,n}(z\partial_{i-1,1}\sigma_{a_{p}-1}) \\
& = \theta^{m-1,n}(z\sigma_{a_{p}}\partial_{i-1,1}) \\
& = \theta^{m-1,n}(x\partial_{i-1,1}) \\
\end{align*}

To obtain the fourth equality, we have used \ref{ThetaDegen} and the fact that $a_{p} - 1 \geq n + 1$.

Next suppose that $i = a_{p} + 1$; then we have:

\begin{align*}
\theta^{m,n}(x)\partial_{a_{p}+1,1} & = \theta^{m-1,n}(z)\sigma_{a_{p}+1}\partial_{a_{p}+1,1} \\
& = \theta^{m-1,n}(z) \\
& = \theta^{m-1,n}(z\sigma_{a_{p}}\partial_{a_{p},1}) \\
& = \theta^{m-1,n}(x\partial_{a_{p},1}) \\
\end{align*}

Finally, suppose $i \geq a_{p} + 2$; note that this implies $i \geq n + 3$. Then we have:

\begin{align*}
\theta^{m,n}(x)\partial_{i,1} & = \theta^{m-1,n}(z)\sigma_{a_{p}+1}\partial_{i,1} \\
& = \theta^{m-1,n}(z)\partial_{i-1,1}\sigma_{a_{p}+1} \\
& = \theta^{m-2,n}(z\partial_{i-2,1})\sigma_{a_{p}+1} \\
& = \theta^{m-1,n}(z\partial_{i-2,1}\sigma_{a_{p}}) \\
& = \theta^{m-1,n}(z\sigma_{a_{p}}\partial_{i-1,1}) \\
& = \theta^{m-1,n}(x\partial_{i-1,1}) \\
\end{align*}

Next we consider case (2): suppose that $x = z\gamma_{b_{q},0}$ in standard form, where $b_{q} \leq n-1$. Then $i \geq b_{q}+3$, and we have:

\begin{align*}
\theta^{m,n}(x)\partial_{i,1} & = \theta^{m,n-1}(z)\gamma_{b_{q},0}\partial_{i,1} \\
& = \theta^{m,n-1}(z)\partial_{i-1,1}\gamma_{b_{q},0} \\
& = \theta^{m-1,n-1}(z\partial_{i-2,1})\gamma_{b_{q},0} \\
& = \theta^{m-1,n}(z\partial_{i-2,1}\gamma_{b_{q},0}) \\
& = \theta^{m-1,n}(z\gamma_{b_{q},0}\partial_{i-1,1}) \\
& = \theta^{m-1,n}(x\partial_{i-1,1}) \\
\end{align*}

Next we consider case (3): suppose that $x = z\gamma_{b_{q},\varepsilon}$ in standard form, where $b_{q} \geq n + 1$. Once again, we must perform a case analysis. First suppose that $i \leq b_{q}$, implying $b_{q} \geq n + 2$. Then we can compute:

\begin{align*}
\theta^{m,n}(x)\partial_{i,1} & = \theta^{m-1,n}(z)\gamma_{b_{q}+1,\varepsilon}\partial_{i,1} \\
& = \theta^{m-1,n}(z)\partial_{i,1}\gamma_{b_{q},\varepsilon} \\
& = \theta^{m-2,n}(z\partial_{i-1,1})\gamma_{b_{q},\varepsilon} \\
& = \theta^{m-1,n}(z\partial_{i-1,1}\gamma_{b_{q}-1,\varepsilon}) \\
& = \theta^{m-1,n}(z\gamma_{b_{q},\varepsilon}\partial_{i-1,1}) \\
& = \theta^{m-1,n}(x\partial_{i-1,1}) \\
\end{align*}

Next suppose that $i = b_{q} + 1$ or $b_{q} + 2$, and $\varepsilon = 0$. Then we have:

\begin{align*}
\theta^{m,n}(x)\partial_{i,1} & = \theta^{m-1,n}(z)\gamma_{b_{q}+1,0}\partial_{i,1} \\
& = \theta^{m-1,n}(z)\partial_{b_{q}+1,1}\sigma_{b_{q}+1} \\
& = \theta^{m-2,n}(z\partial_{b_{q},1})\sigma_{b_{q}+1} \\
& = \theta^{m-1,n}(z\partial_{b_{q},1}\sigma_{b_{q}}) \\
& = \theta^{m-1,n}(z\gamma_{b_{q},0}\partial_{i-1,1}) \\
& = \theta^{m-1,n}(x\partial_{i-1,1}) \\
\end{align*}
 
To obtain the third equality, we used \ref{ThetaFace1} for $\theta^{m-1,n}$ and the assumption that $b_{q} \geq n + 1$. Next suppose that $i = b_{q} + 1$ or $b_{q} + 2$, and $\varepsilon = 1$. Then we have:

\begin{align*}
\theta^{m,n}(x)\partial_{i,1} & = \theta^{m-1,n}(z)\gamma_{b_{q}+1,1}\partial_{i,1} \\
& = \theta^{m-1,n}(z) \\
& = \theta^{m-1,n}(z \gamma_{b_{q},1} \bd_{i-1,1}) \\
& = \theta^{m-1,n}(x \bd_{i-1,1}) \\ 
\end{align*}

Finally, suppose $i \geq b_{q} + 3$, implying $i \geq n + 4$. Then we have:

\begin{align*}
\theta^{m,n}(x)\partial_{i,1} & = \theta^{m-1,n}(z)\gamma_{b_{q}+1,\varepsilon}\partial_{i,1} \\
& = \theta^{m-1,n}(z)\partial_{i-1,1}\gamma_{b_{q}+1,\varepsilon} \\
& = \theta^{m-2,n}(z\partial_{i-2,1})\gamma_{b_{q}+1,\varepsilon} \\
& = \theta^{m-1,n}(z\partial_{i-2,1}\gamma_{b_{q},\varepsilon}) \\
& = \theta^{m-1,n}(z\gamma_{b_{q},\varepsilon}\partial_{i-1,1}) \\
& = \theta^{m-1,n}(x\partial_{i-1,1}) \\
\end{align*} 
 
Next we consider case (4): let $x$ be an $(m-1,n+1)$-cone not covered under any of cases (1) through (3). Then $x\partial_{i-1,1}$ is an $(m-2,n+1)$-cone by \cref{ConeFaceDeg} \ref{1FaceOfCone}, so $\theta^{m-1,n}(x\partial_{i-1,1}) = x\partial_{i-1,1}\gamma_{n+1,0}$ by\ref{ThetaConeWLOG} for $\theta^{m-1,n}$. Furthermore, note that by \cref{FaceCond}, $x\partial_{n+1,1} = x\partial_{m+n+1,0}...\partial_{n+1,1}\sigma_{n+1}...\sigma_{m+n}$. Using the cubical identities, we can rewrite this as $x\partial_{m+n+1,0}...\partial_{n+1,1}\sigma_{n+1}...\sigma_{n+1}$. Then for $i = n+2$, we can compute:

\begin{align*}
\theta^{m,n}(x)\partial_{n+2,1} & = x\gamma_{n+1,0}\partial_{n+2,1} \\
& = x\partial_{n+1,1}\sigma_{n+1} \\
& = x\partial_{m+n+1,0}...\partial_{n+1,1}\sigma_{n+1}...\sigma_{n+1}\sigma_{n+1} \\
& = x\partial_{m+n+1,0}...\partial_{n+1,1}\sigma_{n+1}...\sigma_{n+1}\gamma_{n+1,0} \\
& = x\partial_{n+1,1}\gamma_{n+1,0} \\
& = \theta^{m-1,n}(x\partial_{n+1,1}) \\
\end{align*}
 
While for $i \geq n + 3$, we have:

\begin{align*}
\theta^{m,n}(x)\partial_{i,1} & = x\gamma_{n+1,0}\partial_{i,1} \\
& = x\partial_{i-1,1}\gamma_{n+1,0} \\
& = \theta^{m-1,n}(x\partial_{i-1,1}) \\
\end{align*} 
 
Next we consider case (5). Let $x' \colon C^{m,n-1} \to X$, and consider $\theta^{m,n}(\theta^{m,n-1}(x'))$. Then we can compute:

\begin{align*}
\theta^{m,n}(\theta^{m,n-1}(x'))\partial_{i,1} & = \theta^{m,n-1}(x')\gamma_{n,0}\partial_{i,1} \\
& = \theta^{m,n-1}(x')\partial_{i-1,1}\gamma_{n,0} \\
& = \theta^{m-1,n-1}(x'\partial_{i-2,1})\gamma_{n,0} \\
& = \theta^{m-1,n}(\theta^{m-1,n-1}(x'\partial_{i-2,1})) \\
& = \theta^{m-1,n}(\theta^{m,n-1}(x')\partial_{i-1,1}) \\
\end{align*} 

Finally, in case (6), the identity holds by \cref{theta-lift}.
\end{proof}

Next we consider the identities involving degeneracies and connections.

\begin{proposition}\label{CDE}
$\theta^{m,n}$ satisfies \ref{ThetaDegen}, \ref{ThetaLowCon}, and \ref{ThetaHighCon}. That is:
\begin{itemize}
\item for $x \colon C^{m-1,n} \to X$ and $i \geq n + 1$, $\theta^{m,n}(x\sigma_{i}) = \theta^{m-1,n}(x)\sigma_{i+1}$;
\item for $x \colon C^{m,n-1} \to X$ and $i \leq n - 1$,  $\theta^{m,n}(x\gamma_{i,0}) = \theta^{m,n-1}(x)\gamma_{i,0}$;
\item for $x \colon C^{m-1,n} \to X$ and $i \geq n + 1$, $\theta^{m,n}(x\gamma_{i,\varepsilon}) = \theta^{m-1,n}(x)\gamma_{i+1,\varepsilon}$.
\end{itemize}
\end{proposition}

\begin{proof}
For each identity, we will perform a case analysis based on the standard form of $x$. For \ref{ThetaDegen}, consider an $(m,n)$-cone $x\sigma_{i}$, where $i \geq n + 1$ and the standard form of $x$ is $y\gamma_{b_{1},\varepsilon_1}...\gamma_{b_{q},\varepsilon_q}\sigma_{a_{1}}...\sigma_{a_{p}}$. If the string of degeneracy maps in the standard form of $x$ is empty, or $a_{p} < i$, then the standard form of $x\sigma_{i}$ ends with $\sigma_{i}$, so $\theta^{m,n}(x\sigma_{i}) = \theta^{m-1,n}(x)\sigma_{i+1}$ by definition. So suppose that $a_{p} \geq i$. Then:

\begin{align*}
\theta^{m,n}(x\sigma_{i}) & = \theta^{m,n}(y\gamma_{b_{1},\varepsilon_1}...\gamma_{b_{q},\varepsilon_q}\sigma_{a_{1}}...\sigma_{a_{p}}\sigma_{i}) \\
& = \theta^{m,n}(y\gamma_{b_{1},\varepsilon_1}...\gamma_{b_{q},\varepsilon_q}\sigma_{a_{1}}...\sigma_{a_{p-1}}\sigma_{i}\sigma_{a_{p}+1}) \\
\end{align*}

By assumption, all the indices $a_{1},...,a_{p-1}$, are less than $a_{p}$. Rearranging the expression on the right-hand side of the equation into standard form using the cubical identities will not increase any of these indices by more than 1, so the rightmost map in the standard form of $x\sigma_{i}$, i.e. the degeneracy map with the highest index, is $\sigma_{a_{p}+1}$. Therefore, we can compute:

\begin{align*}
\theta^{m,n}(y\gamma_{b_{1},\varepsilon_1}...\gamma_{b_{q},\varepsilon_q}\sigma_{a_{1}}...\sigma_{a_{p-1}}\sigma_{i}\sigma_{a_{p}+1}) & = \theta^{m-1,n}(y\gamma_{b_{1},\varepsilon_1}...\gamma_{b_{q},\varepsilon_q}\sigma_{a_{1}}...\sigma_{a_{p-1}}\sigma_{i})\sigma_{a_{p}+2} \\
& = \theta^{m-2,n}(y\gamma_{b_{1},\varepsilon_1}...\gamma_{b_{q},\varepsilon_q}\sigma_{a_{1}}...\sigma_{a_{p-1}})\sigma_{i+1}\sigma_{a_{p}+2} \\
& = \theta^{m-2,n}(y\gamma_{b_{1},\varepsilon_1}...\gamma_{b_{q},\varepsilon_q}\sigma_{a_{1}}...\sigma_{a_{p-1}})\sigma_{a_{p}+1}\sigma_{i+1} \\
& = \theta^{m-1,n}(y\gamma_{b_{1},\varepsilon_1}...\gamma_{b_{q},\varepsilon_q}\sigma_{a_{1}}...\sigma_{a_{p-1}}\sigma_{a_{p}})\sigma_{i+1} \\
& = \theta^{m-1,n}(x)\sigma_{i+1} \\
\end{align*}

So $\theta^{m,n}$ satisfies \ref{ThetaDegen}.

Next we will verify \ref{ThetaHighCon}. Consider an $(m,n)$-cone $x\gamma_{i,\varepsilon}$, where $i \geq n + 1$ and the standard form of $x$ is as above. If this standard form contains no degeneracy maps, and either $b_{q} < i$, $b_q = i$ while $\varepsilon_q \neq \varepsilon$, or $x$ is non-degenerate, then the standard form of $x\gamma_{i,\varepsilon}$ ends with $\gamma_{i,\varepsilon}$, so the identity holds by definition. The remaining possibilities for the standard form of $x$ can be divided into various cases. First, suppose that the string of degeneracy maps in the standard form of $x$ is non-empty, i.e. $x = z\sigma_{a_{p}}$ in standard form. By \cref{ConeFaceDeg}, $x = x\gamma_{i,\varepsilon}\partial_{i,\varepsilon}$ is an $(m-1,n)$-cone, so $a_{p} \geq n+1$ by \cref{sa1} \ref{sa1-degen}. Now we must break this into further cases based on a comparison between $i$ and $a_{p}$. If $i < a_{p}$ then, using the cubical identities, \ref{ThetaDegen} for $\theta^{m,n}$, and \ref{ThetaHighCon} for $\theta^{m-1,n}$, we can compute:

\begin{align*}
\theta^{m,n}(x\gamma_{i,\varepsilon}) & = \theta^{m,n}(z\sigma_{a_{p}}\gamma_{i,\varepsilon}) \\
& = \theta^{m,n}(z\gamma_{i,\varepsilon}\sigma_{a_{p}+1}) \\
& = \theta^{m-1,n}(z\gamma_{i,\varepsilon})\sigma_{a_{p}+2} \\
& = \theta^{m-2,n}(z)\gamma_{i+1,\varepsilon}\sigma_{a_{p}+2} \\
& = \theta^{m-2,n}(z)\sigma_{a_{p}+1}\gamma_{i+1,\varepsilon} \\
& = \theta^{m-1,n}(z\sigma_{a_{p}})\gamma_{i+1,\varepsilon} \\
& = \theta^{m-1,n}(x)\gamma_{i+1,\varepsilon} \\
\end{align*}

Next we consider the case $i = a_{p}$:

\begin{align*}
\theta^{m,n}(x\gamma_{a_{p},\varepsilon}) & = \theta^{m,n}(z\sigma_{a_{p}}\gamma_{a_{p},\varepsilon}) \\
& = \theta^{m,n}(z\sigma_{a_{p}}\sigma_{a_{p}+1}) \\
& = \theta^{m-1,n}(z\sigma_{a_{p}})\sigma_{a_{p}+2} \\
& = \theta^{m-2,n}(z)\sigma_{a_{p}+1}\sigma_{a_{p}+2} \\
& = \theta^{m-2,n}(z)\sigma_{a_{p}+1}\gamma_{a_{p}+1,\varepsilon} \\
& = \theta^{m-1,n}(z\sigma_{a_{p}})\gamma_{a_{p}+1,\varepsilon} \\
& = \theta^{m-1,n}(x)\gamma_{a_{p}+1,\varepsilon} \\
\end{align*}

Now we consider the case $i > a_{p}$. Note that this implies $i \geq n+2$, so $i-1 \geq n+1$. Thus we can compute:

\begin{align*}
\theta^{m,n}(x\gamma_{i,\varepsilon}) & = \theta^{m,n}(z\sigma_{a_{p}}\gamma_{i,\varepsilon}) \\
& = \theta^{m,n}(z\gamma_{i-1,\varepsilon}\sigma_{a_{p}}) \\
& = \theta^{m-1,n}(z\gamma_{i-1,\varepsilon})\sigma_{a_{p}+1} \\
& = \theta^{m-2,n}(z)\gamma_{i,\varepsilon}\sigma_{a_{p}+1} \\
& = \theta^{m-2,n}(z)\sigma_{a_{p}+1}\gamma_{i+1,\varepsilon} \\
& = \theta^{m-1,n}(z\sigma_{a_{p}})\gamma_{i+1,\varepsilon} \\
& = \theta^{m-1,n}(x)\gamma_{i+1,\varepsilon} \\
\end{align*}

Next we will verify \ref{ThetaHighCon} in the case where the standard form of $x$ contains no degeneracy maps, and either $i < b_{q}$ or $i = b_{q}$ and $\varepsilon = \varepsilon_q$. In this case we can compute:

\begin{align*}
\theta^{m,n}(x\gamma_{i,\varepsilon}) & = \theta^{m,n}(y\gamma_{b_{1},\varepsilon_1}...\gamma_{b_{q},\varepsilon_q}\gamma_{i,\varepsilon}) \\
& = \theta^{m,n}(y\gamma_{b_{1},\varepsilon_1}...\gamma_{b_{q-1},\varepsilon_{q-1}}\gamma_{i,\varepsilon}\gamma_{b_{q}+1,\varepsilon_{q}}) \\
\end{align*}

Similarly to what we saw when verifying \ref{ThetaDegen}, after we have rearranged the expression on the right-hand side of this equation into standard form, the rightmost map in the expression will still be $\gamma_{b_{q}+1,\varepsilon_{q}}$. Thus we can apply the definition of $\theta^{m,n}$ to compute:

\begin{align*}
\theta^{m,n}(y\gamma_{b_{1},\varepsilon}...\gamma_{b_{q-1},\varepsilon_{q-1}}\gamma_{i,\varepsilon}\gamma_{b_{q}+1,\varepsilon_q}) & = \theta^{m-1,n}(y\gamma_{b_{1}\varepsilon_1}...\gamma_{b_{q-1},\varepsilon_{q-1}}\gamma_{i,\varepsilon_i})\gamma_{b_{q}+2,\varepsilon_{q}} \\
& = \theta^{m-2,n}(y\gamma_{b_{1},\varepsilon_1}...\gamma_{b_{q-1},\varepsilon_{q-1}})\gamma_{i+1,\varepsilon}\gamma_{b_{q}+2,\varepsilon_q} \\
& = \theta^{m-2,n}(y\gamma_{b_{1},\varepsilon_1}...\gamma_{b_{q-1},\varepsilon_{q-1}})\gamma_{b_{q}+1,\varepsilon_q}\gamma_{i+1,\varepsilon} \\
& = \theta^{m-1,n}(y\gamma_{b_{1},\varepsilon_1}...\gamma_{b_{q},\varepsilon_q})\gamma_{i+1,\varepsilon} \\
& = \theta^{m-1,n}(x)\gamma_{i+1,\varepsilon} \\
\end{align*}

Thus $\theta^{m,n}$ satisfies \ref{ThetaHighCon}.

Finally we will verify \ref{ThetaLowCon}. Consider an $(m,n)$-cone $x\gamma_{i,0}$, where $i \leq n - 1$ and the standard form of $x$ is as above. Once again, we must consider several possible cases based on the standard form of $x$. As with \ref{ThetaHighCon}, if the standard form of $x$ contains no degeneracy maps, and either $b_{q} < i$, $b_{q} = i$ while $\varepsilon_{q} = 1$, or $x$ is non-degenerate, then $\gamma_{i,0}$ is the rightmost map in the standard form of $x\gamma_{i,0}$, and the identity holds by definition. Once again, the remaining cases will require computation.

As above, we begin with the case where the string of degeneracy maps in the standard form of $x$ is non-empty. By \cref{ConeFaceDeg} \ref{Low0FaceOfCone}, $x = x\gamma_{i,0}\partial_{i,0}$ is an $(m,n-1)$-cone, so $a_{p} \geq n$ by \cref{sa1}. Then, using the cubical identities, \ref{ThetaDegen} for $\theta^{m,n}$, and \ref{ThetaLowCon} for $\theta^{m-1,n}$, we can compute:

\begin{align*}
\theta^{m,n}(x\gamma_{i,0}) & = \theta^{m,n}(z\sigma_{a_{p}}\gamma_{i,0}) \\
& = \theta^{m,n}(z\gamma_{i,0}\sigma_{a_{p}+1}) \\
& = \theta^{m-1,n}(z\gamma_{i,0})\sigma_{a_{p}+2} \\
& = \theta^{m-1,n-1}(z)\gamma_{i,0}\sigma_{a_{p}+2} \\
& = \theta^{m-1,n-1}(z)\sigma_{a_{p}+1}\gamma_{i,0} \\
& = \theta^{m,n-1}(z\sigma_{a_{p}})\gamma_{i,0} \\
& = \theta^{m,n-1}(x)\gamma_{i,0}
\end{align*}

Next we consider the cases in which the standard form of $x$ contains no degeneracy maps; first, suppose that $b_{q} \geq n$. Then, using the cubical identities, \ref{ThetaHighCon} for $\theta^{m,n}$, and \ref{ThetaLowCon} for $\theta^{m-1,n}$, we can compute:

\begin{align*}
\theta^{m,n}(x\gamma_{i,0}) & = \theta^{m,n}(z\gamma_{b_{q},\varepsilon_q}\gamma_{i,0}) \\
& = \theta^{m,n}(z\gamma_{i,0}\gamma_{b_{q}+1,\varepsilon_q}) \\
& = \theta^{m-1,n}(z\gamma_{i,0})\gamma_{b_{q}+2,\varepsilon_q} \\
& = \theta^{m-1,n-1}(z)\gamma_{i,0}\gamma_{b_{q}+2,\varepsilon_q} \\
& = \theta^{m-1,n-1}(z)\gamma_{b_{q}+1,\varepsilon_q}\gamma_{i,0} \\
& = \theta^{m,n-1}(z\gamma_{b_{q},\varepsilon_q})\gamma_{i,0} \\
& = \theta^{m,n-1}(x)\gamma_{i,0} \\
\end{align*} 

Next we consider the case $b_{q} = n - 1$. Note that $x = x \gamma_{i,0} \bd_{i,0}$ is an $(m,n-1)$-cone by \cref{ConeFaceDeg} \ref{Low0FaceOfCone}; thus $\varepsilon_q = 0$ by \cref{sa1} \ref{sa1-con}. Here we can compute:

\begin{align*}
x\gamma_{i,0} & = y\gamma_{b_{1},\varepsilon_1}...\gamma_{b_{q-1},\varepsilon_{q-1}}\gamma_{n-1,0}\gamma_{i,0} \\
& = y\gamma_{b_{1},\varepsilon_1}...\gamma_{b_{q-1},\varepsilon_{q-1}}\gamma_{i,0}\gamma_{n,0} \\
\end{align*}

As in previous cases, after rearranging this expression into standard form, the rightmost map will still be $\gamma_{n,0}$. Thus $x\gamma_{i,0}$ belongs to case (4) by \cref{cn}, so:

\begin{align*}
\theta^{m,n}(x\gamma_{i,0}) & = x\gamma_{i,0}\gamma_{n+1,0} \\
& = x\gamma_{n,0}\gamma_{i,0} \\
\end{align*}

By \cref{ConeFaceDeg} \ref{Low0FaceOfCone}, $x = x\gamma_{i,0}\partial_{i,0}$ is an $(m,n-1)$-cone, so the fact that $b_{q} = n - 1$ implies that $x$ also belongs to case (4). Thus $x\gamma_{n,0} = \theta^{m,n-1}(x)$, so \ref{ThetaLowCon} is satisfied in this case. 

Finally, we consider the case $i \leq b_{q} \leq n - 2$. Once again, we have $\varepsilon_{q} = 0$ by \cref{sa1} \ref{sa1-con}. Now we can compute:

\begin{align*}
\theta^{m,n}(x\gamma_{i,0}) & = \theta^{m,n}(y\gamma_{b_{1},\varepsilon_1}...\gamma_{b_{q},0}\gamma_{i,0}) \\
& = \theta^{m,n}(y\gamma_{b_{1},\varepsilon_1}...\gamma_{b_{q-1},\varepsilon_{q-1}}\gamma_{i,0}\gamma_{b_{q}+1,0}) \\
\end{align*}

As in previous computations, once the expression on the right-hand side of the equation has been rearranged into standard form, its rightmost map will still be $\gamma_{b_{q}+1,0}$. By assumption, $b_{q}+1 \leq n - 1$, so using the cubical identities, the definition of $\theta^{m,n}$, and \ref{ThetaLowCon} for $\theta^{m,n-1}$, we can compute:

\begin{align*}
\theta^{m,n}(y\gamma_{b_{1},\varepsilon_1}...\gamma_{b_{q-1},\varepsilon_{q-1}}\gamma_{i,0}\gamma_{b_{q}+1,0}) & = \theta^{m,n-1}(y\gamma_{b_{1},\varepsilon_1}...\gamma_{b_{q-1},\varepsilon_{q-1}}\gamma_{i,0})\gamma_{b_{q}+1,0} \\
& = \theta^{m,n-2}(y\gamma_{b_{1},\varepsilon_1}...\gamma_{b_{q-1},\varepsilon_{q-1}})\gamma_{i,0}\gamma_{b_{q}+1,0} \\
& = \theta^{m,n-2}(y\gamma_{b_{1},\varepsilon_1}...\gamma_{b_{q-1},\varepsilon_{q-1}})\gamma_{b_{q},0}\gamma_{i,0} \\
& = \theta^{m,n-1}(y\gamma_{b_{1},\varepsilon_1}...\gamma_{b_{q},0})\gamma_{i,0} \\
& = \theta^{m,n-1}(x)\gamma_{i,0} \\
\end{align*}

Thus $\theta^{m,n}$ satisfies \ref{ThetaLowCon}.
\end{proof}

\begin{proposition}\label{IdF}
$\theta^{m,n}$ satisfies \ref{ThetaTheta}. That is, for $x \colon C^{m,n-1} \to X$, $\theta^{m,n}(\theta^{m,n-1}(x)) = \theta^{m,n-1}(x)\gamma_{n,0}$.
\end{proposition}

\begin{proof}
We proceed by a case analysis on $x$, based on the cases of \cref{ThetaDef}. In our computations, we will freely use the identities for $\theta^{m,n}$ which we have already proven. First suppose that $x = z\sigma_{a_{p}}$ in standard form, for some $a_{p} \geq n$. Then we can compute:

\begin{align*}
\theta^{m,n}(\theta^{m,n-1}(x)) & = \theta^{m,n}(\theta^{m,n-1}(z\sigma_{a_{p}})) \\
& = \theta^{m,n}(\theta^{m-1,n-1}(z)\sigma_{a_{p}+1}) \\
& = \theta^{m-1,n}(\theta^{m-1,n-1}(z))\sigma_{a_{p}+2} \\
& = \theta^{m-1,n-1}(z)\gamma_{n,0}\sigma_{a_{p}+2} \\
& = \theta^{m-1,n-1}(z)\sigma_{a_{p}+1}\gamma_{n,0} \\
& = \theta^{m,n-1}(z\sigma_{a_{p}})\gamma_{n,0} \\
& = \theta^{m,n-1}(x)\gamma_{n,0} \\
\end{align*}

Next let the standard form of $x$ be $z\gamma_{b_{q},0}$ where $b_{q} \leq n-2$. Then we can compute:

\begin{align*}
\theta^{m,n}(\theta^{m,n-1}(x)) & = \theta^{m,n}(\theta^{m,n-1}(z\gamma_{b_{q},0})) \\
& = \theta^{m,n}(\theta^{m,n-2}(z)\gamma_{b_{q},0}) \\
& = \theta^{m,n-1}(\theta^{m,n-2}(z))\gamma_{b_{q},0} \\
& = \theta^{m,n-2}(z)\gamma_{n-1,0}\gamma_{b_{q},0} \\
& = \theta^{m,n-2}(z)\gamma_{b_{q},0}\gamma_{n,0} \\
& = \theta^{m,n-1}(z\gamma_{b_{q},0})\gamma_{n,0} \\
& = \theta^{m,n-1}(x)\gamma_{n,0} \\
\end{align*}

Now let the standard form of $x$ be $z\gamma_{b_{q},\varepsilon}$ where $b_{q} \geq n$. Then we can compute:

\begin{align*}
\theta^{m,n}(\theta^{m,n-1}(x)) & = \theta^{m,n}(\theta^{m,n-1}(z\gamma_{b_{q},\varepsilon})) \\
& = \theta^{m,n}(\theta^{m-1,n-1}(z)\gamma_{b_{q}+1,\varepsilon}) \\
& = \theta^{m-1,n}(\theta^{m-1,n-1}(z))\gamma_{b_{q}+2,\varepsilon} \\
& = \theta^{m-1,n-1}(z)\gamma_{n,0}\gamma_{b_{q}+2,\varepsilon} \\
& = \theta^{m-1,n-1}(z)\gamma_{b_{q}+1,\varepsilon}\gamma_{n,0} \\
& = \theta^{m,n-1}(z\gamma_{b_{q},\varepsilon})\gamma_{n,0} \\
& = \theta^{m,n-1}(x)\gamma_{n,0} \\
\end{align*}

Next, we consider case (4): suppose that $x$ is an $(m-1,n)$-cone not falling under any of cases (1) through (3) (when considered as an $(m,n-1)$-cone). Then $\theta^{m,n-1}(x) = x\gamma_{n,0}$. The assumption that $x$ does not belong to any of cases (1) through (3), together with \cref{sa1}, implies that either it is non-degenerate, or its standard form ends with $\gamma_{n-1,0}$. Either way, the standard form of $x\gamma_{n,0}$ ends with $\gamma_{n,0}$, so it falls under case (4) by \cref{cn}. Thus we can compute:

\begin{align*}
\theta^{m,n}(\theta^{m,n-1}(x)) & = \theta^{m,n}(x\gamma_{n,0}) \\
& = x\gamma_{n,0}\gamma_{n+1,0} \\
& = x\gamma_{n,0}\gamma_{n,0} \\
& = \theta^{m,n-1}(x)\gamma_{n,0} \\
\end{align*}

Next we consider case (5): suppose that $x = \theta^{m,n-2}(x')$ for some $x' \colon C^{m,n-2} \to X$. Then we can compute:

\begin{align*}
\theta^{m,n}(\theta^{m,n-1}(x)) & = \theta^{m,n}(\theta^{m,n-1}(\theta^{m,n-2}(x'))) \\
& = \theta^{m,n}(\theta^{m,n-2}(x')\gamma_{n-1,0}) \\
& = \theta^{m,n-1}(\theta^{m,n-2}(x'))\gamma_{n-1,0} \\
& = \theta^{m,n-2}(x')\gamma_{n-1,0}\gamma_{n-1,0} \\
& = \theta^{m,n-2}(x')\gamma_{n-1,0}\gamma_{n,0} \\ 
& = \theta^{m,n-1}(\theta^{m,n-2}(x'))\gamma_{n,0} \\
& = \theta^{m,n-1}(x)\gamma_{n,0} \\
\end{align*}

Finally, suppose $x$ falls under case (6). Then by \cref{Theta-T}, $\theta^{m,n-1}(x)$ falls under case (5), so $\theta^{m,n}(\theta^{m,n-1}(x)) = \theta^{m,n-1}(x)\gamma_{n,0}$ by definition.
\end{proof}

\begin{proposition}\label{IdG}
$\theta^{m,n}$ satisfies \ref{ThetaConeWLOG}. That is, for $x \colon C^{m-1,n+1} \to X$, $\theta^{m,n}(x) = x\gamma_{n+1,0}$.
\end{proposition}

\begin{proof}
As in previous proofs, we proceed via case analysis on $x$, based on the cases of \cref{ThetaDef}. First suppose that $x$ is an $(m-1,n+1)$-cone whose standard form is $z\sigma_{a_{p}}$. By \cref{sa1} \ref{sa1-degen}, $a_{p} \geq n + 2$. Therefore, by \cref{ConeFaceDeg} \ref{High0FaceOfCone}, $x\partial_{a_{p},0} = z$ is an $(m-2,n+1)$-cone, so $\theta^{m-1,n}(z) = z\gamma_{n+1,0}$ by \ref{ThetaConeWLOG} for $\theta^{m-1,n}$. Thus we can compute:

\begin{align*}
\theta^{m,n}(x) & = \theta^{m-1,n}(z)\sigma_{a_{p}+1} \\
& = z\gamma_{n+1,0}\sigma_{a_{p}+1} \\
& = z\sigma_{a_{p}}\gamma_{n+1,0} \\
& = x\gamma_{n+1,0} \\
\end{align*}

Now let $x$ be an $(m-1,n+1)$-cone whose standard form is $z\gamma_{b_{q},0}, b_{q} \leq n - 1$. Then by \cref{ConeFaceDeg} \ref{Low0FaceOfCone}, $x\partial_{b_{q},0} = z$ is an $(m-1,n)$-cone. So by \ref{ThetaConeWLOG} for $\theta^{m,n-1}$, we have $\theta^{m,n-1}(z) = z\gamma_{n,0}$. Thus we can compute:

\begin{align*}
\theta^{m,n}(x) & = \theta^{m,n-1}(z)\gamma_{b_{q},0} \\
& = z\gamma_{n,0}\gamma_{b_{q},0}  \\
& = z\gamma_{b_{q},0}\gamma_{n+1,0} \\
& = x\gamma_{n+1,0} \\
\end{align*}

Next let $x$ be an $(m-1,n+1)$-cone whose standard form is $z\gamma_{b_{q},\varepsilon}$, where $b_{q} \geq n + 1$. (Note that if $b_{q} = n + 1$, then we may assume $\varepsilon = 0$ by \cref{sa1} \ref{sa1-con}.) Then by \cref{ConeFaceDeg}, $x\partial_{b_{q}+1,\varepsilon} = z$ is an $(m-2,n+1)$-cone, so $\theta^{m-1,n}(z) = z\gamma_{n+1,0}$ by \ref{ThetaConeWLOG} for $\theta^{m-1,n}$. Thus we can compute:

\begin{align*}
\theta^{m,n}(x) & = \theta^{m-1,n}(z)\gamma_{b_{q}+1,\varepsilon} \\
& = z\gamma_{n+1,0}\gamma_{b_{q}+1,\varepsilon} \\
& = z\gamma_{b_{q},\varepsilon}\gamma_{n+1,0} \\
& = x\gamma_{n+1,0} \\
\end{align*}

Finally, case (4) consists of all $(m-1,n+1)$-cones not falling under any of the previous cases, and in this case \ref{ThetaConeWLOG} holds by definition.
\end{proof}

\bibliographystyle{amsalphaurlmod}
\bibliography{general-bibliography-2}

\providecommand{\bysame}{\leavevmode\hbox to3em{\hrulefill}\thinspace}
\providecommand{\MR}{\relax\ifhmode\unskip\space\fi MR }
\providecommand{\MRhref}[2]{%
  \href{http://www.ams.org/mathscinet-getitem?mr=#1}{#2}
}
\providecommand{\href}[2]{#2}
\begin{thebibliography}{HKRS17}

\bibitem[Ara14]{ara:higher-quasicats}
Dimitri Ara, \emph{Higher quasi-categories vs higher {R}ezk spaces}, J.
  K-Theory \textbf{14} (2014), no.~3, 701--749, \href
  {http://dx.doi.org/10.1017/S1865243315000021}
  {\path{doi:10.1017/S1865243315000021}},
  \url{https://doi.org/10.1017/S1865243315000021}.

\bibitem[Bek00]{beke:sheafifiable}
Tibor Beke, \emph{Sheafifiable homotopy model categories}, Math. Proc.
  Cambridge Philos. Soc. \textbf{129} (2000), no.~3, 447--475, \href
  {http://dx.doi.org/10.1017/S0305004100004722}
  {\path{doi:10.1017/S0305004100004722}},
  \url{https://doi.org/10.1017/S0305004100004722}.

\bibitem[BM17]{buchholtz-morehouse:varieties-of-cubes}
Ulrik Buchholtz and Edward Morehouse, \emph{Varieties of cubical sets},
  Relational and Algebraic Methods in Computer Science. RAMICS 2017. (Cham),
  vol. 10226, Springer, 2017, pp.~77--92.

\bibitem[Cis06]{CisinskiAsterisque}
Denis-Charles Cisinski, \emph{Les pr\'efaisceaux comme mod\`eles des types
  d'homotopie}, Ast\'erisque (2006), no.~308, xxiv+390.

\bibitem[Cis14]{CisinskiUniverses}
\bysame, \emph{Univalent universes for elegant models of homotopy types},
  Preprint, 2014.

\bibitem[Cis19]{cisinski:higher-categories-book}
\bysame, \emph{Higher categories and homotopical algebra}, Cambridge Studies in
  Advanced Mathematics, Cambridge University Press, Cambridge, 2019,
  \url{http://www.mathematik.uni-regensburg.de/cisinski/CatLR.pdf}.

\bibitem[CKM20]{campion-kapulkin-maehara}
Timothy Campion, Krzysztof Kapulkin, and Yuki Maehara, \emph{A cubical model
  for $(\infty,n)$-categories}, Submitted, 2020.

\bibitem[GM03]{grandis-mauri}
Marco Grandis and Luca Mauri, \emph{Cubical sets and their site}, Theory and
  Applications of Categories \textbf{11} (2003), no.~8, 185–211.

\bibitem[HKRS17]{hkrs:transfer-thm}
Kathryn Hess, Magdalena K\c{e}dziorek, Emily Riehl, and Brooke Shipley, \emph{A
  necessary and sufficient condition for induced model structures}, J. Topol.
  \textbf{10} (2017), no.~2, 324--369, \href
  {http://dx.doi.org/10.1112/topo.12011} {\path{doi:10.1112/topo.12011}}.

\bibitem[Hov99]{hovey:book}
Mark Hovey, \emph{Model categories}, Mathematical Surveys and Monographs,
  vol.~63, American Mathematical Society, Providence, RI, 1999.

\bibitem[Jar06]{jardine:categorical-homotopy-theory}
J.~F. Jardine, \emph{Categorical homotopy theory}, Homology Homotopy Appl.
  \textbf{8} (2006), no.~1, 71--144,
  \url{http://projecteuclid.org/euclid.hha/1140012467}.

\bibitem[Joy02]{joyal:qcat-kan}
Andr{\'e} Joyal, \emph{Quasi-categories and {K}an complexes}, J. Pure Appl.
  Algebra \textbf{175} (2002), no.~1-3, 207--222, Special volume celebrating
  the 70th birthday of Professor Max Kelly, \href
  {http://dx.doi.org/10.1016/S0022-4049(02)00135-4}
  {\path{doi:10.1016/S0022-4049(02)00135-4}}.

\bibitem[Joy09]{joyal:theory-of-quasi-cats}
\bysame, \emph{The theory of quasi-categories and its applications}, Vol.~II of
  course notes from Simplicial Methods in Higher Categories, Centra de Recerca
  Matem{\`a}tica, Barcelona, 2008, 2009,
  \url{http://mat.uab.cat/~kock/crm/hocat/advanced-course/Quadern45-2.pdf}.

\bibitem[JT07]{joyal-tierney:qcat-vs-segal}
Andr{\'e} Joyal and Myles Tierney, \emph{Quasi-categories vs {S}egal spaces},
  Categories in algebra, geometry and mathematical physics, Contemp. Math.,
  vol. 431, Amer. Math. Soc., Providence, RI, 2007, pp.~277--326, \href
  {http://dx.doi.org/10.1090/conm/431/08278}
  {\path{doi:10.1090/conm/431/08278}}.

\bibitem[KLW19]{kapulkin-lindsey-wong}
Krzysztof Kapulkin, Zachery Lindsey, and Liang~Ze Wong, \emph{A co-reflection
  of cubical sets into simplicial sets with applications to model structures},
  New York J. Math. \textbf{25} (2019), 627--641.

\bibitem[KV20]{kapulkin-voevodsky:cubical-straightening}
Krzysztof Kapulkin and Vladimir Voevodsky, \emph{A cubical approach to
  straightening}, J. Topol. \textbf{13} (2020), no.~4, 1682--1700, \href
  {http://dx.doi.org/10.1112/topo.12173} {\path{doi:10.1112/topo.12173}},
  \url{https://doi.org/10.1112/topo.12173}.

\bibitem[Lur09]{lurie:htt}
Jacob Lurie, \emph{Higher topos theory}, Annals of Mathematics Studies, vol.
  170, Princeton University Press, Princeton, NJ, 2009,
  \url{http://www.math.harvard.edu/~lurie/papers/croppedtopoi.pdf}.

\bibitem[Mal09]{maltsiniotis:connections-strict-test-cat}
Georges Maltsiniotis, \emph{La cat\'{e}gorie cubique avec connexions est une
  cat\'{e}gorie test stricte}, Homology Homotopy Appl. \textbf{11} (2009),
  no.~2, 309--326, \url{http://projecteuclid.org/euclid.hha/1296138523}.

\bibitem[MP89]{makkai-pare}
Michael Makkai and Robert Par\'{e}, \emph{Accessible categories: the
  foundations of categorical model theory}, Contemporary Mathematics, vol. 104,
  American Mathematical Society, Providence, RI, 1989, \href
  {http://dx.doi.org/10.1090/conm/104} {\path{doi:10.1090/conm/104}},
  \url{https://doi.org/10.1090/conm/104}.

\bibitem[Ols09]{olschok:thesis}
Marc Olschok, \emph{On constructions of left determined model structures},
  Ph.D. thesis, Masaryk University, 2009.

\bibitem[Qui67]{quillen:book}
Daniel~G. Quillen, \emph{Homotopical algebra}, Lecture Notes in Mathematics,
  No. 43, Springer-Verlag, Berlin, 1967.

\bibitem[Rez20]{rezk:quasicats}
Charles Rezk, \emph{Stuff about quasicategories}, Unpublished notes, accessed
  June 23 2020, \url{https://faculty.math.illinois.edu/~rezk/quasicats.pdf}.

\end{thebibliography}

\end{document}